                    \def\version{\today}                       %

 \documentclass[reqno,11pt]{amsart}
 \usepackage{amsmath, amsthm, a4, latexsym, amssymb, mathtools, enumitem, xcolor, color, tikz-cd,todonotes}
\usepackage[unicode]{hyperref}
\hypersetup{
        colorlinks,			
        citecolor=blue,		
        urlcolor=blue,
        linkcolor=blue
}
\usepackage{srcltx}

\usepackage[english] {babel}
\usepackage{amsthm}

\def\@rmrk#1#2{\refstepcounter
    {#1}\@ifnextchar[{\@yrmrk{#1}{#2}}{\@xrmrk{#1}{#2}}}

\makeatletter\@addtoreset{equation}{section}\makeatother

 \newfont{\bfit}{cmbxti10 scaled 1200}

 \newcommand{\e}{{\rm e} }

 \newcommand{\eps}{\varepsilon}
 \newcommand{\supp}{{\rm supp}}
  \newcommand{\Aut}{{\rm Aut}}

 \newcommand{\diam}{{\rm diam}}
 
  \newcommand{\Poi}{{\rm Poi}}
 \newcommand{\Isom}{{\rm Isom}}
 \newcommand{\pv}{p_{\scriptscriptstyle{\rm conn}}}
 \newcommand{\pl}{\overline p_{\scriptscriptstyle{\rm conn}}}

 \newcommand{\R}{\mathbb{R}}
\newcommand{\N}{\mathbb{N}}
\newcommand{\Z}{\mathbb{Z}}
\newcommand{\C}{\mathbb{C}}
\newcommand{\F}{\mathbb{F}}

\newcommand{\E}{\mathbb{E}}
\renewcommand{\P}{\mathbb{P}}
\renewcommand{\H}{\mathbb{H}}

 \def\1{{\mathchoice {1\mskip-4mu\mathrm l} 
{1\mskip-4mu\mathrm l}
{1\mskip-4.5mu\mathrm l} {1\mskip-5mu\mathrm l}}}

 \newcommand{\skrib}{{\mathcal B}}
 \newcommand{\skric}{{\mathcal C}}

 \newcommand{\skrif}{{\mathcal F}}
 \newcommand{\skrig}{{\mathcal G}}
 \newcommand{\skrih}{{\mathcal H}}

 \newcommand{\skriv}{{\mathcal V}}
  \newcommand{\skriw}{{\mathcal W}}

\newcommand{\bP}{\mathbf{P}}
\newcommand{\bE}{\mathbf{E}}
\newcommand{\bzero}{\mathbf{0}}

\newcommand{\heap}[2]{\genfrac{}{}{0pt}{}{#1}{#2}}

\newcommand{\ssup}[1] {{\scriptscriptstyle{({#1}})}}

\newenvironment{example}{\refstepcounter{theorem}
{\bf Example \thetheorem\ }\nopagebreak  }%
{\nopagebreak {\hfill\rule{2mm}{2mm}}\\ }

\newtheorem{theorem}{Theorem}[section]
\newtheorem{lemma}[theorem]{Lemma}
\newtheorem{cor}[theorem]{Corollary}
\newtheorem{prop}[theorem]{Proposition}

\newtheorem{maintheorem}{Theorem}
\newtheorem{question}[theorem]{Question}
\newtheorem{problem}[theorem]{Problem}
\newcommand{\eproof}{\hfill \qed \vspace*{5mm}}

\newtheoremstyle{thm}{1.5ex}{1.5ex}{\itshape\rmfamily}{}
{\bfseries\rmfamily}{}{2ex}{}

\newtheoremstyle{rem}{1.3ex}{1.3ex}{\rmfamily}{}
{\itshape\rmfamily}{}{1.5ex}{}
\theoremstyle{rem}
\newtheorem{remark}{{\slshape\sffamily Remark}}[]

\refstepcounter{subsubsection}

\def\thebibliography#1{\section*{References}
  \list%
  {\arabic{enumi}.}
    {\settowidth\labelwidth{[#1]}\leftmargin\labelwidth
    \advance\leftmargin\labelsep
    \parsep0pt\itemsep0pt
    \usecounter{enumi}}
    \def\newblock{\hskip .11em plus .33em minus .07em}
    \sloppy                   
    \sfcode`\.=1000\relax}

\begin{document}
\title[Haagerup property and Group-Invariant Percolation]
{\large Haagerup property and Group-Invariant Percolation}
\author[Chiranjib Mukherjee and Konstantin Recke]{}
\maketitle
\thispagestyle{empty}
\vspace{-0.5cm}

\centerline{\sc By Chiranjib Mukherjee\footnote{Universit\"at M\"unster, Einsteinstrasse 62, M\"unster 48149, Germany, {\tt chiranjib.mukherjee@uni-muenster.de}}
and Konstantin Recke\footnote{Universit\"at M\"unster, Einsteinstrasse 62, M\"unster 48149, Germany {\tt konstantin.recke@uni-muenster.de}}}
\renewcommand{\thefootnote}{}
\footnote{\textit{AMS Subject
Classification:} 22D55, 60K35, 82B43, 20F16, 60D05, 60B99, 20F67.}
\footnote{\textit{Keywords:} Haagerup property, percolation, group-invariance, Kazhdhan's Property (T), spaces with measured walls, two-point functions, amenability, lamplighter groups, co-compact Fuchsian groups.}

\vspace{-0.5cm}
\centerline{\textit{Universit\"at M\"unster}}
\vspace{0.2cm}

\begin{center}
\version
\end{center}

\begin{quote}{\small {\bf Abstract: } Let $\skrig$ be the Cayley graph of a finitely generated, infinite group $\Gamma$. We show that $\Gamma$ has the Haagerup property if and only if for every $\alpha<1$, there is a $\Gamma$-invariant bond percolation $\P$ on $\skrig$ with $\E[\deg_{\omega}(g)]>\alpha\deg_{\skrig}(g)$ for every vertex $g$ and with the two-point function $\tau(g,h)=\P\big[g\leftrightarrow h\big]$ vanishing as $d(g,h)\to\infty$. As an upshot, we also obtain a sufficient condition for the Haagerup property using Bernoulli percolation and the associated threshold $p_{\mathrm{conn}}$ of random connected subgraphs of $\mathcal G$. On the other hand, we show that $\Gamma$ has Kazhdan's property (T) if and only if there exists a threshold $\alpha^*<1$ such that for every $\Gamma$-invariant bond percolation $\P$ on $\skrig$, $\E[\deg_\omega(o)]>\alpha^*\deg(o)$ implies that the two-point function is bounded away from zero. These results in particular answer questions about characterizations of properties of groups beyond amenability through group-invariant percolations,  raised by Russell Lyons ({\em J. Math. Phys.} {\bf41} 1099-1126 (2000)). 

The method of proof is new and is based on a construction of percolations with suitable dependence structures built from invariant point processes on spaces with measured walls. In fact, we develop this new approach further to obtain quantitative estimates: we show that there is an explicit relationship between probabilistic quantities like large marginals and two-function decay of percolations and geometric features captured by growth of wall distances defined by invariant actions on spaces with measured walls.

We apply this relationship to obtain quantitative bounds on the two-point functions, exhibiting in particular  exponential decay of the two-point function in several prominent examples of Haagerup groups, including co-compact Fuchsian groups, co-compact discrete subgroups of $\mathrm{Isom}(\H^n)$ and lamplighters over free groups. 

This method also allows us to extend the aforementioned characterization of property (T) to the setting of relative property (T). As further applications of our methods, we give new proofs of the facts that for Bernoulli percolation at the uniqueness threshold $p_u$ there is no unique infinite cluster for groups with property~(T) as well as with relative property (T).}
\end{quote}


\setcounter{tocdepth}{2}
\setcounter{secnumdepth}{2}

\tableofcontents


\section{\bf Introduction and Main Results}\label{introduction}

Let $\Gamma$ be a finitely generated, infinite group and let $\skrig=(V,E)$ be the right Cayley graph with respect to a finite and symmetric set of generators and equipped with the word metric $d$. In this article, we study the relationship between the geometry of this graph and the behavior of its $\Gamma$-invariant random subgraphs, which are probabilistic objects known as {\em group-invariant percolations}.

More precisely, a probability measure on subsets of $E$ is called a {\em bond percolation} on $\skrig$. This probability measure is called {\em $\Gamma$-invariant} if it is invariant under the action of $\Gamma$.
The most important example of a group-invariant percolation is {\em $p$-Bernoulli bond percolation} in which edges are kept independently at random with fixed probability $p \in [0,1]$ (we refer to Section \ref{sec background percolation} for background on group-invariant percolations). For a bond percolation $\P$, we define the {\em two-point function} $\tau(g,h):= \P\big[g\leftrightarrow h\big]$
to be the probability that $g$ and $h$ are in the same cluster, i.e.~connected component of the induced subgraph of the percolation configuration.

In a seminal paper, Benjamini, Lyons, Peres and Schramm \cite{BLPS99} initiated the study of group-invariant percolations, relating these models with the geometry of the underlying graph and also showing that these are useful tools for understanding classical models such as Bernoulli percolation. Concretely, developing and using the mass-transport principle, the following characterization of amenability of the underlying group was obtained there:

\begin{maintheorem}[Group-invariant percolation and amenability, cf.~{\cite[Theorem 1.1]{BLPS99}}]\label{theorem-AmenableGroups}
Let $\skrig$ be the Cayley graph of a finitely generated group $\Gamma$. Then $\Gamma$ is amenable if and only if for every $\alpha<1$, there exists a $\Gamma$-invariant bond percolation $\P$ on $\skrig$ with
 \begin{equation}\label{eq-BLPS}
\E[\deg_\omega(o)]>\alpha \deg(o)
\end{equation}
and with no infinite clusters. (Here, $\deg_\omega(0)$ is the degree of the identity element $o\in \Gamma$ in the induced subgraph of $\omega$ and $\deg(0)$ is the usual degree of $o$ in the Cayley graph $\mathcal G$).
\end{maintheorem}

In a beautiful survey from 2000 \cite[Sec.~10, p.~1122-1123]{L00}, 
Lyons suggested the problem of searching for characterizations of other classes of groups through group-invariant percolation as part of a program aiming to understand the interplay between geometric group theory and probability. In fact, it was noted there that previous characterizations had been limited to {\it amenability~only}. In particular, 
Lyons \cite[Sec.~10, p.~1123]{L00} also asked whether Kazhdan's property~(T) can be characterized through group-invariant percolation.

\subsection{Outline of Main Results and Methods.} The goal of this article is to answer these questions affirmatively by developing a novel and unifying approach to go beyond amenability. This approach leads to characterizations of the {\it Haagerup property} and of {\it Kazhdan's property (T)}, which unlike amenability, are not quasi-isometry invariant (see Remark \ref{remark-GeomIntHaagerup} and Remark \ref{remark-GeomIntKazhdan}). These properties are of utmost importance in geometric group theory and operator algebras; most notably, the Baum-Connes conjecture has been verified for groups with the Haagerup property \cite{HK01}, while Kazhdan's property~(T) provides an obstruction to extending this proof to the general conjecture (see Section~\ref{section-BackgroundHaagerupKazhdan} for background on these properties). 

At least as important as our characterization results is a new method of proof which we develop here: this is achieved by constructing percolations explicitly from invariant point processes on spaces with measured walls. This approach combines probabilistic and geometric intuition in a satisfactory way, is constructive and provides concrete applications and additional quantitative information along the way, as we will outline now in (A)-(D) below.

\begin{itemize}
\item[(A)] We provide the first probabilistic characterizations of the Haagerup property through invariant percolations. In particular, this property is given a genuinely probabilistic meaning.

The key new feature of this characterization is that it highlights the role of the competitive relationship between two properties a percolation may have, namely {\em large marginals} (in the sense of \eqref{eq-BLPS} for some $\alpha$ close to $1$)  and {\em two-point function decay} (a natural measurement of connectivity, cf.~\eqref{eq-vanish-tau} below). More precisely, we show that: 

\vspace{1mm}

\begin{itemize}
    \item[(1)] The Haagerup property is equivalent to the existence of invariant bond percolations with arbitrarily large marginals and two-point function vanishing at infinity (see Theorem \ref{maintheorem-Haagerup} for the precise statement).$^1$\footnote{$^1$Note that groups with the Haagerup property include in particular amenable groups. This is also reflected by the fact that, while having no infinite clusters (cf.~Theorem \ref{theorem-AmenableGroups}) clearly implies vanishing connectivity, the converse does not hold: in the case of the free group on two generators, which is the simplest example of a non-amenable group with the Haagerup property and where the Cayley graph is the four-regular tree, we have $p_c<1$ and for $p\in (p_c,1)$, there are infinitely many infinite clusters (so $p_u=1$), but the two-point function decays exponentially, see Example \ref{ex-intro1} for details.}
    
    \vspace{1mm}

As a concrete upshot of studying this competitive relationship, we also obtain a sufficient condition for the Haagerup property using {\it Bernoulli percolation}: we show that if there exists a $\Gamma$-invariant random spanning subgraph of $\skrig$ with the connectivity threshold $p_{\mathrm{conn}}=1$ a.s. (equivalently, if there exists a $\Gamma$-invariant random connected subgraph of $\skrig$ with $p_{\mathrm{conn}}=1$), then $\Gamma$ has the Haagerup property (see Theorem \ref{theorem-Spanning}).

    \item [(2)] A crucial advantage of our method of proof (see Section \ref{sec novel ideas} for an outline) is that it captures information beyond the Haagerup property because it can be made {\it quantitative}. Indeed, this is the route we will pursue  and develop an explicit relationship between probabilistic quantities, i.e.~large marginals and two-function decay of percolations, and geometric features calibrated by growth of wall distances defined by invariant actions on spaces with measured walls: concretely, we show in Theorem \ref{theorem-GeneralConstruction} that if $\Gamma$ has an invariant action on a space $(X,\mathcal W,\mathcal B,\mu)$ with measured walls (cf.~Section \ref{subsection-MeasuredWalls}), then there is a $\Gamma$-invariant bond percolation $\P$ with the two properties:     
     \begin{equation}\label{quant}
      \begin{aligned}
    &\hspace{20mm}\E[\mathrm{deg}_\omega(o)]\geq \alpha_p \mathrm{deg}(o), \ \mbox{where}\,\, \alpha_p:=\exp\big(-(1-p) \max_{o\sim h} w(x_0, hx_0)\big) \in (0,1), \,\, \forall p\in (0,1),\\ 
      &\hspace{20mm}  \exp\Big(-(1-p) \max_{o\sim h} w(x_0, hx_0)\,\, |g|\Big) \leq \tau(o,g) \leq \exp\Big(-(1-p)  w(x_0, gx_0)\Big) \quad\forall g\in\Gamma.
      \end{aligned}
    \end{equation}

    Here, $w(x,y)=\mu(\mathcal W(x,y))$ denotes the wall distance between $x\in X$ and $y\in X$ and $x_0\in X$ is a base point -- see Theorem~\ref{theorem-GeneralConstruction} for a precise statement and Corollary \ref{cor-QuantitativeDecay} for its consequence.     It may be instructive to point out that the upper bound in \eqref{quant} does not depend on the choice of Cayley graph but the marginal and more generally the lower bound depend on this choice. The above quantitative relationship between probabilistic and geometric information will be instrumental in deriving many further applications, as we will outline now.  
    
    \item [(3)] As a first application, we study prominent examples of groups with the Haagerup property, namely lamplighters over free groups $\mathbb F_r$ and co-compact Fuchsian groups. 
   
    Concretely, we show that for any Cayley graph of $\Gamma:=H\wr \mathbb F_r$, where $H$ is a finite group, as well as for any planar, non-amenable graph with one end, there exist invariant percolations with arbitrarily large marginals and exponentially decaying two-point function (see Theorem~\ref{theorem4} and Theorem~\ref{theorem5}). 
    In both results, the quantitative estimate \eqref{quant} will play a very important role. 
    These and other examples highlighted throughout this article underscore the fact that our construction is concrete and provides one framework for several percolation models arising naturally.

     \end{itemize}
    
All these results about the Haagerup property will impart new probabilistic and geometric meaning, as underlined by Remark~\ref{remark-ProbIntHaagerup} and Remark \ref{remark-GeomIntHaagerup}. 
    
   \vspace{1mm}

    \item [(B)] 
    Next, using the above approach we characterize Kazhdan's property~(T) through connectivity of invariant percolations by showing that it is equivalent 
    to the existence of a non-trivial {\em threshold} such that every invariant bond percolation with marginals above this threshold has an infinite cluster with a {\it uniformly positive} probability of reaching {\it any given vertex} (see Theorem \ref{maintheorem-Kazhdan} for the precise statement and Remark \ref{remark-ProbIntKazhdan}-\ref{remark-GeomIntKazhdan} for its probabilistic, resp.~geometric, interpretation).$^2$\footnote{$^2$We note that, by Theorem \ref{theorem-AmenableGroups}, non-amenability is equivalent to the existence of a non-trivial threshold such that every invariant bond percolation with marginals above this threshold ''just'' has some infinite cluster with positive probability. This comparison underlines the fact that Kazhdan's property (T) is a particularly strong form of non-amenability.  
    
    We may also point out that Theorem \ref{maintheorem-Haagerup} does not follow from Theorem \ref{maintheorem-Kazhdan} nor vice versa. Indeed, there are groups which have neither Haagerup's nor Kazhdan's property.
    In the present context, this is related to the fact that two-point functions may fail to vanish  at infinity in (at least) two qualitatively different ways: they could be {\it uniformly} bounded away from zero (a phenomenon known as {\it long-range order}), or they could vanish along {\it some} direction(s), but remain bounded away from zero along other directions (known as {\it connectivity decay}) -- we refer to Remark~\ref{rem-CD-LRO}, Remark~\ref{rem-Haagerup-vs-Kazhdan} and Remark~\ref{rem-InBetween} for details and concrete examples of groups exhibiting such behavior.}
    
    Regarding property (T), we obtain the following additional information and applications:

    \vspace{1mm}
    
    \begin{itemize}

    \item [(1)] As a generalization of the current method, we also extend the above result to a characterization of groups with {\em relative property (T)}: we show that for a subgroup $H$ of $\Gamma$, the pair $(\Gamma,H)$ has relative property (T) if and only if there exists a threshold such that for every $\Gamma$-invariant bond percolation with marginals above this threshold has two-point function bounded away from zero on each left $H$-coset -- see Theorem \ref{theorem-RelativeKazhdanGroups} for details.  
    
    \item [(2)] As a further application of the current method to Bernoulli percolation, we give new proofs of the facts that, for groups with property (T) as well with relative property (T), there is no unique infinite cluster at the uniqueness threshold $p_u$ (see Theorem \ref{theorem-UniquenessRelativeKazhdan}). 
    This result is similar in spirit to the study of non-uniqueness at $p_u$ relative to suitable subgroups in~\cite{HP24} where several new examples of groups with non-uniqueness at $p_u$ are obtained and the systematic treatment of this class of groups is initiated.
    \end{itemize}

\vspace{1mm}

\item[(C)] The methods developed for the proofs are new: indeed, they are necessarily different from the literature on the amenable case because the approach of removing boundaries of F{\o}lner sets is not available. In particular, we do not use the mass-transport principle, which is a powerful tool for analyzing quantities of the form $\sum_{g\in \Gamma}\tau(o,g)$, but less useful for treating pointwise behavior of $\tau(o,g)$, $g\in \Gamma$. 

The present methods are also quite different from the dynamical characterizations of Kazhdan's property (T) due to Glasner and Weiss \cite{GW97} and related works~\cite{S80,CW80} (and their analogues for the Haagerup property). 
Indeed, as remarked earlier, we will introduce a new construction of percolations built from invariant Poisson point processes on suitably chosen infinite measure spaces. This method, which we will outline in Section \ref{sec novel ideas}, is quantitative (as outlined in (A2) above), robust (as demonstrated by the subsequent applications) and unifying in the sense that it will cover the constructive part of the proof of Theorem~\ref{theorem-AmenableGroups} about amenability. 

Let us also briefly comment on the main differences between the present characterizations and the existing dynamical characterizations of Glasner and Weiss. From the point of view of percolation theory, as already pointed out by Lyons \cite[Sec.~10, p.~1123]{L00}, the latter are ``abstract''. 
 We emphasize that in the present characterizations by percolations, questions about connectivity are primary, which distinguishes them from the results in \cite{GW97} and highlights the role of the geometry of the Cayley graph for both the Haagerup property and property (T). In particular, our result provides a concrete probabilistic meaning to the Haagerup property, which, despite its enormous importance in operator algebras and group theory has not been used or studied in the field of probability theory on groups, cf.~\cite{LP16,P22}.

\vspace{1mm}

\item [(D)] Finally, the current method regarding group-invariant percolation on Cayley graphs has been used and developed further by the authors in \cite{MR24} to the setting of general 
locally finite, connected graphs and their random subgraphs, where they are used to characterize coarse embeddability into Hilbert space and the (non-equivariant) $L^1$-compression exponents through percolation despite the possible lack of symmetry. Besides highlighting the applicability of the current method, these results additionally justify the quantitative approach alluded to in (A2) above. Indeed, there the $L^1$-compression exponent is shown to be captured by explicit bounds on the {\em rate of decay rate} of two-point functions with large marginals. In particular, the aforementioned competitive and quantified relationship between probabilistic and geometric features imparts an important meaning in a novel manner. This meaning holds true in the present setting of Cayley graphs and carries over to the general, possibly even non-invariant, setting.

In this vein, let us also mention that after the first version of this article appeared, the current method has inspired a novel approach to so-called small-ball estimates for random walks on groups \cite{H24}.
\end{itemize}

In the remaining part of this introduction, we will now provide a detailed description of the results outlined above. 

\subsection{Haagerup and Kazhdan Properties and Invariant Percolations.} \label{subsection-intro1}

For a bond percolation $\P$, define the {\em two-point function} 
\begin{equation}\label{def-tau}
\tau \colon \Gamma\times\Gamma\to[0,1] \, , \, \tau(g,h):= \P\big[g\leftrightarrow h\big] \coloneqq \P\big[ \omega\colon g\, \stackrel{\omega}\longleftrightarrow \, h\big]
\end{equation}
to be the probability that $g$ and $h$ are in the same cluster of the percolation configuration. Clearly, if a $\Gamma$-invariant bond percolation has no infinite clusters, then 
 the two-point function $\tau(o,\cdot)$ vanishes at infinity, written $\tau(o,\cdot)\in C_0(\Gamma)$ and meaning that
\begin{equation}\label{eq-vanish-tau}
\lim_{r\to\infty} \sup_{\heap{g,h \in \Gamma}{d(g,h)>r}} \, \tau(g,h)  = 0.
\end{equation}
As mentioned above, recall that this property does not imply having no infinite clusters. Our first main result shows that two-point function decay in the presence of large marginals is possible if and only if the underlying group has the Haagerup property: recall that $\Gamma$ has the {\em Haagerup property} (or is {\em a-T-menable} in the terminology of Gromov \cite{G93}) if it admits an affine isometric action on a Hilbert space, which is metrically proper (see Section \ref{section-BackgroundHaagerupKazhdan} for background and Theorem~\ref{theorem-HaagerupProperty} for characterizations used in our proofs).

\begin{theorem}[{\bf Group-Invariant Percolation and the Haagerup Property}] \label{maintheorem-Haagerup}
Let $\skrig$ be the Cayley graph of a finitely generated group $\Gamma$. Then $\Gamma$ has the Haagerup property if and only if for every $\alpha<1$, there exists a $\Gamma$-invariant bond percolation $\P$ on $\skrig$ such that $\E[\deg_\omega(o)]>\alpha\deg(o)$ and $\tau(o,\cdot)\in C_0(\Gamma)$.
\end{theorem}

We point out three related results:
\begin{itemize}
    \item First, there is a {\it quantitative version} (recall \eqref{quant}) specifying the decay of the two-point function, large marginals and wall distance (see Theorem \ref{theorem-GeneralConstruction} and its consequences in Section~\ref{sec-generalized}). 
    \item Second, as mentioned in (A1) before, Theorem \ref{maintheorem-Haagerup} implies a sufficient condition for the Haagerup property in terms of Bernoulli percolation on invariant random spanning subgraphs of $\skrig$ (see Theorem~\ref{theorem-Spanning}). 
    \item Third, invariant percolations with infinite clusters and vanishing two-point functions have infinitely many infinite clusters (see Proposition~\ref{prop:ConsequenceVanishingTPF}). 
    \end{itemize}

    Let us now underline the probabilistic and geometric content of Theorem \ref{maintheorem-Haagerup}:


\begin{remark}[{\bf Percolation Interpretation of the Haagerup Property}] \label{remark-ProbIntHaagerup} Theorem \ref{maintheorem-Haagerup} provides the following probabilistic interpretation of the Haagerup property: picture the Cayley graph as an infinite, connected network and consider the task of designing an invariant, random strategy of removing edges in order to disconnect the graph. Consider the strategy successful if vertices, which are far apart, are disconnected with high probability, say tending to one as the distance tends to infinity. As a reasonable constraint, we do not allow arbitrary removal of edges, but associate a cost to the removal in such a way that we allow to remove, on average, at most $1-\alpha$ percent of the edges at each vertex for some $\alpha<1$. Theorem \ref{maintheorem-Haagerup} asserts that there exists a successful strategy at arbitrarily small costs if and only if the group associated to the Cayley graph has the Haagerup property.  \nopagebreak {\hfill\rule{2mm}{2mm}}
\end{remark}

\begin{remark}[{\bf Geometric Interpretation of the Haagerup Property and Non-QI-Invariance}] \label{remark-GeomIntHaagerup}
Theorem \ref{maintheorem-Haagerup} also entails the following concrete insight about the geometry of groups with the Haagerup property: their Cayley graphs are such that vertices at large distances can be separated efficiently, i.e.~in the probabilistic sense of Remark \ref{remark-ProbIntHaagerup}. We emphasize a distinction between the intuitive sense of the word geometric and the sense in which it is frequently used in geometric group theory, where a property is called {\em geometric} if it is invariant under quasi-isometries ({\em QI-invariant}), see for instance \cite[Section~3.6]{BdlHV08}. The Haagerup property is not QI-invariant as shown by Carette \cite{C14}. From this perspective, Theorem~\ref{maintheorem-Haagerup} and the underlying geometric intuition described above may be surprising. 
There is no contradiction between the geometric nature of our characterization and the QI-non-invariance of the Haagerup property (and the same remark will apply to property~(T)) because given a quasi-isometry from one group to another, the push-forward of an invariant percolation need not be an invariant percolation w.r.t.~the other group action. In particular, for quasi-isometric groups, there need not exist an equivariant quasi-isometry which transports high marginal percolations with vanishing two-point function on one group to such percolations on the other.
Informally speaking, relaxing the static setting of the Cayley graph itself to the setting of invariant random subgraphs provides the appropriate amount of freedom to being able to capture the Haagerup property in terms of geometric properties of the graph. \nopagebreak {\hfill\rule{2mm}{2mm}}
\end{remark}

We now proceed to an important, {\em strong negation} of the Haagerup property: recall that $\Gamma$ has {\it property (T)} if there exist $\eps>0$ and a finite subset $K\subset\Gamma$ such that every unitary representation $\pi\colon\Gamma\to B(H)$ for which there exists a unit vector $\varphi$ with
\begin{equation} \label{def-Kazhdan}
    \sup_{s \in K} \, \Vert \pi(s)\varphi - \varphi \Vert < \eps,
\end{equation}
has a non-zero invariant vector (see \cite{BdlHV08} for background and \cite[Section 1.1.2(1)\&(2)]{CCJJV01} for characterizations used in our proofs). While not immediate from this definition, property (T) is a strong negation of both amenability and the Haagerup property. This will be underscored by the following probabilistic characterization of countable Kazhdan groups (every countable Kazhdan group is finitely generated, see e.g.~\cite[Theorem 1.3.1]{BdlHV08}), which is our second main result (see Theorem \ref{theorem-KazhdanGroups}).

\begin{theorem}[{\bf Group-Invariant Percolation and Kazhdan's Property (T)}] \label{maintheorem-Kazhdan}
Let $\skrig$ be the Cayley graph of a finitely generated group $\Gamma$. Then $\Gamma$ has property (T) if and only if there exists $\alpha^*<1$ such that for every $\Gamma$-invariant bond percolation $\P$ on $\skrig$, $\E[\deg_\omega(o)]>\alpha^*\deg(o)$ implies that $\tau(o,\cdot)$ is bounded away from zero, i.e.\
\begin{equation} \label{equ-LongRangeOrder}
\inf_{g,h\in\Gamma} \tau(g,h)>0.
\end{equation}
\end{theorem}

Theorem \ref{maintheorem-Kazhdan} answers affirmatively the question whether property (T) can be characterized through group-invariant percolations raised by Lyons in \cite[p.~1123]{L00}. It is important to point out two aspects. First, the direct consequence (i.e., necessity) of Kazhdan's property was known and is due to Lyons and Schramm \cite{LS99}. Second, the threshold $\alpha^*$ is not elusive and quantitative estimates in terms of Kazhdan pairs are available (see Section \ref{subsection-Kazhdan}, in particular Theorem~\ref{theorem-KazhdanGroups} and Remark \ref{remark-ThresholdKazhdan}, for precise statements). The proof of the converse direction (i.e., sufficiency) is new with the argument relying heavily on the method developed here.

As in the case of the Haagerup property, the next two remarks put Theorem \ref{maintheorem-Kazhdan} into the respective context.

\begin{remark}[{\bf Percolation Interpretation of Kazhdan's Property (T)}] \label{remark-ProbIntKazhdan}
Theorem \ref{maintheorem-Kazhdan} provides the following probabilistic interpretation of property (T): picture the Cayley graph again as an infinite, connected network and consider the task of designing an invariant, random strategy to disconnect the graph with the constraint that, on average, at most $1-\alpha$ percent of the edges may be removed at each vertex. Theorem \ref{maintheorem-Kazhdan} asserts that every strategy which is {\em too cheap} (i.e.~costs less than $1-\alpha^*$) is bound to fail: for such a strategy, pairs of vertices will remain connected with uniformly positive probability, irrespective of their distance. To further interpret this condition, we recall a result of Lyons and Schramm \cite[Theorem 4.1]{LS99} that an invariant, ergodic and insertion-tolerant bond percolation satisfying \eqref{equ-LongRangeOrder} always has a {\em unique} infinite cluster. Thus Theorem \ref{maintheorem-Kazhdan} shows that, on a Kazhdan group, {\em every} such percolation with sufficiently large marginals has a unique infinite cluster, suggesting that Cayley graphs of Kazhdan groups are Cayley graphs with a strong tendency towards forming a unique infinite cluster. This intuition is complemented by a result obtained by Hutchcroft and Pete~\cite{HP20} in a recent breakthrough regarding the ergodic-theoretic cost of Kazhdan groups, which states that Cayley graphs of Kazhdan groups admit invariant percolations with arbitrarily small marginals and with a unique infinite cluster. To summarize, Theorem \ref{maintheorem-Kazhdan} asserts that Cayley graphs of Kazhdan groups are distinguished by the property of being very robust against all attempts of disconnecting them. Interestingly, this robustness property is reminiscent of similar features of expander graphs in the setting of finite graphs. \nopagebreak {\hfill\rule{2mm}{2mm}}
\end{remark}

\begin{remark}[{\bf Geometric Interpretation of Kazhdan's Property (T) and Non-QI-Invariance}] \label{remark-GeomIntKazhdan}
Theorem \ref{maintheorem-Kazhdan} also entails concrete insight about the geometry of Kazhdan groups: their Cayley graphs are {\em highly connected} at all distances and they are thus able to withstand any cheap attempt of disconnecting them (in the probabilistic sense of Remark \ref{remark-ProbIntKazhdan}). We also note that, like the Haagerup property, property (T) is not QI-invariant, see e.g.~\cite[Theorem 3.6.5]{BdlHV08}. This fact may lead to the perception that property~(T) is not reflected in the geometric structure of the Cayley graph in a meaningful or intuitive way -- from this perspective, Theorem \ref{maintheorem-Kazhdan} and the underlying geometric intuition described above may be surprising (to elaborate on that point, property (T) is by definition a property of the complete unitary representation theory of the group, whose encoding in the Cayley graph is unclear and seems to be quite challenging to decipher, if at all possible). Our approach of adding a natural source of randomness thus seems to be very useful to identify rigorously features of the geometry of Kazhdan groups. \nopagebreak {\hfill\rule{2mm}{2mm}} 
\end{remark}

\noindent{\bf Further Results and Applications Related to Property (T).} As mentioned before, in this context we furthermore show the following: 
\begin{enumerate}
\item We consider groups with relative property~(T) (see Section \ref{subsection-RelativeKazhdan} for the definition) and extend Theorem~\ref{maintheorem-Kazhdan} to this setting (see Theorem~\ref{theorem-RelativeKazhdanGroups}) 
\item We provide applications of the corresponding quantitative threshold to a well-known problem in Bernoulli percolation: namely, we give new proofs of the facts that the uniqueness threshold $p_u$ does not belong to the uniqueness phase in Bernoulli bond percolation on any Cayley graph of an infinite group which has property~(T) and also for any group which has property (T) relative to an infinite, normal subgroup (see Theorem \ref{theorem-UniquenessRelativeKazhdan}). For groups with property (T), this result was shown in \cite{LS99} (using a different argument) and for groups with relative property (T) a more general result has appeared in \cite{GT16}, but the argument we give here is very different: our argument is probabilistic and based on the aforementioned threshold, an approach to ruling out a unique infinite cluster at specific values of the survival parameter due to Benjamini, Lyons, Peres and Schramm \cite{BLPS99b} and ideas described by Hutchcroft and Pete in \cite{HP20}.

\item We also provide an estimate of a similar threshold appearing for groups which do not have the Haagerup property, also known as the {\it weak} Kazhdan property (see Proposition \ref{prop-WeakKazhdanThreshold}). Under an additional symmetry assumption, our probabilistic argument for non-uniqueness at the uniqueness threshold also applies to non-Haagerup groups (see Corollary \ref{cor-PuWeakKazhdan}). 
It may be instructive to point out that non-Haagerup groups in general do not have non-uniqueness at~$p_u$. For instance~${\rm SL_3(\mathbb Z)}*\mathbb Z_2$ has $p_u=1$ for every Cayley graph because it has infinitely many ends but has relative property~(T) w.r.t.~the subgroup ${\rm SL}_3(\mathbb Z)$ (which is not normal).
\end{enumerate}

Let us conclude by pointing out that although we state our results for bond percolations, the analogues for {\em site percolations} 
-- i.e.~random subgraphs induced by an invariant random subset of vertices --
are straightforward (for the Haagerup property, see Theorem \ref{cor-SitePercolation}). 

\subsection{Decay Rate of Two-Point Functions for Concrete Examples of Haagerup Groups.} \label{subsection-intro5}

The construction we develop is widely applicable and yields natural models as well as exponential decay in settings of particular interest. This is emphasized in the final part of the article by an in-depth discussion of three classes of examples which have historically been of great relevance for both probability theory and geometric group theory.

\subsubsection{\bf Lamplighters over free groups.} \label{subsection-lamplighters} The first class of examples we treat are so-called {\em lamplighter groups} over free groups. These are restricted wreath products of the form $H\wr\F_r$, where the {\em base group} $\F_r$ is a free group of rank $r\in\N$ and the {\em lamp group} $H$ is finite (see Section \ref{lamplighters} for precise definitions). It is a non-trivial fact \cite{CSV08} that these groups have the Haagerup property. More generally, stability of the Haagerup property under wreath products was established in \cite{CSV12}. Lamplighter groups, especially ones of the form $\Z_2\wr\Z^d$, have received substantial attention in the theory of random walks starting with \cite{KV83}. Let us mention the breakthrough results about identification of the Poisson boundary \cite{E10,E11} for $d\geq5$ and \cite{LP21} for $d=3,4$. See also \cite{KW07} for an investigation of the Poisson boundary of random walks on lamplighters over free groups. For Bernoulli percolation, Lyons and Schramm \cite{LS99} established that any Cayley graph of a lamplighter group over a finitely generated base group with non-trivial, finite lamp group has non-trivial uniqueness phase, i.e.\ $p_u<1$. In particular, this implies that the two-point function of Bernoulli percolation is bounded away from zero for sufficiently large marginals. On the other hand, in the setting of invariant percolations, we prove the following result (see Theorem \ref{theorem-Lamplighter}).

\begin{theorem}[{\bf Large Marginals and Exponential Decay}] \label{theorem4} 
Let $H$ be a finite group, let $r\in\N$ and let $\skrig$ be any Cayley graph of $\Gamma=H\wr\F_r$. Then there exists a constant $C>0$ such that for every $\alpha<1$, there exists a $\Gamma$-invariant bond percolation $\P$ on $\skrig$ with $\E[\deg_\omega(o)] > \alpha \deg(o)$ and with two-point functions having exponential decay, i.e.\
\begin{equation}\label{tau-decay}
\e^{-\beta |g|} \leq \tau(o,g) \leq \e^{-\gamma |g|} \quad \mbox{for all} \, \, g \in \Gamma,
\end{equation}
where $\beta,\gamma>0$ such that $\beta / \gamma \leq C$.
\end{theorem}

\begin{remark}[{\bf On the Geometry of Lamplighters}] \label{rem-GeomLamplighters}
    A distinguishing feature of Theorem \ref{theorem4} is given by the fact that the geometry of the wreath product $H\wr\F_r$ is considerably different from the geometry of the base free group: any Cayley graph of $H\wr\F_r$ has one end and even $p_u<1$ \cite{LS99} (in contrast to $\F_r$). To elaborate on that point, without using our construction (see Theorem \ref{theorem-GeneralConstruction}), even finding an invariant percolation model with large marginals and two-point function vanishing at infinity (not necessarily exponentially) seems to be challenging for these graphs.
\end{remark}

\subsubsection{\bf Co-Compact Fuchsian Groups} \label{subsection-CocompactFuchsian} The second class of examples we treat are {\em planar} Cayley graphs. We are particularly interested in one-ended, nonamenable examples, i.e.\ Cayley graphs which embed in a nice way into the hyperbolic plane $\H^2$, see~\cite{BS01}. Notably, these include the canonical Cayley graphs of co-compact Fuchsian groups, an example considered in percolation theory first in~\cite{L98}. For important results about percolation in the planar hyperbolic setting, Bernoulli as well as general invariant, we refer to \cite{BS01,L98,L01,LP16}. In this context, let us mention that several fundamental conjectures about Bernoulli percolation were recently established for Cayley graphs of Gromov hyperbolic groups in~\cite{H19}.

In the planar hyperbolic setting and more generally for groups with a proper, co-compact action on $n$-dimensional real hyperbolic space, $n\ge1$, our construction yields a natural percolation model defined by removing edges along boundaries of half-spaces of the natural Poisson process (see Section~\ref{hyperbolic} for details). These also exhibit exponential decay.

\begin{theorem}[\bf Large Marginals and Exponential Decay] \label{theorem5} 
Let $\skrig$ be a transitive, nonamenable, planar graph with one end, resp.~let $\skrig$ be the Cayley graph of a discrete co-compact subgroup $\Gamma$ of ${\rm Isom}(\mathbb H^n)$. Then for every $\alpha<1$, there exists an $\Aut(\skrig)$-invariant, resp.~$\Gamma$-invariant, bond percolation $\P$ on $\skrig$ with $\E[\deg_\omega(o)] > \alpha \deg(o)$ and such that $\tau(o,\cdot)$ has exponential decay (i.e., \eqref{tau-decay} holds).
\end{theorem}

\subsubsection{\bf The Special Case of Amenable Groups.}\label{subsection-intro4}

Recall that amenable groups have the Haagerup property~\cite{BCV95}. 
Our construction is unifying in the sense that it covers the constructive part of the proof of Theorem \ref{theorem-AmenableGroups} from \cite{BLPS99} in the case of an amenable group. 
More precisely, we obtain the following result (see Proposition \ref{prop-SpecialCaseAmenability}).

\begin{prop}[{\bf Recovering the Amenable Case}]\label{prop1}
Let $\skrig$ be the Cayley graph of a finitely generated amenable group $\Gamma$. There exists an invariant action of $\Gamma$ on a space with measured walls such that the construction used in the proof of Theorem \ref{maintheorem-Haagerup} yields, for every $\alpha<1$, a $\Gamma$-invariant bond percolation $\P$ on $\skrig$ with $\E[\deg_\omega(o)]>\alpha\deg(o)$ and with no infinite clusters.
\end{prop}

We note that information about the decay rate of the two-point function does not suffice to characterize amenability (indeed, supercritical Bernoulli bond percolation on a regular tree of degree at least three shows that even exponential decay of the two-point function does not imply finiteness of percolation clusters, see Example~\ref{ex-intro1}). However, also in the amenable setting a close connection with positive definite functions, which are not two-point functions, exists. Essentially the idea is that for any percolation with only finite clusters, there is a canonical way to associate a random positive definite function to the cluster of the identity. This approach and a characterization of amenable groups via such positive definite functions were first described in \cite[Theorem 1.6]{MR22}, which served as the motivation of the current work. These positive definite functions provide a large class of probabilistic examples of functions arising naturally in the study of Schur multipliers of $C^*$-algebras associated to~$\Gamma$, which implement certain operator algebraic approximation properties. Their specific form was then used to provide asymptotic decompositions and establish limiting behavior for sequences of such multipliers by means of the group-invariant compactification constructed in \cite{MV16}; we refer the interested reader to \cite{MR22} for details.

\subsection{Novel Ideas of the Proofs.}\label{sec novel ideas}

Let us underline the key ideas involved in the proofs and the main technical contributions of the present work. To start off, we point out that while the Haagerup property has played a major role in operator algebras and geometric group theory, it does not seem to have appeared in concrete probabilistic contexts. In particular, to the best of the authors' knowledge, there are no applications of the Haagerup property to percolation,~cf.~\cite{LP16}. Two observations may explain this:
\begin{itemize}
    \item[(i)] the Haagerup property is a general property, enjoyed by many groups, whose ``common geometric features'' are not at all apparent (cf.~Section~\ref{section-BackgroundHaagerupKazhdan} below).
    \item[(ii)] the Haagerup property has several equivalent reformulations (see e.g.~\cite{CCJJV01}), most of which are  quite abstract and may thus seem unrelated to intuitive probabilistic behavior.
\end{itemize}
To elaborate on the second point, while the proof of any characterization of the Haagerup property may use the reformulation most useful for each direction of the equivalence, the first difficulty clearly lies in formulating a suitable statement. Now the first main contribution of the present work is not only to provide such a statement, but one which fits into a natural probabilistic framework and leads to several ramifications.

More specifically, the properties, which percolations on groups with the Haagerup property should or should not have, are far from clear. However, one possibility to arrive at the statement of Theorem~\ref{maintheorem-Haagerup} is to observe that the ``if'' direction is a consequence of a classical observation by Aizenman and Newman \cite{AN84} that the two-point function of an invariant percolation is positive definite, combined with a particular formulation of the Haagerup property in terms of positive definite functions. Note that there is a priori no reason for the converse to be true and that there would be two general approaches towards proving it, either through abstract reasoning, or by an explicit construction. In this context, recall that a constructive approach was used for amenability \cite{BLPS99} utilizing the existence of F{\o}lner sets and the mass-transport principle, but in the present setting, F{\o}lner sets are not available and the mass-transport principle would be useful for quantities of the from $\sum_{g\in \Gamma}\tau(o,g)$ instead of the desired pointwise behavior of $\tau(o,g)$, $g\in \Gamma$. Nonetheless, a constructive approach is more desirable from a probabilistic point of view. In fact, we are interested not in ``some" construction, but one in which the probabilistic and geometric mechanisms converge, and the ways in which different formulations of the Haagerup property are linked, become evident. While these requirements are somewhat at odds with the need for a ``flexible" construction (cf.~(i) above), there is an important reason for asking for these additional difficulties to be addressed, namely that such an approach is likely to find further applications.

Let us now describe the explicit construction put forward in this article. While a construction with built-in independence is desirable for its tractability, Bernoulli percolation, except for very specific examples, will not have the properties in Theorem \ref{maintheorem-Haagerup}. In general, we therefore have to introduce (strong) dependencies between edges. It is nonetheless instructive to look at the basic example where Bernoulli percolation suffices, which is that of a free group. Essentially, the reason independent percolation suffices here is that every edge plays the role of a ``wall", whose removal breaks the graph into two disconnected components. This simple observation leads to a proof of Theorem~\ref{maintheorem-Haagerup} in the special case of groups acting properly on trees (cf.\ Proposition~\ref{prop-GroupsActingOnTrees}), an important subclass of groups with the Haagerup property. The next important observation is that it is not the tree-structure, but the ''wall-structure", which makes this argument work. This is demonstrated by our proof of Theorem~\ref{maintheorem-Haagerup} for groups acting properly on {\em spaces with walls}~\cite{HP98} (cf.\ Proposition~\ref{prop-GroupsActingOnSpacesWithWalls}). Here a countable set of ''walls'' is deleted independently, and the resulting random set of walls is used to define percolations with more intricate dependence structure. In both of these constructions, the probabilistic and geometric mechanisms at play are clearly visible. To prove Theorem~\ref{maintheorem-Haagerup}, we then develop a machinery which yields similar behavior for all groups with the Haagerup property: {\color{blue}we use} infinite measure spaces called {\em spaces with measured walls} \cite{CMV04}, which are quadruplets $(X,\skriw,\skrib,\mu)$ where $X$ is a set, $\skriw$ is a set of partitions of $X$ into two classes (a partition is called a {\em wall}), $\skrib$ is a $\sigma$-algebra on $\skriw$ and $\mu$ is a measure on $\skrib$ such that for every two distinct points $x$ and $y$ in $X$, the set $\skriw(x,y)$ of walls separating $x$ and $y$ is $\skrib$-measurable and $\mu(\skriw(x,y))<\infty$. 

We will be interested mainly in {\em invariant actions} on spaces with measured walls, which are actions on the set $X$ which extend to a measure-preserving action on $(\skriw,\skrib,\mu)$ (see Section \ref{subsection-MeasuredWalls} for details). Now, in the continuum setup, the next key new idea is to employ invariant (Poisson) point processes on spaces with measured walls, i.e.~the continuous analogue of deleting walls independently. The resulting random countable set of walls is then used to build invariant bond percolations with suitable dependence structures (see Theorem \ref{theorem-GroupsWithTheHaagerupProperty}) satisfying the required properties of large marginals and vanishing connectivity. 

This construction is the second main contribution of the present work. In fact, this approach provides a generalized machinery to build invariant percolations, with explicit control of marginals and two-point function decay, from invariant actions on spaces with measured walls and growth of wall distance -- recall \eqref{quant} and see Theorem \ref{theorem-GeneralConstruction} and Corollary \ref{cor-QuantitativeDecay} for details. In this general form, the construction addresses all the aforementioned difficulties and, indeed, can be successfully applied in several further situations, as briefly mentioned before and as we will describe concretely below.

First, regarding the characterization of property (T) given in Theorem \ref{maintheorem-Kazhdan}, the situation is reversed compared to that of the Haagerup property. Here, the implication of Kazhdan's property~(T) follows from the aforementioned observation that two-point functions are positive definite (this is the direction which was previously known \cite{LS99} -- here a different, quantitative proof is given based on~\cite{IKT09}). Note that property (T), which is essentially designed to be a very strong version of non-amenability, has strong {\it consequences} for probabilistic models, including percolation. In fact, it is in part due to this reason that property~(T) has played a notable role in percolation, cf.~\cite{LP16}. Now the fact that the resulting condition for percolation actually {\it suffices to characterize property (T)} is the new contribution and the tool for deducing this new direction is also our construction of building percolations using spaces with measured walls and the associated quantitative information \eqref{quant} (see Section \ref{subsection-Kazhdan} for details).

Let us note that the current method is quite robust: 
it is more general as shown by the aforementioned characterization of relative property (T) (Theorem \ref{theorem-RelativeKazhdanGroups}), the subsequent application of the method to Bernoulli percolation at the uniqueness threshold for such groups  (Theorem~\ref{theorem-UniquenessRelativeKazhdan}) and the fact that we also recover the amenable case (see Proposition \ref{prop-SpecialCaseAmenability}). The key ingredient in the proofs of Theorem \ref{theorem4} and Theorem \ref{theorem5} is the quantitative approach. For these applications, we use the existence of measure definite functions with linear growth.
Existence of such functions is known in the second case, and has been proved in the first case in \cite{CSV12} (we present a simplified, self-contained argument here).

\subsection{Background on the Haagerup and Kazhdan Properties.} \label{section-BackgroundHaagerupKazhdan}

A locally compact, second countable group $G$ has the {\em Haagerup property}, or is {\em a-T-menable} in the terminology of Gromov \cite{G93}, if it admits a continuous, affine, isometric action on a Hilbert space~$H$, which is metrically proper (i.e.~for every bounded $B\subset H$, the set $\{g\in G\colon gB\cap B\neq \emptyset\}\subset G$ has compact closure), see e.g.~\cite{CCJJV01}. Similarly, the group $G$ has {\em Kazhdan's property (T)} if every continuous, affine, isometric action on a Hilbert space has a fixed point (for finitely generated groups, this definition is equivalent to the one given in \eqref{def-Kazhdan}, see e.g.~\cite{BdlHV08,CCJJV01}). There are alternative characterizations of both properties in terms of positive and negative definite functions, which are classical and which we review in Section \ref{sec-prel-Haagerup}, resp.~Section \ref{subsection-Kazhdan}.

The origins of the Haagerup property can be traced back to seminal work of Haagerup \cite{H79} on approximation properties of group $C^*$-algebras of free groups. Groups with the Haagerup property form a large class with classical examples provided by compact groups, amenable groups, free groups, Coxeter groups, groups acting properly on trees and isometry groups of real hyperbolic spaces, see e.g.~\cite{CCJJV01}. We refer the reader to \cite{CCJJV01} for more examples, background as well as a survey of various results about the Haagerup property. On the other hand, groups with property (T) form a quite special class of groups, the main examples being ${\rm SL}_n(\R)$ (and ${\rm SL}_n(\Z)$) for $n\geq3$. For more examples, background and a survey of various striking results about property (T), see \cite{BdlHV08}. Two measure-theoretic characterizations of the Haagerup and Kazhdan properties were previously known. On the one hand, Jolissaint \cite[Theorem 2.1.3]{CCJJV01} showed that a locally compact, second countable group has the Haagerup property if and only if it has a measure preserving action on a standard probability space which is strongly mixing and contains a F{\o}lner sequence. 
This result fits into the same circle of ideas as the characterization of property (T) due to Schmidt \cite{S80} and Connes and Weiss~\cite{CW80}, which states that $G$ has property (T) if and only if every measure preserving ergodic action is strongly ergodic; see \cite[Section~2.1]{CCJJV01} for more details.
On the other hand, answering a question posed in \cite{CCJJV01}, Glasner \cite[Theorem 13.21]{G03} showed that a countable group has the Haagerup property if and only if the closure of the set of $\Gamma$-invariant, mixing probability measures on $\{0,1\}^\Gamma$ is convex. This characterization is an analogue of the well-known probabilistic characterization of property~(T) due to Glasner and Weiss~\cite{GW97}, which states that a countable group $\Gamma$ has property (T) if and only if in the set of $\Gamma$-invariant probability measures on $\{0,1\}^\Gamma$, the set of ergodic measures is closed, equivalently, is not dense. In comparison with these existing results, perhaps distinguishing features of Theorem~\ref{maintheorem-Haagerup} and Theorem~\ref{maintheorem-Kazhdan} are provided by the probabilistic meaning entailed and the concrete insight provided about the geometry of the group, see Remarks \ref{remark-ProbIntHaagerup}-\ref{remark-GeomIntKazhdan}.  On a similar note, recall that the Haagerup property is a generalization of amenability \cite{BCV95,CCJJV01}, while property (T) is a strong negation -- both of these facts are illustrated in a new and geometric way by Theorem \ref{maintheorem-Haagerup}, resp.~Theorem \ref{maintheorem-Kazhdan}, in combination with Theorem \ref{theorem-AmenableGroups}.

\subsection{Organization of the Rest of the Article.} \label{subsection-intro6}

The rest of the article is organized as follows. Section \ref{preliminaries} provides necessary background about the Haagerup property, invariant percolations and point processes. Section~\ref{haagerup} deals with the results about the Haagerup property announced in Section~\ref{subsection-intro1}. In particular, it contains the proof of Theorem \ref{maintheorem-Haagerup} as well as the proof of our general construction and the associated quantitative estimates. The content of Section~\ref{kazhdan} are precise statements and proofs of the results about Kazhdan's property (T), which were alluded to in Section~\ref{subsection-intro1}. More specifically, this section is split into three subsections: the proof of Theorem~\ref{maintheorem-Kazhdan} and estimates of the appearing threshold are the content of Section~\ref{subsection-Kazhdan}. The extension of Theorem~\ref{maintheorem-Kazhdan} to groups with relative property (T) and the aforementioned application to Bernoulli percolation are shown in Section~\ref{subsection-RelativeKazhdan}. Finally, in Section~\ref{subsection-WeakKazhdan}, we present the quantitative converse to Theorem~\ref{maintheorem-Haagerup}. In Section~\ref{amenable}, we discuss spaces with measured walls naturally associated to amenable groups and provide a proof of Proposition~\ref{prop1}. Section~\ref{lamplighters} deals with the example of lamplighter groups and contains the proof of Theorem~\ref{theorem4}. The description of our construction for certain hyperbolic groups and the proof of Theorem~\ref{theorem5} are carried out in Section~\ref{hyperbolic}.  Section \ref{section-outlook} contains a brief outlook.

\section{\bf Preliminaries}\label{preliminaries}
In this section provide some relevant facts pertinent to both geometric group theory and probability.

\subsection{The Haagerup Property, Positive Definite and Conditionally Negative Definite Functions.}\label{sec-prel-Haagerup}

We start with some background on the Haagerup property, which, by definition, as mentioned in the introduction, is satisfied by 
a locally compact, second countable group $G$ if it admits a continuous, affine, isometric action on a Hilbert space, which is metrically proper, see \cite{CCJJV01}.
In the current section, we collect the most relevant characterizations of this property for finitely generated 
groups $\Gamma$. We first recall the following standard definitions. 

\begin{itemize}
\item {\it Positive definiteness.} A function $k \colon \Gamma \times \Gamma \to \C$ is called a {\em positive definite kernel} if the matrix $[k(g,h)]_{g,h\in F}$ is positive for any finite subset $F\subset\Gamma$, i.e.\
$$
\sum_{i,j=1}^n \overline{a_i}a_j k(g_i,g_j) \geq 0
$$
for every $g_1,\ldots,g_n\in\Gamma$ and $a_1,\ldots,a_n\in\C$. It is called {\em invariant} if it is invariant under the diagonal action of $\Gamma$. A function $\varphi \colon \Gamma \to \C$ is called {\em positive definite} if $k_\varphi(g,h)\coloneqq\varphi(g^{-1}h)$ is a positive definite kernel. It is called {\em normalized} if $\varphi(e)=1$.

\medskip

\item {\it Conditional negative definiteness.} A function $k \colon \Gamma \times \Gamma \to [0,\infty)$ is called a {\em conditionally negative definite kernel} if $k(g,g)=0$ and $k(g,h)=k(h,g)$ for every $g,h\in\Gamma$ and 
$$
\sum_{i,j=1}^n a_i a_j k(g_i,g_j) \leq 0
$$
for every $g_1,\ldots,g_n\in\Gamma$ and $a_1,\ldots,a_n\in\R$ with $\sum_{i=1}^n a_i = 0$. It is called {\em invariant} if it is invariant under the diagonal action of $\Gamma$. A function $\psi \colon \Gamma \to [0,\infty)$ is called {\em conditionally negative definite} if $k_\psi(g,h)\coloneqq\psi(g^{-1}h)$ is a conditionally negative definite kernel.

\medskip

\item {\it Measure definiteness.} The notion of measure definiteness will be needed in Section \ref{sec-generalized}. 
Following \cite[Definition 0.1.1]{RS98}, we say that a function $k\colon\Gamma\times\Gamma\to[0,\infty)$ is a {\em measure definite kernel} if there exists a measure space $(\Omega,\skrib,\mu)$ and a map $S\colon\Gamma\to\skrib$, $g\mapsto S_g$ such that 
$$
k(g,h)=\mu(S_g \Delta S_h)
$$ 
for every $g,h\in\Gamma$. Similarly, a function $\psi\colon\Gamma\to[0,\infty)$ is called {\em measure definite} if the kernel $k_\psi(g,h)=\psi(g^{-1}h)$ is measure definite.
\end{itemize}
According to a well-known result of Schoenberg, a function $\psi: \Gamma \to \C$ is conditionally negative definite if and only if for every $t\geq 0$, the map $g\mapsto \exp(- t \psi(g))$ is positive definite. 
The origins of the Haagerup property can be traced back to the seminal work of Haagerup \cite{H79} on group $C^*$-algebras of free groups, where, as a key ingredient, it was shown that the word-length function 
$|\cdot|$ on the finitely generated free group $\F_r$ (w.r.t.\ the standard set of generators)  is conditionally negative definite. Therefore, by Schoenberg's correspondence, Haagerup's result can be restated by saying that, for each $t\geq 0$, the map $\F_r \ni g \mapsto \exp(- t|g|)$ is positive definite. Note that if 
$$
\varphi_j(g):=\exp\Big(- \frac 1j |g|\Big) \qquad j \geq 1, g \in \Gamma, 
$$
 then each $\varphi_j$ vanishes at infinity and the sequence converges pointwise to $1$. 
In this vein, below we summarize the characterizations of the Haagerup property in terms of positive and conditionally negative definite functions that are relevant for our purposes: 
\begin{theorem}[Characterizations of the Haagerup property, cf.~{\cite[Theorem 2.1.1]{CCJJV01}}] \label{theorem-HaagerupProperty}
Let $\Gamma$ be a finitely generated group. Then the following are equivalent:
\begin{enumerate}
\item[{\rm (i)}] $\Gamma$ has the Haagerup property;
\item[{\rm (ii)}] there exists a conditionally negative definite function $\psi$ on $\Gamma$ that is proper (the latter condition means that $\lim_{|g|\to\infty}\psi(g)=\infty$). 
\item[{\rm (iii)}] there exists a sequence of normalized, positive definite functions on $\Gamma$ which vanish at infinity and converge pointwise to $1$.
\end{enumerate}
\end{theorem}

There are further characterizations of the Haagerup property -- for example, it is also known that $\Gamma$ has the Haagerup property if and only if it admits a mixing (or $C_0$) representation containing almost invariant vectors \cite{CCJJV01}. Since $\Gamma$ is amenable if and only if its left-regular representation admits almost invariant vectors, it follows that amenable (e.g.\ abelian or compact) groups satisfy the Haagerup property.

\subsection{Group-Invariant Percolations.}\label{sec background percolation}

We first introduce some graph-theoretic terminology. Let $\skrig=(V(\skrig),E(\skrig))$ be a graph with vertex set $V(\skrig)$ and symmetric edge set $E(\skrig)\subset V(\skrig)\times V(\skrig)$. We denote edges by $[u,v]$. Two vertices $u$ and $v$ are called {\em adjacent} or {\em neighbors}, written $u\sim v$, if $[u,v]$ is an edge. The {\em degree} $\deg(v)=\deg_{\skrig}(v)$ of a vertex $v\in V(\skrig)$ is the number of vertices which are adjacent to it. The {\em graph distance} between two vertices $u,v\in V(\skrig)$ is defined to be the length of a shortest path in $\skrig$ which connects $u$ and $v$ and is denoted by $d(u,v)=d_{\skrig}(u,v)$. A {\em tree} is a connected graph without cycles and a {\em forest} is a graph all of whose connected components are trees. The {\em number of ends} of $\skrig$ is defined to be the supremum of the number of infinite components of $\skrig\setminus K$ over all finite subsets $K$ of $\skrig$; in particular $\skrig$ has {\em one end} if removing any finite set from $\skrig$ leaves exactly one infinite connected component.

An {\em automorphism} of $\skrig$ is a bijection of the set vertex set which preserves adjacency and $\Aut(\skrig)$ denotes the group of all automorphisms of $\skrig$ equipped with the topology of pointwise convergence. A subgroup $G\subset\Aut(\skrig)$ is {\em transitive} if $|V(\skrig)/G|=1$. It is {\em quasi-transitive} if $|V(\skrig)/G|<\infty$. The graph $\skrig$ is called {\em transitive}, resp.\ {\em quasi-transitive}, if $\Aut(\skrig)$ is transitive, resp.\ quasi-transitive. A quasi-transitive graph is called {\em amenable} if $\inf |\partial K|/|K|=0$, where $K$ runs over finite non-empty subsets of $V(\skrig)$ and $\partial K$ denotes {\em edge boundary} of $K$, that is the set of all edges with exactly one endpoint in $K$. If a quasi-transitive graph satisfies $\inf |\partial K|/|K|>0$ it is called {\em non-amenable}.

Given the Cayley graph $\skrig=(\Gamma,E)$ of a finitely generated group $\Gamma$, the action defined by left multiplication identifies $\Gamma$ with a transitive subgroup of $\Aut(\skrig)$. Recall from Section \ref{introduction} that a $\Gamma$-invariant probability measure on subsets of $E$ is called a {\em $\Gamma$-invariant bond percolation}. The percolation configurations are denoted by $\omega$ and we will often identify $\omega$ with the corresponding induced subgraph. A {\em cluster} of $\omega$ is a connected component of the induced subgraph and we write $C(g)$ for the cluster containing $g\in\Gamma$. We write $\deg_\omega(g)$ for the degree of $g\in\Gamma$ as a vertex in the configuration $\omega$. For $g, h \in \Gamma$, the {\em two-point function} 
$$\tau(g,h):= \P\big[ g\leftrightarrow h\big]
$$ 
is defined to be the probability that $g$ and $h$ are in the same cluster of the percolation configuration. For further background on the subject of group-invariant percolation, we refer the interested reader to \cite{BLPS99,G05,L00,LP16,LS99,P22}.

\subsubsection{\bf Positive Definite Functions in Invariant Percolations.} The main link between the objects defined in Section \ref{sec-prel-Haagerup} and invariant percolation is provided by the following lemma.

\begin{lemma}[Two-point functions of invariant percolations are positive definite functions] \label{lemma-TPFPosDef}
Let $\skrig$ be the Cayley graph of a finitely generated group $\Gamma$ and let $\P$ be a $\Gamma$-invariant bond percolation on $\skrig$. Then $\tau(\cdot,\cdot)$ defines a positive definite invariant kernel. In particular $\tau(o,\cdot)$ defines a positive definite function.
\end{lemma}

\proof As pointed out in the introduction, this observation is due to Aizenman and Newman \cite{AN84}. For the convenience of the reader we include a proof. Let $\P$ be a $\Gamma$-invariant bond percolation on $\skrig$ with percolation configuration $\omega$. Invariance of the kernel $\tau(\cdot,\cdot)$ follows from $\Gamma$-invariance of $\P$. Let $g_1,\ldots,g_n\in\Gamma$ and $a_1,\ldots,a_n\in\C$ be arbitrary and let $\skric$ denote the set of open clusters in $\omega$. Then
\begin{flalign*}
\sum_{i,j=1}^n \overline{a_i}a_j\tau(g_i,g_j) &= \sum_{i,j} \overline{a_i}a_j \E\big[ \1_{\{g_i \leftrightarrow g_j\}}\big] = \E \biggl[  \ \sum_{i,j} \overline{a_i}a_j \1_{\{g_i \leftrightarrow g_j\}} \biggr] \\
& =  \E \biggl[ \ \sum_{C\in\skric} \, \sum_{\{g_i,g_j\} \subset C} \overline{a_i}a_j \biggr] = \E \biggl[ \ \sum_{C\in\skric} \, \Bigl| \sum_{g_i \in C} a_i \Bigr|^2 \biggr] \geq 0,
\end{flalign*}
which proves positive definiteness.
\eproof

\subsubsection{\bf Bernoulli Percolation, and the Critical Parameters $p_c$, $p_u$ and $\pv$.}\label{background Bernoulli perc}

As introduced in Section~\ref{introduction}, {\em $p$-Bernoulli bond percolation} is the bond percolation in which every edge is kept independently at random with fixed probability $p\in[0,1]$. By a result of Newman and Schulman \cite{NS81} (see also \cite[Theorem 7.5]{LP16}), it is a consequence of insertion tolerance and ergodicity that the number of infinite clusters in this model is constant a.s.\ with this constant being equal to either $0,1$ or $\infty$. The {\em critical parameter} $p_c=p_c(\skrig)$ is defined as the infimum over all $p\in[0,1]$ such that in $p$-Bernoulli bond percolation, there exists an infinite cluster almost surely. Similarly, the {\em uniqueness threshold} $p_u=p_u(\skrig)$ is the infimum over all $p\in[0,1]$ such that in $p$-Bernoulli bond percolation, there is a unique infinite cluster almost surely. By the well-known {\em Harris-FKG inequality}, increasing events are positively correlated in any Bernoulli bond percolation \cite[Section 5.8]{LP16}.

\begin{example}\label{ex-intro1} \hspace{-1.mm}(Two-point function vanishing at infinity vs.~finite clusters) As mentioned in the introduction, vanishing at infinity of the two-point function does not necessarily imply having no infinite clusters. Indeed, let $T$ denote the $3$-regular tree. It is well-known that $p_c(T)<1$, see e.g.\ \cite[Theorem 5.15]{LP16}. Therefore supercritical $p$-Bernoulli bond percolation $\P_p$, that is to say with $p\in(p_c,1)$, produces an infinite cluster a.s. To see that there are infinitely many infinite clusters, it therefore suffices to rule out the possibility that there is a unique infinite cluster. If there was a unique infinite cluster, then for every two vertices $u$ and $v$ the Harris-FKG inequality and transitivity would imply that $\P_p(u\leftrightarrow v)\geq\P_p(|C(u)|=\infty,|C(v)|=\infty)\geq\P_p(|C(u)|=\infty)^2>0$. This contradicts the fact that the two-point function of $\P_p$ clearly vanishes at infinity.
\end{example}

We now define another threshold $\pv$ pertinent to Bernoulli percolation, which we will relate to the Haagerup property in Theorem \ref{theorem-Spanning}. Let $\skrih$ be a countable, locally finite, connected graph $\skrih=(V,E)$. The need to go beyond Cayley graphs at this stage is due to the fact that we want to relate the Haagerup property to phase transitions in Bernoulli percolation performed on {\em configurations} of invariant percolations, which clearly need not satisfy any form of transitivity. This parameter was introduced by Schonmann \cite[Section 3.1]{S01} as a natural critical point associated with Bernoulli percolation. Following the notation in \cite[Section 3.1]{S01}, we define
$$
\pv \coloneqq \pv (\skrih) \coloneqq \sup \bigg\{ p\in[0,1] \colon  \lim_{r\to\infty} \sup_{\heap{u,v \in V}{d(u,v)>r}} \tau_p(u,v) = 0 \bigg\},
$$
where $\tau_p$ denotes the two-point function of $p$-Bernoulli bond percolation on $\skrih$. A related parameter, also introduced by Schonmann \cite[Section 3.1]{S01}, is defined as
$$
\begin{aligned}
\pl \coloneqq \pl(\skrih) &\coloneqq \inf \bigg\{ p\in[0,1] \colon \inf_{u,v\in V} \tau_p(u,v) > 0 \bigg\} \\
&= \sup \bigg\{ p\in[0,1] \colon \inf_{u,v\in V} \tau_p(u,v) = 0 \bigg\}.
\end{aligned}
$$
It is immediate from the definitions that
\begin{equation} \label{equ-parameters1}
p_c(\skrih) \leq \pv(\skrih) \leq \pl(\skrih).
\vspace{1mm}
\end{equation}
If $\skrih$ is quasi-transitive, an application of the Harris-FKG inequality as in Example \ref{ex-intro1} shows that $\pl\leq p_u$. Thus, for quasi-transitive $\skrih$ we may extend (\ref{equ-parameters1}) as follows
\begin{equation}
p_c(\skrih) \leq \pv(\skrih) \leq \pl(\skrih) \leq p_u(\skrih).
\vspace{1mm}
\end{equation}
By Lyons and Schramm \cite[Theorem 4.1]{LS99}, $p_u=\pl$ for Cayley graphs. We note that for every tree $\pv=\pl=p_u=1$. For a further discussion of these parameters see \cite{S01}.

\subsection{Point Processes.} \label{subsection-PointProcesses}

Let $(X,\skrib)$ be a measurable space. A measure on $X$ is called {\it$s$-finite} if it is a countable sum of finite measures. Note that every $\sigma$-finite measure is $s$-finite. Let $\mathbf N \coloneqq \mathbf N(X)$ denote the space of all measures on $X$ which can be written as countable sums of finite, $\N_0$-valued measures. We equip $\mathbf N$ with the $\sigma$-field $\mathcal N \coloneqq \mathcal N (X)$ generated by the collection of all sets of the form 
$$
\{\mu \in \mathbf N(X) \colon \mu(B)=k\} \quad B\in\skrib,k\in\N_0.
$$ 
A {\it point process} on $X$ is a random element of $(\mathbf N,\mathcal N)$. A point process $\eta$ on $X$ is called {\it proper} if there exist random elements $\xi_1,\xi_2,\ldots$ and an $\overline{\N_0}$-valued random variable $M$ such that 
$$
\eta=\sum_{i=1}^M \xi_i \quad \mbox{a.s.}
$$
Let $\lambda$ be an $s$-finite measure. A {\em Poisson process} with {\em intensity measure} $\lambda$ is a point process $\eta$ on $X$ with the following two properties:
\begin{enumerate}
\item for every $B\in\skrib$ the distribution of $\eta(B)$ is Poisson with parameter $\lambda(B)$, that is to say
$$
\P[\eta(B)=k]= \e^{-\lambda(B)} \frac{\lambda(B)^k}{k!} \quad \text{for every } k \in \N_0;
$$
\item for every $n\in\N$ and every collection of pairwise disjoint sets $B_1,\ldots,B_n\in\skrib$, the random variables $\eta_p(B_1),\ldots,\eta_p(B_n)$ are independent.
\end{enumerate}
We need the standard fact that such processes exist.

\begin{theorem}[Existence of the Poisson process, cf.~{\cite[Theorem 3.6]{LP17}}] \label{theorem-ExistencePoissonProcess} 
Let $\lambda$ be an $s$-finite measure on a measurable space $(X,\skrib)$, then there exists a Poisson process on $X$ with intensity measure $\lambda$.
\end{theorem}

A useful property of Poisson processes is that they satisfy a version of the Harris-FKG inequality, which we now recall \cite[Section 20.3]{LP17}. Given $B\in\skrib$, a measurable real-valued function $f$ on $\mathbf N(X)$ is called {\em increasing on} $B$ if 
$$
f(\mu+\delta_x)\geq f(\mu) \quad \mbox{for all} \, \, \mu\in\mathbf N(X) \, \, \mbox{and all} \, \, x\in B,
$$
and is called {\em decreasing on} $B$ if $(-f)$ is increasing on $B$.

\begin{theorem}[Harris-FKG inequality for the Poisson process, cf.~{\cite[Theorem 20.4]{LP17}}] \label{theorem-Harris-FKG-Poisson}
Let $\eta$ be a proper Poisson process on a measurable space $(X,\skrib)$ with $\sigma$-finite intensity measure $\lambda$ and law $\P_\eta$. Let $B\in\skrib$ and let $f,f'\in L^2(\P_\eta)$ be increasing on $B$ and decreasing on $X\setminus B$, then 
$$
\E\big[f(\eta)f'(\eta)\big] \geq \E\big[f(\eta)\big] \, \E\big[f'(\eta)\big].
$$
\end{theorem}

\medskip

\section{\bf Haagerup Property and Quantitative Estimates} \label{haagerup}

In this section, we will prove Theorem \ref{maintheorem-Haagerup} and provide our general construction and develop quantative estimates. Theorem~\ref{maintheorem-Haagerup} is proved in Section \ref{subsection-MeasuredWalls}. Before considering this general case, we demonstrate our result for two sub-classes of groups with the Haagerup property: groups acting properly on trees, see Section \ref{subsection-Trees}, and groups acting properly on spaces with walls, see Section \ref{subsection-Walls}. In these cases, the construction is more concrete and they can thus serve as useful intuition. Moreover, they provide interesting examples, which are illustrated along the way. In Section~\ref{sec-generalized}, we state the generalized version of our construction. We then provide the sufficient criterion for the Haagerup property in terms of Bernoulli percolation, see Section \ref{section-HaagerupBernoulli}. In Section~\ref{section-HaagerupSite}, we demonstrate the modifications required to obtain the analogous results for site percolation. In Section~\ref{subsec clusterfreq}, we show that invariant percolations whose two-point function vanishes at infinity do not have finitely many infinite clusters.

\subsection{Percolation and Groups Acting on Trees.} \label{subsection-Trees}

Let $T$ be a locally finite tree. We say that $\Gamma$ {\em acts properly on $T$} if it acts on $T$ via automorphisms and such that $|\{ g\in\Gamma \colon gF \cap F \neq \emptyset \}|<\infty$ for every finite $F\subset T$.

\begin{prop} \label{prop-GroupsActingOnTrees} 
Let $\skrig$ be the Cayley graph of a finitely generated group $\Gamma$. If $\Gamma$ acts properly on a locally finite tree, then for every $\alpha<1$, there exists a $\Gamma$-invariant bond percolation $\P$ on $\skrig$ such that $\E[\deg_\omega(o)]>\alpha\deg(o)$ and $\tau(o,\cdot)\in C_0(\Gamma)$.
\end{prop}

\proof Let $T=(V(T),E(T))$ be a locally finite tree on which $\Gamma$ acts properly and fix a root $x_0 \in V(T)$. For $p\in[0,1]$, let $\P_p$ denote $p$-Bernoulli bond percolation on $T$. If $p<1$ we have 
\begin{equation} \label{equ: TPF decay tree}
\lim_{d(x_0,x) \to \infty} \P_p\big[x_0 \leftrightarrow x\big] = 0
\end{equation}
because the removal of any edge in the unique path between $x_0$ and $x$ disconnects the two vertices. 

Let $\xi_p$ denote the configuration of a $p$-Bernoulli bond percolation on $T$. We define the configuration $\omega_p$ of a bond percolation on $\skrig=(V,E)$ as follows: for each edge $e=[g,h]\in E$, let $e$ be in $\omega_p$ if and only if one of the following holds:
\begin{enumerate}
\item $gx_0=hx_0$,
\item $gx_0\neq hx_0$ but $gx_0$ is connected to $hx_0$ in $\xi_p$.
\end{enumerate}
Let $\bP_p$ denote the law of $\omega_p$. Note that $\mathbf P_p$ defines a $\Gamma$-invariant bond percolation on $\skrig$. Indeed, this follows from the fact that it is a {\em factor}, i.e.~push-forward under an equivariant measurable map, of the $\Aut(T)$-invariant measure $\mathbb P_p$.

We note that $\bP_p$ clearly has marginals arbitrarily close to $1$ for $p$ sufficiently close to $1$. This implies that $\bE_p[\deg_\omega(o)]>\alpha \deg(o)$ for $p$ sufficiently close to $1$. Hence it suffices to prove that the two-point function $\tau_p(o,g)=\bP_p[o\leftrightarrow g]$ of $\bP_p$ satisfies $\tau_p(o,\cdot)\in C_0(\Gamma)$ for every $p<1$. To show this, simply observe that on the event $\{o\leftrightarrow g\}$ the unique path which joins $x_0$ to $gx_0$ must be contained in $\xi_p$. Thus 
$$
\tau_p(o,g) = \bP_p\big[o \leftrightarrow g\big] \leq \P_p\big[x_0 \leftrightarrow gx_0\big].
$$
Moreover, since the action of $\Gamma$ on $T$ is proper, we have that $\lim_{g \to\infty} d(x_0,gx_0)=\infty$. Combined with Equation (\ref{equ: TPF decay tree}) we therefore obtain
$$
\lim_{g \to \infty} \tau_p(o,g) \leq \lim_{g \to \infty} \P_p\big[x_0 \leftrightarrow gx_0\big]=0,
$$
which concludes the proof. \eproof

\begin{example} \label{example-WordLengthFree} \hspace{-1.5mm}(Bernoulli-percolation-proof that tree distances are conditionally negative definite)
Let us comment on the particular example when $\Gamma$ is a finitely generated free group acting on its canonical Cayley graph. In this case it is easy to see that the two-point function of $p$-Bernoulli bond percolation is given by
$$
\tau_p(g,h)=p^{d(g,h)}=\e^{\log(p)d(g,h)}.
$$
Consider the distance function $|g|=d(o,g)$ on $\Gamma$. Since $p a\mapsto \log(p)$ is a bijection from $(0,1)$ to $(-\infty,0)$, we obtain that the function $g \mapsto \e^{-\lambda|g|}$ is positive definite for every $\lambda>0$. By Schoenberg's theorem, the distance function is conditionally negative definite, see e.g.\ \cite[Theorem C.3.2]{BdlHV08}. This recovers a fundamental result due to Haagerup \cite[Lemma 1.2]{H79} using a short probabilistic proof which involves only the most natural percolation model, Bernoulli percolation. Let us point out that an even easier proof was already known, which involves showing that the distance function is measure definite, which immediately implies conditional negative definiteness -- measure definiteness in turn may be seen by writing down explicitly an embedding of the Cayley tree into an $L^1$-space or by writing the distance as a wall distance; these ideas appear for example in \cite{CMV04,CSV12,L13}. Our purpose of presenting the above probabilistic proof is to underline  a motivating idea, which as we will explore, will generalize to a much broader class of groups with the Haagerup property for which the distance function may be far from being conditionally negative definite.  
\end{example}

\subsection{Percolation and Groups Acting on Spaces with Walls.}\label{subsection-Walls}

The following terminology was introduced by Haglund and Paulin in \cite{HP98}. Let $X$ be a set and let $\skriw$ be a set of partitions of $X$ into two classes. A partition $W\in\skriw$ is called a {\em wall} and we say that $W$ {\em separates} two distinct points $x$ and $y$ in $X$ if the points belong to different classes in $W$. The pair $(X,\skriw)$ is a {\em space with walls} if, for every two distinct points $x$ and $y$ in $X$, the set
$$
\skriw(x,y) \coloneqq \{ W \in \skriw \colon W \text{ separates } x \text{ and } y \} 
$$
is finite. In this case $w(x,y) \coloneqq |\skriw(x,y)|$ is called the {\em number of walls separating $x$ and $y$}. We say that a finitely generated group $\Gamma$ {\em acts properly} on a space with walls $(X,\skriw)$ if $\Gamma$ acts on $X$, preserves the family of walls and for some $x_0\in X$ (and hence for all $x\in X$) the function $\Gamma \to \N_0 \, , \, g \mapsto w(x_0,gx_0)$ is {\em proper}, that is to say $\lim_{g\to\infty} w(x_0,gx_0)=\infty.$

\begin{prop} \label{prop-GroupsActingOnSpacesWithWalls} 
Let $\skrig$ be the Cayley graph of a finitely generated group $\Gamma$. If $\Gamma$ acts properly on a space with walls, then for every $\alpha<1$, there exists a $\Gamma$-invariant bond percolation $\P$ on $\skrig$ such that $\E[\deg_\omega(o)]>\alpha\deg(o)$ and $\tau(o,\cdot)\in C_0(\Gamma)$.
\end{prop}

\proof Let $\Gamma$ act properly on the space with walls $(X,\skriw)$ and fix a point $x_0\in X$. Let $p\in(0,1)$ and let $\{Z_W\}_{W \in \skriw}$ be iid Bernoulli-$p$-distributed random variables defined on some probability space $(\Omega,\skrif,\P)$. We define the configuration $\omega_p$ of a bond percolation on $\skrig=(V,E)$ as follows: for each edge $e=[g,h]\in E$, let $e$ be contained in $\omega_p$ if and only if one of the following holds:
\begin{enumerate}
\item $gx_0=hx_0$,
\item $gx_0\neq hx_0$ and for every $W \in \skriw(gx_0,hx_0)$ we have $Z_W=1$.
\end{enumerate}
Let $\bP_p$ denote the law of $\omega_p$. Then $\mathbf P_p$, as a factor of an iid product measure, defines a $\Gamma$-invariant bond percolation.

We observe that $\bP_p$ has marginals arbitrarily close to $1$ for $p$ sufficiently close to $1$: consider any edge $e=[g,h]\in E$, then
\begin{flalign*}
\bP_p\big[ [g,h] \in \omega_p \big] & = \1_{\{gx_0=hx_0\}} + \1_{\{gx_0\neq hx_0\}}\P\big[ Z_W=1 \text{ for all } W \in \skriw(gx_0,hx_0)\big] \\
& =  \1_{\{gx_0=hx_0\}} + \1_{\{gx_0\neq hx_0\}} p^{w(gx_0,hx_0)},
\end{flalign*}
which converges to $1$ as $p\to1$. Hence $\bE_p[\deg_\omega(o)]>\alpha \deg(o)$ for $p$ sufficiently close to $1$. It remains to show that the two-point function $\tau_p(o,g)=\bP_p[o\leftrightarrow g]$ of $\bP_p$ satisfies $\tau_p(o,\cdot)\in C_0(\Gamma)$. To see this let $g\in\Gamma$ be arbitrary, let $o=h_0,h_1,h_2,\ldots,h_n=g$ be any path of vertices joining $o$ to $g$ in $\skrig$ and let $W$ be any wall separating $x_0$ and $gx_0$. Let $B \sqcup B^c=X$ be the partition determined by $W$ and assume without loss of generality that $x_0\in B$. By assumption $gx_0\in B^c$ and thus 
$$
i \coloneqq \max \{ j \colon 0\leq j \leq n-1 \text{ and } h_j x_0 \in B\}
$$
is well defined and satisfies $h_i x_0 \in B$ as well as $h_{i+1} x_0 \in B^c$. In other words the wall $W$ also separates $h_i x_0$ and $h_{i+1} x_0$. With this observation, we now prove the claim: assume that $o$ is connected to $g$ in $\omega_p$, then there exists some path of open edges $e_1,\ldots,e_n$ joining $o$ to $g$ in $\skrig$. The above argument shows that for every wall $W\in \skriw(x_0,gx_0)$ there exists an edge $e=[h,\gamma] \in \{e_1,\ldots,e_n\}$ such that $W\in\skriw(hx_0,\gamma x_0)$. By definition, $e$ being open implies that $Z_W=1$. Therefore we obtain that on the event $\{o\leftrightarrow g\}$, we must have $Z_W=1$ for every $W \in \skriw(x_0,gx_0)$. Thus
$$
\tau_p(o,g) = \bP_p\big[o \leftrightarrow g\big] \leq \P\big[ Z_W = 1 \text{ for every } W \in \skriw(x_0,gx_0)\big] = p^{w(x_0,gx_0)}.
$$
Since $\Gamma$ acts properly, it follows that
$$
\lim_{g\to\infty} \tau_p(o,g) \leq \lim_{g\to\infty} p^{w(x_0,gx_0)}=0,
$$
which concludes the proof. \eproof

\begin{example}\label{ex-wallsZd} \hspace{-1.mm}(Wall structures for Euclidean lattices and trees)  
We give two familiar examples of spaces with walls:

$\bullet$ A locally finite tree $T$ has a natural structure of space with walls by declaring $\skriw$ to consist of partitions into two connected components which are obtained by removing an edge. A group acts properly on $(T,\skriw)$ if and only if it acts properly on the tree $T$.

$\bullet$ The graph $\Z^2$ can be split into two discrete half-spaces orthogonal to the horizontal axis, resp.\ vertical axis, by cutting for fixed $x\in\Z$ along all edges of the form $[(x,y),(x+1,y)]$, resp.\ by cutting for fixed $y\in\Z$ along all edges of the form $[(x,y),(x,y+1)]$. This defines a family of walls on $\Z^2$ such that the number of walls separating two points is equal to their graph distance. Clearly the same applies in the case of $\Z^d$ for some $d\in\N$. 
\end{example}

\subsection{Invariant Point Processes and Groups Acting on Spaces with Measured Walls.} \label{subsection-MeasuredWalls}

 In this section, we will prove Theorem \ref{maintheorem-Haagerup}. As mentioned in Section \ref{sec novel ideas}, a key idea will be to introduce invariant Poisson point processes on suitable infinite measure spaces. These measure spaces generalize the spaces used in Section \ref{subsection-Trees} and Section \ref{subsection-Walls} in such a way that the main intuition behind the constructions of the previous two sections is still valid. For this purpose, the appropriate generalization turns out to be the concept of a space with measured walls, which is due to Cherix, Martin and Valette~\cite{CMV04}, and which we now recall.
 
 Let $X$ be a set, $\skriw$ a set of walls on $X$, $\skrib$ a $\sigma$-algebra on $\skriw$ and $\mu$ a measure on $\skrib$. The $4$-tuple $(X,\skriw,\skrib,\mu)$ is a {\em space with measured walls} if, for every two distinct points $x$ and $y$ in $X$, the set $\skriw(x,y)$ of walls separating $x$ and $y$ is $\skrib$-measurable and $w(x,y)\coloneqq\mu(\skriw(x,y))<\infty$. Note that if $\mu$ is a counting measure, then we recover a space with walls. We say that a finitely generated group $\Gamma$ {\em acts properly} on a space with measured walls $(X,\skriw,\skrib,\mu)$ if $\Gamma$ acts on $X$, preserves the family of walls, preserves the function $w(\cdot,\cdot)$ and for some $x_0\in X$ (and hence for all $x\in X$) the function $\Gamma \to [0,\infty) \, , \, g \mapsto w(x_0,gx_0)$ is proper. This again extends the definition of a proper action on a space with walls in the case that $\mu$ is a counting measure. Note that 
\begin{equation} \label{equ:defwalldistance}
\Gamma\times\Gamma\to[0,\infty) \, , \, D(g,h)\coloneqq w(gx_0,hx_0)
\end{equation} 
defines a pseudo-distance on $\Gamma$ which we call the {\em wall distance}.

It was proved in \cite[Proposition 1]{CMV04} that a locally compact group which acts properly on a space with measured walls has the Haagerup property. Moreover it was proved in \cite[Theorem 1.(1)]{CMV04} that a countable discrete group with the Haagerup property admits a proper action on a space with measured walls. In fact, building on the ideas developed in \cite{RS98}, the proof given in \cite{CMV04}  provides explicitly a space with measured walls and a corresponding proper action, see also \cite[Proposition 7.5.1]{CCJJV01}. Namely, for a finitely generated group $\Gamma$, we may consider the following:
\begin{itemize}
\item[(1)] Let $X=\Gamma$ on which $\Gamma$ acts via left-multiplication and let 
$$
\Omega = \{0,1\}^\Gamma \setminus \{ \omega_0, \omega_1 \},
$$
where $\omega_0(g)=0$ and $\omega(g)=1$ for every $g\in\Gamma$ -- i.e.~we consider the set of all non-empty and non-full subsets of $\Gamma$, equipped with the Borel $\sigma$-algebra $\skrib(\Omega)$. Then $\Gamma$ has a natural measurable action on $\Omega$ defined by $g\omega(x)=\omega(g^{-1}x)$.
\item[(2)] For $\omega\in\Omega$ define $W_\omega=\{\omega,\omega^c\}$ and let $\skriw=\{W_\omega \colon \omega \in \Omega\}$. Clearly $\skriw$ defines a family of walls on $X$ which is preserved by the action of $\Gamma$.
\item[(3)] The set $\skriw$ can be identified with the quotient of $\Omega$ by the fixed-point free involution $\omega \mapsto \omega^c$. Let $\skrib$ be the direct image of $\skrib(\Omega)$. 
\item[(4)] For $g\in\Gamma$ define $S_g=\{\omega\in\Omega \colon \omega(g)=1\} \in \skrib(\Omega)$. Let $\psi$ be a conditionally negative definite function on $\Gamma$. By \cite[Proposition 1.4]{RS98}, there exists a regular Borel measure $\nu$ on $\Omega$ such that
$$
\sqrt{\psi(g^{-1}h)}=\nu(S_g\Delta S_h) \quad \text{for all } g,h\in\Gamma.
$$
In fact, every open subset of $\Omega$ is $\sigma$-compact and we may assume that $\nu(K)<\infty$ for every compact subset $K\subset\Omega$, see \cite[Proof of Proposition 1.2]{RS98}. In particular, $\nu$ is $\sigma$-finite.
Moreover, it follows from the proof of \cite[Theorem 2.1]{RS98} that $\nu$ is $\Gamma$-invariant. Let $\mu$ denote the direct image of $\nu$ on $\skriw$. In particular, $\mu$ is $\Gamma$-invariant, $\sigma$-finite and
\begin{equation} \label{def mu}
\mu(\skriw(g,h)) = \nu(S_g\Delta S_h) = \sqrt{\psi(g^{-1}h)} < \infty \quad \text{for all } g,h\in\Gamma.
\end{equation}
\end{itemize}
Then $(\Gamma,\skriw,\skrib,\mu)$ defines a space with measured walls. Moreover, the induced $\Gamma$ action on the set of walls is measurable, leaves $\mu$ invariant and satisfies $\skriw(\gamma g,\gamma h)= \gamma \skriw(g,h)$. In particular the function $w(\cdot,\cdot)$ is preserved. Now if $\Gamma$ has the Haagerup property then there exists a proper, conditionally negative definite function $\psi$ on $\Gamma$. In this case, $\Gamma$ clearly acts properly on $(\Gamma,\skriw,\skrib,\mu)$.

\begin{theorem} \label{theorem-GroupsWithTheHaagerupProperty} 
Let $\skrig$ be the Cayley graph of a finitely generated group $\Gamma$. If $\Gamma$ has the Haagerup property, then for every $\alpha<1$, there exists a $\Gamma$-invariant bond percolation $\P$ on $\skrig$ such that $\E[\deg_\omega(o)]>\alpha\deg(o)$ and $\tau(o,\cdot)\in C_0(\Gamma)$.
\end{theorem}

The key ingredient of the proof of Theorem \ref{theorem-GroupsWithTheHaagerupProperty} is the following

\begin{lemma}\label{lemma3.5}
    Let $\skrig$ be the Cayley graph of a finitely generated group $\Gamma$. Let $\psi$ be a conditionally negative definite function on $\Gamma$. Then for every $p\in(0,1)$ there exists a $\Gamma$-invariant bond percolation $\mathbf P_p$ with the following properties: for every $e=[g_1,g_2]\in E$, $g\in\Gamma$
    \begin{equation}
        \mathbf P_p[ e\in\omega ]=\e^{-(1-p)\sqrt{\psi(g_1^{-1}g_2)}}, \qquad \mathbf P_p[o \leftrightarrow g] \leq \e^{-(1-p)\sqrt{\psi(g)}}.
    \end{equation}
\end{lemma}

\proof Given the conditionally negative definite function $\psi$ on $\Gamma$, consider the action of $\Gamma$ on the space with measured walls $(\Gamma,\skriw,\skrib,\mu)$ described in (1)--(4) above Theorem \ref{theorem-GroupsWithTheHaagerupProperty}. Let $p\in(0,1)$ and let $\eta_p$ be a Poisson process on $\skriw$ with intensity measure $(1-p)\mu$ defined on some probability space $(\Omega,\skrif,\P)$. Existence follows from Theorem \ref{theorem-ExistencePoissonProcess} because the measure $\mu$ is $\sigma$-finite as recalled above \eqref{def mu}. We define the configuration $\omega_p$ of a bond percolation on $\skrig=(V,E)$ as follows: for each edge $e=[g,h]\in E$, let $e$ be contained in $\omega_p$ if and only if $\eta_p(\skriw(g,h))=0$. Let $\bP_p$ denote the law of $\omega_p$. We claim that $\bP_p$ defines a $\Gamma$-invariant bond percolation. We first show that, for every $g\in\Gamma$, the process $\eta_p \circ g$ has the same law under $\P$ as $\eta_p$. To shorten notation, denote the intensity measure by $\lambda \coloneqq (1-p)\mu$ and recall that $\lambda$ is $\Gamma$-invariant. Now let $n\geq1$ and consider pairwise disjoint $B_1,\ldots,B_n\in\skrib$. Then 
$$
\big(\eta_p(B_1),\ldots,\eta_p(B_n)\big) \overset{d}{=} \Poi\big(\lambda(B_1)\big)\otimes \ldots \otimes \Poi\big(\lambda(B_n)\big).
$$ 
 Moreover, $gB_1,\ldots,gB_n\in\skrib$ are pairwise disjoint with $\lambda(gB_i)=\lambda(B_i)$ for all $i\in\{1,\ldots,n\}$ by $\Gamma$-invariance of $\lambda$. It follows that 
 $$
\big(\eta_p \circ g(B_1),\ldots,\eta_p \circ g(B_n)\big) = \big(\eta_p (gB_1),\ldots,\eta_p (gB_n)\big)  \overset{d}{=} \Poi\big(\lambda(B_1)\big)\otimes \ldots \otimes \Poi\big(\lambda(B_n)\big).
$$ 
By \cite[Proposition 2.10 (i) $\Leftrightarrow$ (ii)]{LP17} the point process $\eta_p \circ g$ has the same law as $\eta_p$. The fact that $\mathbf P_p$ defines a $\Gamma$-invariant bond percolation now follows from the fact that it is a factor of $\mathbb P^{\eta_p}$. 

Let $e=[g,h]\in E$, then
\begin{equation}\label{eq1}
\bP_p\big[ [g,h] \in \omega \big] = \P\big[ \eta_p(\skriw(g,h))=0 \big] = \e^{-(1-p)\mu(\skriw(g,h))}.
\end{equation}
Let  $g\in\Gamma$. Recall from the proof of Proposition \ref{prop-GroupsActingOnSpacesWithWalls} that, given any path of edges $e_1,\ldots,e_n$ joining $o$ to $g$ in $\skrig$ and any wall $W\in\skriw(o,g)$, there exists $e\in\{e_1,\ldots,e_n\}$ such that $W$ also separates the endpoints of the edge~$e$. Writing $\skriw(e)$ for the set of walls separating the endpoints of an edge $e$, it follows that 
$$
\skriw(o,g) \subset \bigcup_{i=1}^n \skriw(e_i).
$$
We obtain that on the event that $o$ and $g$ are in the same cluster of $\omega_p$, we have $\eta_p(\skriw(o,g))=0$. Thus
\begin{equation}\label{eq2}
 \tau_p(o,g) = \bP_p\big[o \leftrightarrow g\big] \leq  \P\big[ \, \eta_p(\skriw(o,g))=0 \big] = \e^{-(1-p)\mu(\skriw(o,g))}. 
\end{equation}
Then \eqref{def mu}, \eqref{eq1} and \eqref{eq2} complete the proof of the lemma.\eproof

\medskip

{\bf\noindent Proof of Theorem \ref{theorem-GroupsWithTheHaagerupProperty}.} 
Suppose $\Gamma$ has the Haagerup property. Then let $\psi$ be a proper, conditionally negative definite function on $\Gamma$. In particular, by Lemma \ref{lemma3.5}, for every $p\in(0,1)$ there exists a $\Gamma$-invariant bond percolation $\mathbf P_p$ with the following properties: 
$$
\mathbf P_p[ e\in\omega ]=\e^{-(1-p)\sqrt{\psi(g^{-1}h)}}
$$
for every $e=[g,h]\in E$ and two-point function $\tau_p(o,g)=\bP_p[o\leftrightarrow g]$ of $\bP_p$ satisfying
$$
\tau_p(o,g) \leq \e^{-(1-p)\sqrt{\psi(g)}}.
$$
Again, $\bP_p$ has marginals arbitrarily close to $1$ for $p$ sufficiently close to $1$ as
$$
\mathbf P_p[ e\in\omega ]=\e^{-(1-p)\sqrt{\psi(g^{-1}h)}}
$$
which, since $0 \leq \psi(g^{-1}h) < \infty$, converges to $1$ as $p\to1$. To conclude the proof, it remains to show that the two-point function satisfies $\tau_p(o,\cdot)\in C_0(\Gamma)$. Since $\psi(g)\to\infty$ as $g\to\infty$ by properness of $\psi$,
$$
\lim_{g\to\infty} \tau_p(o,g) \leq \lim_{g\to\infty} \e^{-(1-p)\sqrt{\psi(g)}}=0,
$$
which concludes the proof. \eproof

\noindent{\bf Proof of Theorem \ref{maintheorem-Haagerup}.} Suppose that for every $\alpha<1$, there exists a $\Gamma$-invariant bond percolation $\P$ on $\skrig$ such that $\E[\deg_\omega(o)]>\alpha\deg(o)$ and $\tau(o,\cdot)\in C_0(\Gamma)$. Choose a sequence $(\P_i)_{i\in\N}$ of $\Gamma$-invariant bond percolations on $\skrig$ such that the two point function $\tau_i(o,\cdot)$ of each $\P_i$ vanishes at infinity and such that
$$
\lim_{i\to\infty} \E_i[\deg_\omega(o)]=\deg(o).
$$
Then $\tau_i(o,\cdot)$ is a sequence of positive definite functions, see Lemma \ref{lemma-TPFPosDef}, which vanish at infinity and which converge pointwise to $1$. That is to say $\Gamma$ has the Haagerup property.  The converse direction is exactly Theorem \ref{theorem-GroupsWithTheHaagerupProperty} above. \eproof

\subsection{Quantitative Estimates: Wall Distance, Marginals and Two-Point Functions.}\label{sec-generalized}

Our next result is a general version of the construction used in the proof of Theorem \ref{theorem-GroupsWithTheHaagerupProperty}. We are primarily interested in such a generalization because there exist many known actions of groups on spaces with measured walls which are well understood or have useful geometric features. Our general construction applies to all of these. We first need to introduce the concept of invariance for actions on spaces with measured walls. We say that $\Gamma$ has an {\em invariant action} on a space with measured walls $(X,\skriw,\skrib,\mu)$, if $\Gamma$ acts on $X$ and the maps
$$
\skriw \to \skriw \, , \, \{B,B^c\} = W \mapsto gW=\{ gB, gB^c \} 
$$ 
define a $\skrib$-measurable $\Gamma$-action on $\skriw$ for which the measure $\mu$ is invariant. 

 We briefly compare this definition to the one from \cite{CMV04} recalled above \eqref{equ:defwalldistance}. Here, instead of requiring that $\Gamma$ acts on $X$ such that the wall distance is preserved, we also require $\Gamma$ to preserve the measure~$\mu$.

\begin{theorem}[\bf Quantitative Estimate on Two-Point Function] \label{theorem-GeneralConstruction} 
Let $\skrig$ be the Cayley graph of a finitely generated group $\Gamma$. Suppose that $\Gamma$ has an invariant action on a space with measured walls $(X,\skriw,\skrib,\mu)$ and fix $x_0\in X$. For $p\in(0,1)$, define 
$$
\alpha_p \coloneqq \exp\Bigl( -(1-p) \max_{h\sim o} \, w(x_0,hx_0) \Bigr) \in (0,1).
$$ 

Then there exists a $\Gamma$-invariant bond percolation $\P$ on $\skrig$ such that $\E[\deg_\omega(o)] \geq \alpha_p \deg(o)$ and
$$
\exp \Bigl( -(1-p) \max_{h\sim o} \, w(x_0,hx_0) |g| \Bigr)  \leq \tau(o,g) \leq \exp\bigl( -(1-p) w(x_0,gx_0)  \bigr) \quad \text{for every } g \in \Gamma.
$$
\end{theorem}

\proof We first show that $\Gamma$ admits an invariant action on a space with measured walls $(X_0,\skriw_0,\skrib_0,\mu_0)$ such that the measure $\mu_0$ is $s$-finite and satisfies $\mu_0(\skriw_0(gx_0,hx_0))=\mu(\skriw(gx_0,hx_0))$ for all $g,h\in\Gamma$. Consider the set of walls
$$
\skriw_0 \coloneqq \bigcup_{g,h \in \Gamma} \skriw(gx_0,hx_0) \subset \skriw.
$$
Note that $\skriw_0 \in \mathcal B$ as a countable union of $\skrib$-measurable sets. Let $\skrib_0$ be the restriction of $\skrib$ to $\skriw_0$ and let $\mu_0$ be the restriction of $\mu$ to $\skrib_0$. Then $(X,\skriw_0,\skrib_0,\mu_0)$ is a space with measured walls because for every two distinct points $x$ and $y$ in $X$ the set $\skriw_0(x,y)$ of walls separating $x$ and $y$ is given by
$$
\skriw_0(x,y)= \skriw(x,y) \cap \skriw_0 \in \skrib_0
$$
and satisfies $\mu_0(\skriw_0(x,y))=\mu(\skriw(x,y) \cap \skriw_0) \leq \mu(\skriw(x,y))<\infty$. Moreover, the measure $\mu_0$ is $s$-finite because $\mu_0(\skriw(gx_0,hx_0))<\infty$ for every $g,h\in\Gamma$. Finally, for every $\gamma,g,h \in\Gamma$ and every $W\in\skriw(gx_0,hx_0)$ we have that $\gamma W \in \skriw(\gamma gx_0,\gamma hx_0)$, thus $\skriw_0$ is a $\Gamma$-invariant set and therefore the action of $\Gamma$ restricts to an invariant action of $\Gamma$ on $(X,\skriw_0,\skrib_0,\mu_0)$ as claimed.  For $p\in(0,1)$, define a bond percolation as in the proof of Lemma~\ref{lemma3.5} but now using the action on $(X_0,\skriw_0,\skrib_0,\mu_0)$. The fact that $\P[e\in\omega] =\e^{-(1-p)w(gx_0,hx_0)}$ for every $e=[g,h]\in E$ and thus
$\E[\deg_\omega(o)] \ge \alpha_p \deg(o)$
as well as the upper bound $\tau(o,g) \leq \exp( -(1-p) w(x_0,gx_0))$ for every $g \in \Gamma$ can be proved in the same way. 

For the lower bound, let $g\in\Gamma$ and choose a shortest path $e_1,\ldots,e_{|g|}$ of edges joining $o$ to $g$ in~$\skrig$. Recall that $\skriw(e_i)$ denotes the set of walls separating the endpoints of $e_i$ and define $\skriv \coloneqq \bigcup_{i=1}^{|g|} \skriw(e_i).$
By the union bound and invariance
$$
\mu_0(\skriv) \leq \sum_{i=1}^{|g|} \mu_0(\skriw(e_i)) \leq |g| \max_{e \in E} \mu(\skriw(e)) = |g| \max_{g\sim o} \, w(x_0,gx_0).
$$
Finally, note that $\{o\leftrightarrow g\} \supset \{ e_1,\ldots,e_{|g|} \in \omega\} = \{ \eta(\skriv)=0\}$ and thus 
$$
\tau(o,g) = \P\big[o \leftrightarrow g\big] \geq \bP\big[ \eta(\skriv)=0\big] = \e^{-(1-p)\mu(\skriv)} \geq \exp \Bigl( -(1-p) |g| \max_{g\sim o} \, w(x_0,gx_0) \Bigr).
$$
The proof of Theorem \ref{theorem-GeneralConstruction} is thus complete. \eproof

\subsubsection{\bf Consequence of Theorem \ref{theorem-GeneralConstruction}.} 

Our general construction result Theorem~\ref{theorem-GeneralConstruction} entails significant quantitative information as we will now explain. We start with the description of the relationship between the growth of wall distances and the decay of the two-point functions of the associated percolations. To state this precisely, we first set up the relevant terminology. Let $\Gamma$ be a finitely generated group equipped with a word length $|\cdot|$ and let $\rho\colon[0,\infty]\to[0,\infty]$ be a function.

\begin{itemize}
\item Let $D$ be a left-invariant pseudo-distance on $\Gamma$. Following \cite[Definition 7.1]{CSV12}, we say that $\rho$ is a {\em compression function} for $D$ if $D(o,g)\geq \rho(|g|)$ for all $g\in\Gamma$.

\medskip

\item If $\Gamma$ has an invariant action on a space with measured walls $(X,\skriw,\skrib,\mu)$ such that for some $x_0\in X$, the {\em wall distance} $D(g,h)=w(gx_0,hx_0)$ admits $\rho$ as a compression function, we say that the action has {\em $\rho$-growth}. If $\rho$ is linear, we say that the action has {\em linear growth}.

\medskip

\item Let $C>0$ be a constant and let $\tau$ be the two-point function of a $\Gamma$-invariant bond percolation. We say that $\tau$ has {\em $(C,\rho)$-exponential decay} if there exist $\beta,\gamma>0$ such that $\beta/\gamma\leq C$ and
$$
\e^{-\beta |g|} \leq \tau(o,g) \leq \e^{-\gamma \rho(|g|)} \quad \mbox{for all} \, \, g \in \Gamma.
$$
If $\rho$ is linear, we say that $\tau$ has {\em exponential decay}. 

\end{itemize}

Note that wall distances associated to an invariant action of $\Gamma$ on a space with measured walls and measure definite kernels on $\Gamma$ can have at most linear growth (by Example \ref{example-WordLengthFree}, resp.\ Example \ref{ex-wallsZd}(ii), finitely generated free groups, resp.~$\Z^d$ with $d\in\N$, admit wall distances with linear growth).

\begin{cor}\label{cor-QuantitativeDecay}
Let $\skrig$ be the Cayley graph of a finitely generated group $\Gamma$ and let $\rho\colon[0,\infty]\to[0,\infty]$ be a function with $\{\rho=0\}=\{0\}$. Consider the following conditions:
\begin{enumerate}
\item[{\rm(i)}] There exists an invariant action of $\Gamma$ on a space with measured walls with $\rho$-growth;
\item[{\rm(ii)}] There exists $C>0$ such that for every $\alpha<1$, there exists a $\Gamma$-invariant bond percolation $\P$ on $\skrig$ with $\E[\deg_\omega(o)]>\alpha\deg(o)$ and with $(C,\rho)$-exponential decay.
\end{enumerate}
Then condition (i) implies condition (ii).
\end{cor}

The geometric meaning of Corollary \ref{cor-QuantitativeDecay} is explained in the following remark.

\begin{remark}[Connection with the equivariant $L^1$-compression exponent] \label{remark-L1Compression} 
Condition (i) in Corollary \ref{cor-QuantitativeDecay} is related to the {\em equivariant $L^1$-compression exponent} $\alpha_1^{\scriptscriptstyle{\#}}(\Gamma)$ of the group $\Gamma$, which is defined to be the supremum over all $\alpha\in[0,1]$, such that there exists an isometric action of $\Gamma$ on an $L^1$-space $E$ together with a $\Gamma$-equivariant map $f\colon\Gamma\to E$ with 
$$
\Vert f(g)-f(h) \Vert_E \geq |g^{-1}h|^\alpha
\vspace{1mm}
$$
for all $g,h\in\Gamma$. For background on this exponent, see \cite{GK04,NP11}. As observed in \cite[Section~7]{CSV12}, the results in \cite{CDH10} imply that 
\begin{equation}\label{eq-L1Rem-Equ}
\Vert f(g)-f(h) \Vert_E
\vspace{1mm}
\end{equation}
is a wall distance associated to an invariant action of $\Gamma$ on a space with measured walls. Thus $\alpha_1^{\scriptscriptstyle{\#}}(\Gamma)$ may equivalently be realized as the supremum over all $\alpha\in[0,1]$ such that there exists an invariant action of $\Gamma$ on a space with measured walls with $\rho_\alpha$-growth for the particular choice $\rho_\alpha(t)=t^\alpha$. In the special case $\rho=\rho_\alpha$, condition (i) therefore amounts to requiring $\alpha_1^{\scriptscriptstyle{\#}}(\Gamma)\geq\alpha.$ In particular, Corollary~\ref{cor-QuantitativeDecay} yields a consequence of ''good'' equivariant $L^1$-compression for the decay rates of two-point functions of invariant bond percolation models.
\end{remark}

\noindent{\bf Proof of Corollary \ref{cor-QuantitativeDecay}.} Let $\Gamma$ have an invariant action on a space with measured walls $(X,\skriw,\skrib,\mu)$ with $\rho$-growth, i.e.\ for some $x_0\in X$ we have that $w(x_0,gx_0) \geq \rho(|g|)$ for all $g\in\Gamma$. Note that $C \coloneqq \max_{g\sim o} \, w(x_0,gx_0) \in (0,\infty)$ and let $p\in(0,1)$. By Theorem \ref{theorem-GeneralConstruction}, there exists a $\Gamma$-invariant bond percolation $\P$ on $\skrig$ such that $\E[\deg_\omega(o)] \geq \e^{-(1-p)C} \deg(o)$ and  
$$
\exp\bigl( -(1-p)C|g| \bigr) \leq \tau(o,g) \leq \exp\bigl( -(1-p) w(x_0,gx_0)  \bigr) 
$$
for every $g\in\Gamma$. Define $\beta_p \coloneqq (1-p)C >0$ and $\gamma_p \coloneqq (1-p)>0$. It follows that $\beta_p/\gamma_p=C$ and
$$
\e^{-\beta_p |g|} \leq \tau(o,g) \leq \e^{-\gamma_p \rho(|g|)} \quad \mbox{for all} \, \, g \in \Gamma,
$$ 
Since $\lim_{p\to1} \e^{-(1-p)C} = 1$, choosing $p$ close to $1$ yields a percolation with the desired properties. \eproof

\subsection{Haagerup Property and Bernoulli Percolation.} \label{section-HaagerupBernoulli}

In this section, we prove a sufficient criterion for the Haagerup property, which is inspired by the amenability criterion given in \cite[Theorem 5.3]{BLPS99}. It would be interesting to know whether our sufficient condition (or a similar condition for Bernoulli percolation on random spanning subgraphs) characterizes the Haagerup property, see Question~\ref{qu:HaagerupBernoulli}. For the definition of $\pv$ we refer to Section \ref{background Bernoulli perc}.

\begin{theorem}[\bf Sufficient Criterion using Bernoulli Percolation] \label{theorem-Spanning}
Let $\skrig$ be the Cayley graph of a finitely generated group $\Gamma$. The following conditions are equivalent:
\begin{enumerate}
\item[{\rm (i)}] there exists a $\Gamma$-invariant random spanning subgraph of $\skrig$ with $\pv=1$ a.s.;
\item[{\rm (ii)}] there exists a $\Gamma$-invariant random nonempty connected subgraph of $\skrig$ with $\pv=1$ w.p.p.
\end{enumerate}
Moreover, if $\Gamma$ satisfies {\rm (i)}, or equivalently {\rm (ii)}, then $\Gamma$ has the Haagerup property.
\end{theorem}

\proof 
The implication ''(i) $\Rightarrow$ (ii)'' is obvious.

''(ii) $\Rightarrow$ (i)'': We use a strategy of efficiently augmenting an invariant percolation with a unique (necessarily infinite) cluster to an invariant random spanning subgraph, see e.g.\ \cite[Proposition 2.1]{HP20} for another application of this strategy. Let $\omega$ be a $\Gamma$-invariant random nonempty connected subgraph of $\skrig$ with $\pv=1$ with positive probability. Conditioning on this event, we may assume that $\pv=1$ a.s. Now we augment the configuration $\omega$ as follows: for every vertex which is not already in $\omega$, choose an edge uniformly among those edges connecting it to a vertex closer to $\omega$ in the graph distance of the underlying Cayley graph $\skrig$. The resulting configuration, denoted by $\xi$, is clearly spanning and invariant. Moreover, $\xi \setminus \omega$ is a forest. Therefore the property ''$\pv(\xi)=1$ a.s.'' is trivially inherited from the assumption that $\pv(\omega)=1$ a.s.

Finally, let us prove that (i) (and therefore (ii)) implies the Haagerup property. Therefore, let us assume that condition (i) holds. Let $\omega$ be a $\Gamma$-invariant random spanning subgraph of $\skrig$ with $\pv=1$ a.s. Let $p\in(0,1)$ and let $\tau_p$ denote the two-point function of the bond percolation $\xi_p$ obtained by performing $p$-Bernoulli bond percolation on $\omega$. As the two-point function of a $\Gamma$-invariant bond percolation, $\tau_p(o,\cdot)$ defines a positive definite function. Recalling Theorem \ref{theorem-HaagerupProperty}, it is therefore sufficient to show that $\tau_p(o,\cdot)$ vanishes at infinity and converges pointwise to $1$ as $p\to1$. To that end, fix $g\in\Gamma$. Fixing $\omega$, we see that
$$
\lim_{p\to1} \, \P\Big[ o\xleftrightarrow{\xi_p} g \, \big| \, \omega\Big] = 1
$$
because $o$ and $g$ are joined by some finite path in the spanning configuration $\omega$, which, with probability converging to one is fully contained in $p$-Bernoulli bond percolation as $p$ converges to $1$. The bounded convergence theorem yields that indeed $\lim_{p\to1} \tau_p(o,g)=1$.
We now verify that $\tau_p(o,\cdot)$ vanishes at infinity. To see this, it suffices to show that for every sequence $(g_n)_{n=1}^\infty \subset\Gamma$ with $g_n\to\infty$, we have that $\tau_p(o,g_n)\to0$. Fix such a sequence. Then conditional on the configuration $\omega$, the probability that $o$ and $g_n$ are in the same cluster of $\xi_p$ converges to zero a.s.\ as $n\to\infty$ by assumption. The claim follows again from the bounded convergence theorem.
 \eproof

We may also point out that for nonempty connected subgraphs in groups with property~(T), not only $\pv<1$ but the stronger property $p_u<1$ holds, see Proposition~\ref{prop:Robust} for details.

\subsection{Analogue for Site Percolation.} \label{section-HaagerupSite}

\begin{theorem}[\bf Group-Invariant Site Percolation and the Haagerup Property] \label{cor-SitePercolation}
Let $\skrig$ be the Cayley graph of a finitely generated group $\Gamma$. Then $\Gamma$ has the Haagerup property if and only if for every $\alpha<1$, there exists a $\Gamma$-invariant site percolation $\P$ on $\skrig$ such that $\P[o\in\omega]>\alpha$ and $\tau(o,\cdot)\in C_0(\Gamma)$.
\end{theorem}
\proof The "if" direction follows exactly as in the proof of Theorem \ref{maintheorem-Haagerup} above.

To prove the converse, assume that $\Gamma$ has the Haagerup property. By Theorem \ref{maintheorem-Haagerup} there exists a sequence $(\P_i)_{i=1}^\infty$ of $\Gamma$-invariant bond percolations such that the corresponding two-point functions $\tau_i(o,\cdot)$ vanish at infinity for every $i\in\N$ and such that
$$
\lim_{i\to\infty} \E_i[ \deg_\omega(o)]=\deg(o).
$$
Let $\omega_1,\omega_2,\ldots$ denote configurations with law $\P_i$ defined on some common probability space $(\Omega,\skrif,\P)$. For each $i\in\N$, define the configuration $\xi_i$ of a site percolation as follows: given the configuration $\omega_i$, remove all vertices which are elements of the vertex boundary of some cluster of $\omega_i$. More formally, for a subset $U\subset V$, let
$$
\partial_V(U)=\{x \in V \setminus U \colon \text{there exists } y \in U \text{ with } x \sim y \}
$$
denote the \textit{vertex boundary} of $U$. Let $\mathcal C$ be the collection of all clusters of $\omega_i$ and define, for $x\in V$, that $x\in\xi_i$ if and only if
$$
x \in V \setminus \bigcup_{C \in \mathcal C} \partial_V(C).
$$
Let $\bP_i$ denote the law of $\xi_i$. Because the law of $\omega_i$ is $\Gamma$-invariant, $\bP_i$ is a $\Gamma$-invariant site percolation. Now note that
$$
\{o \in \xi_i \} \supset \bigl\{ [o,g] \in \omega_i \text{ for all } g \sim o \bigr\}
$$
because if all edges incident to $o$ are contained in $\omega_i$, then $o$ and all its neighbors belong to the same cluster and therefore $o$ is not an element of the vertex boundary of any cluster. This implies that
$$
\lim_{i\to\infty} \bP_i[o\in\omega] = \lim_{i\to\infty} \P[o \in \xi_i] \geq \lim_{i\to\infty} \P\bigl[ \, [o,g] \in \omega_i \text{ for all } g \sim o \bigr] = 1.
$$
On the other hand, note that if $o$ and $g$ are not in the same cluster of $\omega_i$, then the vertex boundary of the cluster of $o$ is not contained in $\xi_i$ and therefore separates $o$ and $g$ in $\xi_i$. Taking complements yields that
$$
\{o \leftrightarrow g \text{ in } \xi_i \} \subset \{ o \leftrightarrow g \text{ in } \omega_i \}.
$$
Hence the two-point function of $\bP_i$ is bounded by $\tau_i$, thus also decaying at infinity. It follows that $\bP_i$, for $i$ large enough, has the desired properties. \eproof

\subsection{Vanishing Two-Point Function Implies Infinitely Many Infinite Clusters.}\label{subsec clusterfreq}

We now show that for invariant percolations having a two-point function vanishing at infinity implies that there almost surely are not finitely many infinite clusters. 

\begin{prop} \label{prop:ConsequenceVanishingTPF} Let $\mathcal G$ be the Cayley graph of a finitely generated group $\Gamma$ and let $\mathbb P$ be a $\Gamma$-invariant bond percolation on $\mathcal G$ such that $\tau(o,\cdot)\in C_0(\Gamma)$. Then $\mathbb P$ has a.s.~either only finite clusters or infinitely many infinite clusters.
\end{prop}

It will be clear from the proof that the same result holds for site percolations as well as mixed percolations. We first handle the ergodic case. 

\begin{lemma} \label{lm:ConsequenceVanishingTPFergodic} Let $\mathcal G$ be the Cayley graph of a finitely generated, infinite group $\Gamma$ and let $\mathbb P$ be a $\Gamma$-invariant ergodic bond percolation on $\mathcal G$ such that $\tau(o,\cdot)\in C_0(\Gamma)$. Then $\mathbb P$ either a.s.~has only finite clusters or a.s.~has infinitely many infinite clusters.
\end{lemma}
\begin{proof}
Let $\omega$ denote the configuration of $\mathbb P$ and let $X:=(X_i)_{i=0}^\infty$ be an independent simple random walk on $\mathcal G$ started at $o$ whose law is denoted by $P^X_o$. We will use the fact that random walk on a locally finite, connected, infinite graph cannot be positive recurrent, see e.g.~\cite[Chapter 2]{LP17}, and thus for every $r\ge0$, a.s.
\begin{equation} \label{lm:RW}
\lim_{n\to\infty} \frac{1}{n} \sum_{i=0}^{n-1} \1_{\{d(o,X_i)>r\}} = 1.
\end{equation}

Since $\mathbb P$ is ergodic, the number of infinite clusters is a constant a.s. We will now show that this constant is either $0$ or $\infty$ by ruling out the cases $\{1,2,\ldots\}$.

\medskip

{\noindent\em Case 1.} For clarity of the exposition, we first rule out the possibility of a unique infinite cluster. To reach a contradiction, suppose that there exists a unique infinite cluster $C_\infty$ in $\omega$ a.s.  Set 
$$
c:=\mathbb P(o\in C_\infty)>0.
$$ 
Then $(\1_{C_\infty}(X_i))_{i\ge 0}$ is a stationary sequence with respect to $\mathbf P=\P\otimes P^X_o$ and, by Birkhoff's ergodic theorem,
$$
\alpha(C_\infty,X):= \lim_{n\to\infty} \frac{1}{n} \sum_{i=0}^{n-1} \1_{C_\infty}(X_i) 
$$
exists a.s. It can be checked that $\alpha(C_\infty):=\alpha(C_\infty,X)$ a.s.~does not depend on $X$ (see~\cite[Section~8]{HJ06} for details).
Ergodicity of $\omega$ implies that a.s.
\begin{equation} \label{equ:frequency}
 \lim_{n\to\infty} \frac{1}{n} \sum_{i=0}^{n-1} \1_{C_\infty}(X_i) = \mathbb E [\1_{C_\infty}(X_0)] = c
\end{equation}

Fix $r\ge0$. By \eqref{lm:RW} and \eqref{equ:frequency}, we obtain that a.s.
$$
\lim_{n\to\infty} \frac{1}{n} \sum_{i=0}^{n-1} \1_{C_\infty}(X_i) \1_{\{d(o,X_i)>r\}} = c.
$$
In particular, conditional on the event $\{o\in C_\infty\}$, we obtain that a.s.
\begin{equation} \label{equ:ConditionalFrequency}
\lim_{n\to\infty} \frac{1}{n} \sum_{i=0}^{n-1} \1_{\{X_i\leftrightarrow o, d(o,X_i)> r \}} = c.
\end{equation}
On the other hand, by the bounded convergence theorem and the fact that $\omega$ and $X$ are independent, we have that
\begin{equation} \label{equ:FrequTPFBound}
\mathbb E \bigg[ \lim_{n\to\infty} \frac{1}{n} \sum_{i=0}^{n-1} \1_{\{X_i\leftrightarrow o, d(o,X_i)> r \}} \bigg] \le \sup_{\heap{g \in \Gamma}{d(o,g)>r}} \, \tau(o,g).
\end{equation}
Combining \eqref{equ:ConditionalFrequency} and \eqref{equ:FrequTPFBound}, we obtain that
$$
0 < c \cdot \mathbb P( o\in C_\infty) = \mathbb E \bigg[ \1_{\{o \in C_\infty\}} \lim_{n\to\infty} \frac{1}{n} \sum_{i=0}^{n-1} \1_{\{X_i\leftrightarrow o, d(o,X_i)> r \}} \bigg] \le \sup_{\heap{g \in \Gamma}{d(o,g)>r}} \, \tau(o,g) \overset{r\to\infty}{\longrightarrow} 0,
$$
which provides the desired contradiction.

\medskip

{\noindent\em General case.} Suppose that $\omega$ has $k\in\{1,2,\ldots\}$ infinite clusters a.s. On a larger probability space, let $C_\infty^{\ssup 1}, \ldots, C_\infty^{\ssup k}$ be an invariant choice of the infinite clusters, i.e.~each $C_\infty^{(i)}$ is an infinite cluster of~$\omega$, every infinite cluster of~$\omega$ appears exactly once and the law of each $C_\infty^{(i)}$ is $\Gamma$-invariant.$^3$\footnote{{\color{blue}$^3$One way to do this is by the following construction, cf.~\cite[p.~1828]{LS99}. Let $\{U(v):v\in V\}$ be ${\rm Unif}[0,1]$ random variables such that $\{U(v),\omega: v\in V\}$ are all independent. Let $\{v_1,v_2,\ldots\}$ be an enumeration of $V$. Define a labeling $\kappa_\omega$ of the vertices by $\kappa_\omega(v)=U(v_j)$ if $j\in\{1,2,\ldots\}$ is minimal with $v_j\in C(v)$. Consider the joint law of $(\omega,\kappa_\omega)$. Conditional on $\omega$, every cluster has a label (in the sense that all vertices belonging to this cluster have this label) and the collection of these labels are i.i.d.~${\rm Unif}[0,1]$ random variables. Hence the law of $(\omega,\kappa_\omega)$ does not depend on the enumeration of $V$ and it follows that it is $\Gamma$-invariant. Now order the infinite clusters according to increasing labels.}} Let $C_\infty$ denote the union of $C_\infty^{\ssup 1}, \ldots, C_\infty^{\ssup k}$. By Birkhoff's theorem and ergodicity, we again obtain that a.s.
$$
\lim_{n\to\infty} \frac{1}{n} \sum_{i=0}^{n-1} \1_{C_\infty}(X_i) = \mathbb P(|C_\omega(o)|=\infty)>0.
$$
Birkhoff's theorem also shows that for every $1\le m \le k$ 
$$
\alpha(C_\infty^{(m)}):=\lim_{n\to\infty} \frac{1}{n} \sum_{i=0}^{n-1} \1_{C_\infty^{(m)}}(X_i)
$$ 
exists a.s. We thus obtain that a.s.
$$
\mathbb P(|C_\omega(o)|=\infty) = \lim_{n\to\infty} \frac{1}{n} \sum_{i=0}^{n-1} \1_{C_\infty}(X_i) = \sum_{m=1}^k \lim_{n\to\infty} \frac{1}{n} \sum_{i=0}^{n-1} \1_{C_\infty^{(m)}}(X_i).
$$
Let $C_\infty^{(M)}$ be chosen uniformly from those $C_\infty^{(1)},C_\infty^{(2)},\ldots,C_\infty^{(k)}$ with maximal $\alpha(C_\infty^{(m)})$. Then a.s.
\begin{equation}\label{ratio}
\lim_{n\to\infty} \frac{1}{n} \sum_{i=0}^{n-1} \1_{C_\infty^{(M)}}(X_i) \ge \mathbb P(|C_\omega(o)|=\infty)/k.
\end{equation}
Since $C_\infty^{(M)}$ is an invariant random infinite subset of $\mathcal G$, also $\mathbb P(o \in C_\infty^{(M)}) > 0.$ Thus, for $r\ge0$ fixed, conditional on $\{o \in C_\infty^{(M)}\}$, we again have that a.s. 
$$
\lim_{n\to\infty} \frac{1}{n} \sum_{i=0}^{n-1} \1_{\{X_i\leftrightarrow o, d(o,X_i)> r \}} \ge \mathbb P(|C_\omega(o)|=\infty)/k > 0,
$$
which leads to the same contradiction as the one derived after \eqref{equ:FrequTPFBound}.
\end{proof}

We now handle the general invariant case using ergodic decomposition and the following lemma about cluster frequencies from \cite[Lemma 4.2]{LS99}. 

\begin{lemma} \label{lm:frequ}
Let $\mathcal G=(V,E)$ be the Cayley graph of a finitely generated, infinite group $\Gamma$. Then there exists a $\Gamma$-invariant measurable function ${\rm freq}: \{0,1\}^E \to[0,1]$ with the following property. Let $\mathbb P$ be a $\Gamma$-invariant bond percolation with configuration $\omega$ on~$\mathcal G$. Let $o\in V$ and let $P^X_o$ be the law of simple random walk $(X_i)_{i=0}^\infty$ on $\mathcal G$ started in $o$. Then $\P\otimes P^X_o$-a.s.
$$
\lim_{n\to\infty} \frac{1}{n} \sum_{i=0}^{n-1} \mathbf 1_{\{ X_i\in C\}} = {\rm frequ}(C) \quad \text{for every cluster~} C \text{ of } \omega.
$$ 
\end{lemma}

\begin{proof}[Proof of Proposition~\ref{prop:ConsequenceVanishingTPF}] Let 
$$
\mathbb P(\cdot) = \int \mu(\cdot) \mathbb Q( d \mu)
$$
be the ergodic decomposition of $\mathbb P$, i.e.~$\mathbb Q$ is a probability measure on $\Gamma$-invariant ergodic bond percolations $\mu$. 

For a percolation configuration $\omega$, let $N_\infty(\omega)$ be the number of infinite clusters. By ergodicity, for every $\mu$ in the above decomposition, $\mu(N_\infty(\omega)=k)=1$ for some $k\in\{0,1,2,\ldots\}\cup\{\infty\}$. Define
$$
E_k:=\{ \mu : \mu (N_\infty(\omega)=k)=1 \}
$$
and, for $\delta>0$, define
$$
E_{k,\delta}:=\{ \mu : \mu (N_\infty(\omega)=k)=1 \ \mbox{and} \ \mu (|C_\omega(o)|=\infty)\ge \delta \}.
$$

To reach a contradiction, suppose that $\mathbb P$ is not supported on configurations with either only finite clusters or with infinitely many infinite clusters. Then there is $k_0\in\{1,2,\ldots\}$ such that $\mathbb Q(E_{k_0})>0$. Since for every $\mu\in E_{k_0}$ we have that $\mu( |C_\omega(o)|=\infty)>0$, we also have that there exists $\delta_0>0$ such that $\mathbb Q(E_{k_0,\delta_0})>0$. 

Let $X:=(X_i)_{i=0}^\infty$ be independent simple random walk on $\mathcal G$ started at $o$. We now make several uses of the fact that in every invariant bond percolation on $\mathcal G$, every cluster has a well-defined measurable frequency in the sense of Lemma~\ref{lm:frequ}. First, we have seen in the proof of Lemma~\ref{lm:ConsequenceVanishingTPFergodic} (recall \eqref{ratio} above) that every $\mu\in E_{k_0,\delta_0}$ is supported on configurations in which there exists a cluster of frequency at least $\delta_0/k_0$. Since $q:=\mathbb Q(E_{k_0,\delta_0})>0$, it follows that also under $\mathbb P$, with probability at least $q$, there exists a cluster of frequency at least $\delta_0/k_0$. Let $K$ be a uniformly chosen cluster of frequency at least $\delta_0/k_0$ (note that there can be at most $\lceil k_0/\delta_0 \rceil$) if such a cluster exists and let $K$ be empty otherwise. Then $q_0:=\mathbb P(o\in K)>0$ and, conditional on $\{o\in K\}$, a.s.
$$
\lim_{n\to\infty} \frac{1}{n} \sum_{i=0}^{n-1} \1_{\{X_i\leftrightarrow o, d(o,X_i)> r \}} \ge \delta_0/k_0>0.
$$
This leads to the same contradiction as the one derived after \eqref{equ:FrequTPFBound}.\end{proof}

\section{\bf Kazhdan's Property (T) and Applications} \label{kazhdan}

The main theme of this section is the appearance of thresholds. These thresholds are such that invariant bond percolations with expected degree exceeding the threshold do not exhibit decay of the two-point function. Failure to decay at infinity may occur in two (in some sense three) qualitatively distinct ways, which depend on the strength of Kazhdan's property~(T) the group has, i.e.~whether it has actual property (T), relative property (T) or only fails to have the Haagerup property. This section is divided into three subsections corresponding to each of these cases and we refer to Remarks~\ref{rem-CD-LRO}~-~\ref{rem-Haagerup-vs-Kazhdan} and Remark \ref{rem-InBetween} for the comparison. In Section \ref{subsection-Kazhdan}, we provide the threshold for Kazhdan groups and prove the characterization of Kazhdan's property (T) announced in Theorem \ref{maintheorem-Kazhdan}. In Section~\ref{subsection-RelativeKazhdan}, we extend this result to characterize relative property (T) and provide the application of this threshold to Bernoulli percolation. In Section \ref{subsection-WeakKazhdan}, we study the threshold appearing for groups which fail to have the Haagerup property and provide another application to Bernoulli bond percolation, which applies to Cayley graphs satisfying a weak form of spherical symmetry.

\subsection{Kazhdan's Property (T)} \label{subsection-Kazhdan}

We start by recalling the definition of Kazhdan's property (T). If $(\pi,H)$ is a unitary representation of $\Gamma$, $\varphi \in H$ is an {\it invariant vector} if $\pi(s)\varphi=\varphi$ for every $s\in\Gamma$. Let $\eps>0$ and $K\subset\Gamma$ finite, then a unit vector $\varphi \in H$ is called {\it $(\eps,K)$-invariant} if 
$$
\sup_{s \in K} \Vert \pi(s)\varphi - \varphi \Vert < \eps.
$$
The pair $(\eps,K)$ is called a {\it Kazhdan pair} for $\Gamma$ if every unitary representation of $\Gamma$ which has an $(\eps,K)$-invariant vector also has a non-zero invariant vector. The group $\Gamma$ has {\it property (T)} if there exist $\eps>0$ and a finite set $K\subset\Gamma$ such that $(\eps,K)$ is a Kazhdan pair. The goal of this subsection is to prove the following characterization of property (T) in terms of invariant percolations.

\begin{theorem}[\bf Group-Invariant Percolation and Kazhdan's Property (T)] \label{theorem-KazhdanGroups}
Let $\skrig$ be the Cayley graph of a finitely generated group $\Gamma$. Then $\Gamma$ has property (T) if and only if there exists a threshold $\alpha^*<1$ such that for every $\Gamma$-invariant bond percolation $\P$ on $\skrig$, $\E[\deg_\omega(o)]>\alpha^*\deg(o)$ implies that $\tau$ is bounded away from zero, i.e.\
$$
\inf_{g,h\in\Gamma} \tau(g,h)>0.
$$
In fact, the threshold $\alpha^\star$ satisfies 
$$
\alpha^* \leq \inf_{(\eps,K)} \bigg[1-\eps^2/\big(4\deg(o)d_{K,\skrig}\big)\bigg],
$$
where the infimum runs over all Kazhdan pairs $(\eps,K)$ for $\Gamma$ and $d_{K,\skrig} \coloneqq \max \big\{ d(o,g) \colon g \in K\big\}$.

\end{theorem}


\vspace{1mm}

\begin{remark}[\bf Connectivity Decay vs.~Long-Range Order] \label{rem-CD-LRO} 
A two-point function $\tau$ may fail to vanish at infinity in two qualitatively distinct ways. Following Lyons and Schramm \cite{LS99}, we say that $\tau$ exhibits {\em connectivity decay}, if
    \begin{equation}
    \inf_{g,h \in \Gamma} \tau(g,h) = 0,
    \end{equation}
    and otherwise say that $\tau$ exhibits {\em long-range order}, i.e.
    \begin{equation}
    \inf_{g,h \in \Gamma} \tau(g,h) > 0.
    \end{equation}
    If $\tau$ fails to vanish at infinity, connectivity decay is a particularly interesting case because it implies that $\tau(o,g)$ decays to zero in at least one direction $g\to\infty$, while it stays bounded away from zero in some other direction. This is not exceptional behavior. In fact, it already appears in Bernoulli bond percolation \cite[Remark 1.3]{LS99}: take the canonical Cayley graph of the free product $\Z^2*\Z_2$, i.e.~the Cayley graph corresponding to the presentation $\langle a,b,c \mid ab=ba,c=c^{-1}\rangle$. 
    For $p\in(1/2,1)$, we claim that $\tau_p$ does not vanish at infinity and exhibits connectivity decay. Indeed, since removal of any edge corresponding to the generator $c$ breaks the graph into two connected components, $\tau_p$ exhibits connectivity decay. On the other hand, there is a unique infinite cluster in every $\Z^2$ fiber (the set of pairs $g,h$ such that $g^{-1}h$ is in the group generated by $a$ and $b$), and thus $\tau_p$ does not vanish at infinity. 
\end{remark}

\begin{remark}[\bf Haagerup vs.~Kazhdan's Property] \label{rem-Haagerup-vs-Kazhdan}
The fact that a two-point function which does not vanish at infinity may exhibit either connectivity decay or long-range order as described in Remark \ref{rem-CD-LRO} corresponds to the fact that Theorem \ref{maintheorem-Haagerup} is not the negation of Theorem \ref{maintheorem-Kazhdan}, nor vice-versa. 
 Long-range order is the strongest form of non-vanishing at infinity, which quite remarkably is exhibited by invariant bond percolations with large marginals on Kazhdan groups irrespective of the specific model. Examples of groups which satisfy neither the Haagerup nor Kazhdan's property, together with an explanation of how this manifests in the distinction between connectivity decay and long-range order (as a condition on all models) are given in Remark \ref{rem-InBetween}.\end{remark}

\proof[{\bf Proof of Theorem \ref{theorem-KazhdanGroups}}.] 
The proof naturally splits into two parts, namely showing that the existence of a long-range order threshold is both necessary and sufficient.

\noindent{\bf ($``\Leftarrow"$ Long-Range Order Threshold is Sufficient for Property (T):)} Let us recall
that the sufficiency part of Theorem \ref{theorem-KazhdanGroups} is new and its proof will rely on our construction used in the proof of Theorem \ref{theorem-GroupsWithTheHaagerupProperty}. More precisely, assume that there exists a threshold $\alpha^*<1$ such that for every $\Gamma$-invariant bond percolation $\P$ on $\skrig$ with $\E[\deg_\omega(o)]>\alpha^*\deg(o)$, the two-point function is bounded away from zero. Recall that if every conditionally negative definite function $\psi$ on $\Gamma$ is bounded, then $\Gamma$ has property~(T) (this is classical, see e.g.~\cite{CCJJV01}). Let $\psi$ be a conditionally negative definite function. By Lemma \ref{lemma3.5},there exis $\Gamma$-invariant bond percolations $\bP_p$ on $\skrig$, for $p\in(0,1)$, such that 
$$
\lim_{p\to1} \bE_p[\deg_\omega(o)]=\deg(o)
$$
and such that
$$
\tau_p(o,g) \leq \e^{-(1-p)\sqrt{\psi(g)}} \quad \mbox{for every} \, \, g\in\Gamma, p\in(0,1).
$$
Then, for a fixed $p$ sufficiently close to $1$, the two-point function $\tau_p$ of $\bP_p$ is bounded away from zero. Hence $\sqrt\psi$, and thus $\psi$, must be bounded. It follows that $\Gamma$ has property (T).  \eproof

\noindent{\bf ($``\Rightarrow"$ Long-Range Order Threshold is Necessary for Property (T):)} As mentioned before, the part of Theorem \ref{theorem-KazhdanGroups} concerning the qualitative consequence of property (T) is due to Lyons and Schramm \cite[Remark 6.7]{LS99} and was applied there to show that~$p_u<1$, and to prove non-uniqueness at~$p_u$, for Cayley graphs of Kazhdan groups \cite[Corollary 6.6]{LS99}. There it was noted that this qualitative implication may be deduced from a classical fact that if $\Gamma$ has property~(T), then every sequence of normalized, positive definite functions on $\Gamma$ which converges pointwise to $1$, converges to $1$ uniformly on~$\Gamma$ (see e.g.~\cite{CCJJV01}). In \cite[Theorem 5.6]{IKT09}, it was shown that for a group with property (T) and a Kazhdan pair $(\eps, K)$ with $K$ a generating set, every invariant, insertion-tolerant, ergodic bond percolation $\P$ on the Cayley graph with respect to $K$, which satisfies $\inf_{e} \P\big(e\in\omega\big)\geq 1- \frac{\eps^2}2$, has a unique infinite cluster. Based on this method, and for the convenience of the reader, we will provide a quantitative proof of the consequence of property (T), i.e.~of the assertion that
\begin{equation} \label{eq-EstimateThreshold}
\beta=\inf_{(\eps,K)} \bigg[1-\eps^2/\big(4\deg(o)d_{K,\skrig}\big)\bigg],
\end{equation}
with the infimum running over all Kazhdan pairs $(\eps,K)$ for $\Gamma$ and $d_{K,\skrig} \coloneqq \max \{ d(o,g) \colon g \in K\}$, provides a threshold for long-range order. We point out that there is a minor difference to the argument from \cite{IKT09}, namely that we will consider an a-priori fixed Cayley graph of $\Gamma$ and not one built using the set~$K$, thereby getting slightly different constants.

We recall the following lemma, which follows from \cite{DR95}.

\begin{lemma}[{\cite[Proposition 5.3]{IKT09}}] \label{lemma-DR95}
\label{lemma: Kazhdan lower bound} Let $\Gamma$ be a group with property (T), let $(\eps,K)$ be a Kazhdan pair for $\Gamma$, let $\delta>0$ and let $\varphi$ be a normalized, positive definite function. Then 
$$
\min_{g\in K} \, {\rm Re} (\varphi(g))\geq1-\frac{\delta^2\eps^2}{2},
$$
where ${\rm Re}(z)$ for $z\in\C$ denotes the real part, implies that
$$
\inf_{g\in\Gamma} \, {\rm Re}(\varphi(g))\geq 1-2\delta^2.
$$
\end{lemma}

Now assume that $\Gamma$ has property (T). Let $(\eps,K)$ be any Kazhdan pair for~$\Gamma$ and recall that $d_{K,\skrig} \coloneqq \max \{ d(o,g) \colon g \in K \}.$ Let $\P$ be a $\Gamma$-invariant bond percolation on $\skrig=(V,E)$ with two-point function $\tau$ and define 
$$
m\coloneqq \inf_{e\in E} \P[e\in\omega] = \min_{o\sim g} \tau(o,g).
$$
For every $g\in\Gamma$, we then have the lower bound $\tau(o,g)\geq1-d(o,g)(1-m)$: to see this, let $(o,g)$ denote a shortest path of edges joining $o$ to $g$, then
\begin{flalign*}
\tau(o,g) & =\P\big[ o \leftrightarrow g\big] \geq \P\big[ (o,g) \subset \omega\big] = \P\bigl[ e \in \omega \text{ for all } e \in (o,g) \bigr] \\
& \geq 1-\sum_{e\in(o,g)} \P[e \notin \omega] \\
& \geq 1-d(o,g)(1-m).
\end{flalign*}
In particular, it follows that $\min_{g\in K} \tau(o,g) \geq 1-d_{K,\skrig}(1-m).$ Hence 
\begin{equation} \label{equ-ProofTMarginals}
m>1-\delta^2\eps^2/(2d_{K,\skrig})
\end{equation}
for some $\delta>0$ implies that
$$
\min_{g\in K} \tau(o,g) >1-\frac{\delta^2\eps^2}{2}.
$$
Rewriting 
$$
\E[\deg_\omega(o)] = \sum_{o \sim g} \tau(o,g) \leq m + (\deg(o)-1)
$$
shows that
\begin{equation} \label{equ-ProofTDegree}
\E[\deg_\omega(o)] > \biggl( 1 - \frac{\delta^2\eps^2}{2\deg(o)d_{K,\skrig} } \biggr) \deg(o) = \deg(o)-\frac{\delta^2\eps^2}{2d_{K,\skrig}},
\end{equation}
implies \eqref{equ-ProofTMarginals}. Thus every $\Gamma$-invariant bond percolation satisfying \eqref{equ-ProofTDegree}, also satisfies \eqref{equ-ProofTMarginals}. Now suppose that \eqref{equ-ProofTDegree} holds and $\delta \in (0,1/\sqrt{2})$, then Lemma \ref{lemma-DR95} implies
$$
\inf_{g,h} \tau(g,h) = \inf_{g\in\Gamma} \tau(o,g) > 1-2\delta^2 > 0.
$$
This shows that $\alpha^*$ can be taken to be
$$
1 - \frac{\delta^2\eps^2}{2\deg(o)d_{K,\skrig} } <1,
$$
as claimed. Letting $\delta\to 1/\sqrt{2}$, we in particular see that the optimal threshold is less than or equal to
\begin{equation} \label{eq-alphastar}
\inf \biggl\{ 1-\frac{\eps^2}{4\deg(o)d_{K,\skrig}} \colon (\eps,K) \, \, \mbox{is a Kazhdan pair for} \, \, \Gamma \biggr\},
\end{equation}
i.e.~the threshold $\beta$ defined in \eqref{eq-EstimateThreshold}. This concludes the proof of the direct implication. \eproof

\begin{remark}[\bf Marginals and Expected Degree] \label{remark-ThresholdKazhdan}
We have presented a threshold for the expected degree to keep the exposition consistent. As seen in the proof, such a threshold is essentially equivalent to a threshold for marginals. The lower bound on the expected degree may e.g.~be replaced by the assumption that each individual edge is present with probability at least $1-(\delta^2\eps^2)/(2d_{K,\skrig})$. If $S$ denotes the generating set used to define $\skrig$, then $d_{S,\skrig}=1$ and it is a standard fact that $(\eps_S,S)$ is a Kazhdan pair for $\Gamma$ for suitable $\eps_S>0$. Letting $\delta\to 1/\sqrt{2}$, it follows that every $\Gamma$-invariant bond percolation which contains each individual edge with probability at least $1-\eps_S^2/4$, has a two-point function bounded away from zero. This recovers the bound in \cite[Corollary 5.4 (ii)]{IKT09}. Moreover, it shows that the optimal threshold for the expected degree is less than or equal to $1-\eps_S^2/(4\deg(o))$.  \nopagebreak {\hfill\rule{2mm}{2mm}}
\end{remark}

\subsection{Relative Property (T) and Non-Uniqueness at $p_u$.} \label{subsection-RelativeKazhdan}

In this section, we find an extension of Theorem \ref{theorem-KazhdanGroups} to groups which have relative property~(T) and provide an application to the question of uniqueness in Bernoulli bond percolation. To start off, recall the definition of relative property~(T). Let $\Gamma$ be a finitely generated group and let $H$ be a subgroup. The pair $(\Gamma,H)$ has {\em relative property~(T)} if every unitary representation of $\Gamma$ on a Hilbert space which almost has invariant vectors, has a non-zero $H$-invariant vector, see \cite[Definition 1.4.3]{BdlHV08}. The notion of relative property~(T) for a pair was considered implicitly in Kazhdan's work \cite{K67} and explicitly by Margulis \cite{M82}. A well-known example is the pair $(\Z^2\rtimes{\rm SL}_2(\Z),\Z^2)$, which has relative property~(T) but $\Z^2\rtimes{\rm SL}_2(\Z)$ does not have property~(T). Characterizations of property~(T) can often be extended to formulations for relative property~(T), see e.g.\ \cite{BR95,J05,C06}. In this vein, using our construction from Theorem \ref{theorem-GeneralConstruction} we provide the following extension of Theorem \ref{theorem-KazhdanGroups} to the situation of relative property~(T) for the purpose of then providing an application of that ``long-range order'' threshold to Bernoulli percolation in Theorem \ref{theorem-UniquenessRelativeKazhdan}.

\begin{theorem}[\bf Group-Invariant Percolation and Relative Property (T)] \label{theorem-RelativeKazhdanGroups} 
Let $\Gamma$ be a finitely generated group, let $H$ be a subgroup and let $\skrig$ be a Cayley graph of $\Gamma$. Then the pair $(\Gamma,H)$ has relative property (T) if and only if there exists a threshold $\alpha^*<1$ such that for every $\Gamma$-invariant bond percolation $\P$ on $\skrig$, $\E[\deg_\omega(o)]>\alpha^*\deg(o)$ implies that $\tau$ is bounded away from zero on each left $H$-coset, i.e.\ for every $g\in\Gamma$
$$
\inf_{h\in H} \tau(g,gh) > 0.
$$
\end{theorem}

\proof Suppose that there exists a threshold $\alpha^*<1$ such that for every $\Gamma$-invariant bond percolation $\P$ on $\skrig$, $\E[\deg_\omega(o)]>\alpha^*\deg(o)$ implies that $\tau$ is bounded away from zero on each left $H$-coset. We recall that the methods of \cite{RS98,CMV04} yield the following characterization of relative property (T) in terms of invariant actions on spaces with measured walls, as observed in \cite[Theorem 3.3.6]{J07}. Let $\Gamma$ be a countable group and let $H$ be a subgroup. Then the pair $(\Gamma,H)$ has relative property (T) if and only if for every invariant action of $\Gamma$ on a space with measured walls $(X,\skriw,\skrib,\mu)$, the $H$-orbits are bounded in the sense that for every $x_0\in X$
$$
\sup_{h\in H} w(x_0,hx_0) = \sup_{h\in H} \mu(\skriw(x_0,hx_0)) < \infty.
$$
We may thus show that the pair $(\Gamma,H)$ has relative property (T) by verifying that every invariant action of $\Gamma$ on a space with measured walls has bounded $H$-orbits. To that end, suppose that $\Gamma$ has an invariant action on a space with measured walls $(X,\skriw,\skrib,\mu)$ and let $x_0\in X$. By Theorem \ref{theorem-GeneralConstruction}, for suitable $p\in(0,1)$, there exists a $\Gamma$-invariant bond percolation $\P$ on $\skrig$ such that $\E[\deg_\omega(o)]>\alpha^* \deg(o)$ and 
\begin{equation} \label{eq-Rel1}
\tau(o,g) \leq \exp \bigl( -(1-p)w(x_0,gx_0) \bigr) \quad \mbox{for every} \, \, g \in \Gamma.
\end{equation}
It follows from our assumption that
\begin{equation} \label{eq-Rel2}
\inf_{h\in H} \tau(o,h) > 0.
\end{equation}
Combining \eqref{eq-Rel1} and \eqref{eq-Rel2} shows that
$$
\sup_{h\in H} w(x_0,hx_0)<\infty,
$$
i.e.\ the $H$-orbit of $x_0$ is bounded. Since $x_0\in X$ was arbitrary, all $H$-orbits are bounded and it follows that the pair $(\Gamma,H)$ has relative property (T).

Conversely, suppose there exists a sequence $\P_n$ of $\Gamma$-invariant bond percolations on $\skrig$ such that 
$$
\lim_{n\to\infty} \E_n[\deg_\omega(o)]=\deg(o)
$$
and such that for every $n\in\N$ there exists $g_n\in\Gamma$ with
$$
\inf_{h\in H} \tau_n(g_n,g_nh) = 0.
$$
Then the sequence $\tau_n(o,\cdot)=\tau_n(g_n,g_n \cdot)$ is a sequence of positive definite functions converging pointwise to $1$, but not uniformly on $H$. Since a pair $(\Gamma,H)$ has relative property (T) if and only if every sequence of positive definite functions on $\Gamma$ which converges pointwise to $1$, converges to $1$ uniformly on $H$ (see e.g.~\cite[Theorem 1.1]{C06}), we conclude that the pair $(\Gamma,H)$ does not have relative property (T). \eproof

\begin{remark}[\bf Groups without the Haagerup Property and without Property (T)] \label{rem-InBetween} Returning to the discussion in Remark \ref{rem-Haagerup-vs-Kazhdan}, we now elaborate on the subtle behavior of groups in between groups with the Haagerup and Kazhdan properties: recall that $\Z^2\rtimes{\rm SL}_2(\Z)$ does not have property~(T), but does have relative property (T) with respect to an infinite subgroup. Theorem \ref{theorem-RelativeKazhdanGroups} thus shows that there are groups such that invariant bond percolations with sufficiently large marginals fail to vanish at infinity by being bounded away from zero only in a specific direction. Even more curiously, the existence of groups without the Haagerup property and without property (T) with respect to any infinite subgroup (see e.g.~\cite{C06}) implies that there are groups such that two-point functions of invariant bond percolations with sufficiently large marginals fail to vanish at infinity, but there can be no ''common subgroup-direction'' in which they are bounded away from zero in the following sense: for every infinite subgroup $H$, there exist models with arbitrarily large marginals which are not bounded away from zero on $H$-cosets (Theorem \ref{theorem-KazhdanGroups} and Theorem \ref{theorem-RelativeKazhdanGroups}.
\end{remark}

Recall that for Cayley graphs of groups with property (T), Bernoulli percolation at the uniqueness threshold does not have a unique infinite cluster, see \cite[Corollary 6.6]{LS99}.
The fact that this also holds for groups $\Gamma$ which admit an infinite normal subgroup $H$ such that the pair $(\Gamma,H)$ has relative property~(T) has appeared in \cite{GT16}. Note that this in particular implies $p_u<1$ for these Cayley graphs. We will give alternative proofs of these results using Theorem \ref{theorem-KazhdanGroups} and Theorem \ref{theorem-RelativeKazhdanGroups}. Due to the fact that these proofs only use invariant percolations, we believe them to be of interest and in this vein, 
we will include a proof of the strongest result.

\begin{theorem}[\bf Property (T) and Relative Property (T) imply Non-Uniqueness at $p_u$] \label{theorem-UniquenessRelativeKazhdan} 
Let $\Gamma$ be a finitely generated group and let $H$ be an infinite, normal subgroup such that the pair $(\Gamma,H)$ has relative property~(T). Let $\skrig$ be any Cayley graph  of $\Gamma$. Then $p_u(\skrig)$-Bernoulli bond percolation does not have a unique infinite cluster a.s.

In particular, $p_u(\skrig)$-Bernoulli bond percolation on any Cayley graph of a countable group with property (T) does not have a unique infinite cluster a.s.
\end{theorem}

\proof Let $\omega_{p_u}$ denote the configuration of $p_u\coloneqq p_u(\skrig)$-Bernoulli bond percolation on $\skrig$. Suppose that there is a unique infinite cluster $C_\infty$ in $\omega_{p_u}$ a.s. We use a construction of Benjamini, Lyons, Peres and Schramm \cite{BLPS99b}: for each $g\in\Gamma$ let $C(g)$ denote the set of vertices in $C_\infty$ which are closest to $g$ in the graph distance. Let $\gamma_\eps$ denote the configuration of an $\eps$-Bernoulli bond percolation on $\skrig$ which is independent of $\omega_{p_u}$. Now define a bond percolation $\xi_\eps$ as follows: let the edge $[g,h]$ be contained in $\xi_\eps$ if and only if $\max \{ d(g,C_\infty),d(h,C_\infty) \}<1/\eps$ holds and additionally $C(g)$ and $C(h)$ are contained in the same connected component of $\omega \setminus \gamma_\eps$. It follows from the bounded convergence theorem that $\xi_{\eps}$ has marginals arbitrarily close to $1$ when $\eps>0$ is sufficiently small, see \cite[Section 3]{BLPS99b}. Now Theorem~\ref{theorem-RelativeKazhdanGroups} implies that for some fixed, sufficiently small $\eps>0$, we have that $\tau_{\xi_\eps}$ is bounded away from zero on (left) $H$-cosets. We claim that there exists a constant $c>0$ such that 
\begin{equation}\label{eq-ProofRelKazhdan1}
c \, \tau_{\xi_\eps}(g,h) \leq \tau_{\omega_{p_u}\setminus\gamma_\eps}(g,h) \quad \mbox{for all} \, \, g,h\in\Gamma.
\end{equation}
To see this, observe that $\omega_{p_u}\setminus\gamma_\eps$ is a Bernoulli bond percolation with parameter $(1-\eps)p_u>0$. In particular
$$
c_1 \coloneqq \inf_{u \in \partial B_{1/\eps}(o)} \tau_{\omega_{p_u}\setminus\gamma_\eps} (o,u) > 0.
$$
By transitivity, for every $g\in\Gamma$,
$$
c_1=\inf_{u \in \partial B_{1/\eps}(g)} \tau_{\omega_{p_u}\setminus\gamma_\eps} (g,u).
$$
Let $g,h\in\Gamma$, $u\in\partial B_{1/\eps}(g)$ and $v \in \partial B_{1/\eps}(h)$. Applying the Harris-FKG inequality, we see that
\begin{align}
\tau_{\omega_{p_u}\setminus\gamma_\eps}(g,h) & \geq \P \Bigl( g \xleftrightarrow{\omega_{p_u}\setminus\gamma_\eps}u, u \xleftrightarrow{\omega_{p_u}\setminus\gamma_\eps} v, v \xleftrightarrow{\omega_{p_u}\setminus\gamma_\eps} h \Bigr) \nonumber \\
& \geq c_1^2 \tau_{\omega_{p_u}\setminus\gamma_\eps}(u,v). \label{eq-ProofRelKazhdan2}
\end{align}
Now by the definition of $\xi_\eps$ we have that
$$
\Bigl\{ g \xleftrightarrow{\xi_\eps} h \Bigr\} \subset \Bigl\{ B_{1/\eps}(g) \xleftrightarrow{\omega_{p_u}\setminus\gamma_\eps} B_{1/\eps}(h) \Bigr\}.
$$
Applying monotonicity, the union bound and inequality (\ref{eq-ProofRelKazhdan2}), we obtain that
\begin{align*}
\tau_{\xi_\eps}(g,h) & \leq \P \Bigl[  B_{1/\eps}(g) \xleftrightarrow{\omega_{p_u}\setminus\gamma_\eps} B_{1/\eps}(h) \Bigr] \\
& \leq \sum_{\substack{u \in \partial B_{1/\eps}(g) \\ v \in \partial B_{1/\eps}(h)}} \tau_{\omega_{p_u}\setminus\gamma_\eps}(u,v) \\
& \leq \frac{| \partial B_{1/\eps}(o) |^2}{c_1^2} \tau_{\omega_{p_u}\setminus \gamma_\eps} (g,h).
\end{align*}
This proves the inequality (\ref{eq-ProofRelKazhdan1}) with the constant $c=(|\partial B_{1/\eps}(o)|/c_1)^{-2}$. In particular, the fact that $\tau_{\xi_\eps}$ is bounded away from zero on $H$-cosets implies that the same holds for $\tau_{\omega_{p_u}\setminus\gamma_\eps}$.

To finish the proof, we now closely follow the argument of Hutchcroft and Pete in \cite[Section 3]{HP20}. The $H$-cosets $gH$ form a partition of the vertices of $\mathcal G$ and it will be useful to measure cluster frequencies inside cosets as follows. Enumerate the elements of $H$ as $\{h_1,h_2,\ldots\}$, let $(Z_i)_{i=1}^\infty$ be an i.i.d.~sequence of $H$-valued random variables with $\mathbb P(Z_i=h_k)=2^{-k}$ and write $\mathbf P_\gamma$ for the law of the random walk $(X_n)_{n=0}^\infty$ defined by $X_{n+1}=X_n Z_{n+1}$ for $n\ge 1$ and $X_0:=\gamma\in\Gamma$. Then an analogue of Lemma~\ref{lm:frequ} is that for every $gH$, there exists an $H$-invariant measurable function ${\rm freq}_{gH}: \{0,1\}^E \to[0,1]$ such that if $\mu$ is a $\Gamma$-invariant bond percolation with configuration $\omega$ on~$\mathcal G$, then
$$
\lim_{n\to\infty} \frac{1}{n} \sum_{i=0}^{n-1} \mathbf 1_{\{ X_i\in C\}} = {\rm freq}_{gH}(C) \quad \text{for every cluster~} C \text{ of } \omega
$$ 
 $\mu\otimes\mathbf P_\gamma$-a.s.~for every $\gamma\in gH$. Now recall that $\tau_{\omega_{p_u}\setminus\gamma_\varepsilon}$ is bounded away from zero on $H$-cosets. By Fubini's theorem and dominated convergence it follows that for every $gH$ there is a positive probability that there exists an $\omega_{p_u}\setminus\gamma_\varepsilon$-cluster $C_{gH}$ with ${\rm freq}_{gH}(C_{gH})>0$. Since $H$ is infinite, ergodicity implies that such a cluster exists a.s. For every $gH$, choose an $\omega_{p_u}\setminus\gamma_\varepsilon$-cluster $C_{gH}'$ uniformly at random from the finitely many clusters maximizing ${\rm freq}_{gH}$. Since $H$ is an infinite normal subgroup, it can be checked that cosets sharing a neighbor in $\mathcal G$ are
connected in $\mathcal G$ by infinitely many edges. Hence {\em sprinkling} will almost surely connect all clusters $C_{gH}'$: for every $\varepsilon'>0$, adding an independent $\varepsilon'$-Bernoulli bond percolation to $\omega_{p_u}\setminus\gamma_\varepsilon$ almost surely connects the clusters $C_{gH}'$ in neighboring $H$-cosets and hence all clusters $C_{gH}'$. Now deleting all clusters of the resulting percolation configuration other than the unique cluster containing $\bigcup C_{gH}'$, we obtain a $\Gamma$-invariant bond percolation with a unique infinite cluster. 

Up to this point, the above proof is identical to the one in \cite{HP20} except for the step deducing the existence of a cluster of positive frequency from a two-point function lower bound.$^4$\footnote{$^4$We may also point out that the argument using positive definiteness can also be used to prove~\cite[Corollary~2.7]{HP20} (and its version for relative property~(T)) which is the part of the proof of the main result of \cite{HP20} that uses the characterization of property~(T) due to Glasner and Weiss. Indeed, in the notation of \cite{HP20}, the measures $\mu_i$ weak$^*$-converge to the measure $p\delta_V+(1-p)\delta_\varnothing$, which implies that $\tau_i(g,h):=\mu_i(g \leftrightarrow h) \to p$ for all $g,h\in \Gamma$ (to see this, note that $\tau_i(g,h)\le \mu_i(g\in \omega)=p$ and, for a fixed path $\rho_{g,h}$ from $g$ to $h$, $\tau_i(g,h)\ge \mu_i( \rho_{g,h}\subset\omega) \to \mu(\rho_{g,h}\subset\omega) = p$). Hence the positive definite functions $\varphi_i:=\tau_i/p$ converge pointwise to $1$ and property~(T) implies that they converge to $1$ uniformly. This implies that, for $i$ sufficiently large, $\tau_i$ has long-range order and hence $\mu_i$ has a cluster of positive frequency with positive probability.}

With these preparations, we may conclude the proof as follows. By ergodic decomposition and the proof of Lemma~\ref{lm:ConsequenceVanishingTPFergodic}, the unique infinite cluster constructed above has a positive frequency for independent random walk on $\mathcal G$ almost surely. On the other hand, adding an independent $\varepsilon'$-Bernoulli bond percolation to $\omega_{p_u}\setminus\gamma_\varepsilon$ yields a percolation $\eta$ whose law is that of $\bigl((1-\varepsilon)p_u+(1-(1-\varepsilon)p_u)\varepsilon'\bigr)$-Bernoulli bond percolation. Since $\eta$ contains the unique infinite cluster constructed above, we obtain that $\eta$ has a cluster of positive frequency. The indistinguishability theorem of Lyons and Schramm~\cite[Theorem~1.1]{LS99} thus implies that $\eta$ has a uinque infinite cluster. For $\varepsilon'$ sufficiently small, this contradicts the definition of $p_u$.

We also note that in the special case of Cayley graphs of countable groups with property (T), the above argument can be shortened because the ''sprinkling'' step is not required. Indeed, if $\omega_{p_u}$ has a unique infinite cluster, then we may apply Theorem \ref{theorem-KazhdanGroups} instead of Theorem~\ref{theorem-RelativeKazhdanGroups} to directly conclude that $\tau_{\xi_\eps}$ is bounded away from zero on all of $\Gamma$ for $\eps$ sufficiently small. Since \eqref{eq-ProofRelKazhdan1} remains valid, {\color{blue}\cite[Theorem 4.1]{LS99}} 
shows that $\omega_{p_u\setminus\gamma_\eps}$ has a unique infinite cluster for $\eps$ sufficiently small, leading to the same contradicition as before. This proof is new also in this setting.
 \eproof

The following result was pointed out to us by an anonymous referee who also pointed out that it follows from \cite[Theorem 1.1, Theorem 1.3]{GJM25} combined with \cite[Theorem 6.16]{AL07}. The result is similar in spirit to Theorem~\ref{theorem-Spanning} and we have thus included an alternative proof based on similar ideas and the notion of cluster frequencies used in Section \ref{subsec clusterfreq}.

\begin{prop}[{\bf Robustness in Kazhdan Groups}] \label{prop:Robust} Let $\mathcal G$ be the Cayley graph of a finitely generated, infinite group $\Gamma$. Suppose that $\Gamma$ has property~(T). Then every $\Gamma$-invariant random nonempty connected subgraph $\omega$ of $\mathcal G$ has $p_u(\omega)<1$ a.s.
\end{prop}
\begin{proof} Without loss of generality, assume that $\omega$ is ergodic. In particular, $p_u(\omega)$ is a constant a.s. 

To reach a contradiction, suppose that $p_u(\omega)=1$ a.s. By augmenting the configuration as in the proof of Theorem~\ref{theorem-Spanning}, we obtain a $\Gamma$-invariant random spanning subgraph $\xi$ with $p_u(\xi)=1$ a.s. Let $p<1$ and let $\xi[p]$ denote the configuration of Bernoulli bond percolation on $\xi$. Then 
$$
\tau_p(g,h) := \mathbb P\big( g \overset{\xi[p]}{\longleftrightarrow} h \big) 
$$
is a positive definite function, and since $\xi$ is connected and spanning, $\lim_{p\to1} \tau_p(g,h) = 1$. Positive definiteness and property~(T) then imply that $\tau_p(\cdot,\cdot)\to1$ uniformly as $p\to 1$. In particular, there exists $p_0<1$ and $\delta>0$ such that $\tau_{p_0}(g,h)\geq \delta>0$ for all $g,h$. This implies that there is a positive probability that there exists an infinite cluster $C_{p_0}$ of $\xi[p_0]$ of frequency $\mathrm{frequ}(C_{p_0})\geq \delta>0$ in the sense of Lemma~\ref{lm:frequ}. Since $\xi$ is $\Gamma$-invariant and connected, the random rooted graph $(\xi,o)$ is a unimodular random graph, see e.g.~\cite{AL07}. The latter property then implies that for every $p<1=p_u(\xi)$ there are either only finite clusters or infinitely many infinite clusters in $\xi[p]$. It can be checked as in~\cite{LS99} that $\xi[p]$ has indistinguishable infinite clusters, implying that infinite clusters have the same frequency, which contradicts the fact that $\xi[p_0]$ has a cluster of positive frequency. Alternatively, the proof of uniqueness monotonicity~\cite{HP99} as given in~\cite[Theorem 5.4]{HJ06} can be adapted to the setting of unimodular random graphs to show that for every $p>p_0$ there exists a unique infinite cluster, which leads to the same contradiction. 
\end{proof}

\subsection{Weak Kazhdan Property and a Quantitative Converse to Theorem \ref{maintheorem-Haagerup}.} \label{subsection-WeakKazhdan}

In this subsection, we consider groups which do not have the Haagerup property. Following Glasner \cite[Definition 13.16]{G03}, we call such groups {\em weakly Kazhdan} or say they have {\em property $(T_w)$}. Consider now a Cayley graph $\skrig$ of a group $\Gamma$ with property $(T_w)$. Then Theorem \ref{theorem-GroupsWithTheHaagerupProperty} implies the existence of a threshold $\alpha^*_w \coloneqq \alpha^*_w(\skrig)<1$ such that whenever $\P$ is a $\Gamma$-invariant bond percolation with $\E[\deg_\omega(o)]>\alpha^*_w \deg(o)$, then $\tau(o,\cdot)$ does not vanish at infinity. We now give an upper bound on $\alpha^*_w$ in terms of another family of parameters associated with weakly Kazhdan groups. Since $\Gamma$ does not have the Haagerup property, there exist $\eps>0$ and $K\subset\Gamma$ finite with the following property:
\begin{equation}\label{eq-WeakKazhdan}
\begin{aligned}
&\text{every normalized positive definite function} \ \varphi \ \text{on} \ \Gamma  \ \text{satisfying}  \\
&\qquad \min_{g\in K} \varphi(g) \geq 1-\eps \ \text{does not vanish at infinity}.
\end{aligned}
\end{equation}
 Define the parameters
$$
\begin{aligned}
&d_{K,\skrig} \coloneqq \max \big\{ d(o,g) \colon g \in K\big\}, \\
&\alpha_{\eps,K,\skrig}\coloneqq 1- \frac{\eps}{\deg(o)d_{K,\skrig}}, \qquad\mbox{and}\\
&\kappa_w(\skrig)\coloneqq \inf\bigg\{\alpha_{\eps,K,\skrig} \colon \mbox{the pair $(\eps,K)$ satisfies \eqref{eq-WeakKazhdan}}\bigg\}.
\end{aligned}
$$

\vspace{1mm}

\begin{prop}[\bf Estimate of the Threshold for Weak Kazhdan Groups] \label{prop-WeakKazhdanThreshold}
Let $\skrig$ be the Cayley graph of a finitely generated group $\Gamma$ which has property $(T_w)$. Then $\alpha^*_w(\skrig) \leq \kappa_w(\skrig)$.
\end{prop}

\begin{proof} Let $\eps>0$ and $K\subset\Gamma$ finite such that every normalized positive definite function $\varphi \colon \Gamma \to [0,1]$ which satisfies
\begin{equation} \label{eq-WeakKazhdan1}
\min_{g\in K} \varphi(g) \geq 1-\eps,
\end{equation}
does not vanish at infinity. Let $\P$ be a $\Gamma$-invariant bond percolation on $\skrig=(V,E)$. 
Recall from the computations below Lemma~\ref{lemma: Kazhdan lower bound} that $\min_{g\in K} \varphi(g) \ge 1- d_{K,\mathcal G}(1-m)$, where $m={\rm min}_{o\sim g} \tau(o,g)$, and the fact that
\begin{equation} \label{equ:WeakKazhdanDegree}
\E[\deg_\omega(o)] > \alpha_{\eps,K,\skrig} \deg(o) = \Big( 1- \frac{\eps}{\deg(o)d_{K,\skrig}} \Big) \deg(o)
\end{equation}
implies that $m>1-\varepsilon/d_{K,\mathcal G}$, cf.~\eqref{equ-ProofTDegree}. Combining these two shows that \eqref{equ:WeakKazhdanDegree} implies that $\tau$ satisfies \eqref{eq-WeakKazhdan1} and hence does not vanish at infinity. The claim follows.
\end{proof}

\begin{remark}[\bf Marginals and Expected Degree] \label{remark-Marginals}
The lower bound on the expected degree may again be replaced by the assumption that each individual edge is present with probability strictly above $\inf_{(\eps,K)} 1-\eps/d_{K,\skrig}$, where the infimum is over all pairs $(\eps,K)$ satisfying~(\ref{eq-WeakKazhdan}).  \nopagebreak {\hfill\rule{2mm}{2mm}}
\end{remark}

\begin{remark}[\bf Threshold for Site Percolation]
In the setting of site percolation, Corollary \ref{cor-SitePercolation} implies the existence of a possibly different threshold $\alpha^*_{w,{\rm site}}\coloneqq \alpha^*_{w,{\rm site}}(\skrig)<1$ such that whenever $\P$ is a $\Gamma$-invariant site percolation satisfying $\P[o\in\omega]>\alpha^*_{w,{\rm site}}$, then $\tau(o,\cdot)$ does not vanish at infinity. It follows from essentially the same argument as in above proof of Proposition \ref{prop-WeakKazhdanThreshold} that $\alpha^*_{w,{\rm site}}\leq \inf_{(\eps,K)} 1-\eps/d_{K,\skrig}$ with the infimum taken over all pairs $(\eps,K)$ satisfying (\ref{eq-WeakKazhdan}).  \nopagebreak {\hfill\rule{2mm}{2mm}}
\end{remark}

We point out that our estimates of these thresholds (i.e., upper bounds away from $1$) follow from similar arguments in the cases of site and bond percolation and are presumably not sharp. Sharp thresholds for the existence of infinite clusters in automorphism-invariant bond, respectively site, percolations on the $d$-regular tree, $d\ge3$, are known from the pioneering work \cite{H97} and are $2/d$ and $d/(2d-2)$ respectively (in fact, there it is also shown that the critical cases where marginals are equal to these thresholds produce infinite clusters). Note that $d/(2d-2)>2/d$ for $d\ge3$. 

To the best of the authors' knowledge it remains open whether the thresholds for infinite clusters in non-amenable groups~\cite{BLPS99}, which coincide with the ones from \cite{H97} for trees, are sharp in general.

As pointed out by Lyons \cite{L00}, the distinction between connectivity decay and long-range order also reveals a lack of spherical symmetry in Cayley graphs (where a Cayley graph is called {\em spherically symmetric} if for all vertices with $d(o,g)=d(o,h)$, there exists an automorphism which fixes $o$ and maps $g$ to $h$). As a final application of the results in this section, we show that a (formally) weaker notion of spherical symmetry suffices to guarantee non-uniqueness at $p_u$ (and $p_u<1$) for weak Kazhdan groups.

\begin{cor} \label{cor-PuWeakKazhdan}
Let $\skrig$ be the Cayley graph of a finitely generated group $\Gamma$ which has property $(T_w)$. Suppose that $\skrig$ satisfies the following spherical symmetry assumption: 
\begin{itemize}
\item[{\rm(Sym)}] For every $p\in(0,1)$, the two-point function $\tau_p(\cdot,\cdot)$ of $p$-Bernoulli bond percolation vanishes at infinity if and only if it exhibits connectivity decay. 
\end{itemize}
Then $p_u(\skrig) <1$. Moreover, $p_u(\skrig)$-Bernoulli bond percolation does not have a unique infinite cluster.
\end{cor}

\proof The assertion that $p_u<1$ follows directly from the threshold. Indeed, Proposition \ref{prop-WeakKazhdanThreshold} implies that $\tau_p(o,\cdot)$ does not vanish at infinity for $p$ sufficiently close to $1$. Then (Sym) implies that it exhibits long-range order, i.e.\ is bounded away from zero. By \cite[Theorem 4.1]{LS99}, this implies that the corresponding $p$-Bernoulli bond percolation has a unique infinite cluster.  The assertion that there is not a unique infinite cluster at the uniqueness threshold follows from the same argument as employed in the first part of the proof of Theorem \ref{theorem-UniquenessRelativeKazhdan}, replacing the use of Theorem \ref{theorem-RelativeKazhdanGroups} with Proposition \ref{prop-WeakKazhdanThreshold}, and then upgrading non-vanishing at infinity to long-range order through (Sym).\eproof

Note also that in the setting of Corollary \ref{cor-PuWeakKazhdan}, Remark \ref{remark-Marginals} implies that $p_u(\skrig)\leq\inf_{(\eps,K)} 1-\eps/d_{K,\skrig}$ with the infimum taken over all pairs $(\eps,K)$ satisfying (\ref{eq-WeakKazhdan}).

\section{\bf Applications to Concrete Examples of Haagerup Groups}\label{sec 5}

In this section we will prove the results outlined in Section \ref{subsection-intro5}. In particular, Theorem \ref{theorem4}, Theorem~\ref{theorem5} and Proposition \ref{prop1} will be proved in Section \ref{lamplighters}, Section \ref{hyperbolic} and Section \ref{amenable} respectively.

\subsection{Lamplighters over Free Groups: Proof of Theorem \ref{theorem4}.} \label{lamplighters}

Here we will prove the following result, stated earlier in Theorem \ref{theorem4}:

\begin{theorem}[\bf Large Marginals and Exponential Decay] \label{theorem-Lamplighter} 
Let $H$ be a finite group, let $r\in\N$ and let $\skrig$ be any Cayley graph of $\Gamma=H\wr\F_r$. Then there exists a constant $C>0$ such that for every $\alpha<1$, there exists a $\Gamma$-invariant bond percolation $\P$ on $\skrig$ with $\E[\deg_\omega(o)] > \alpha \deg(o)$ and with
$$
\e^{-\beta |g|} \leq \tau(o,g) \leq \e^{-\gamma |g|} \quad \mbox{for all} \, \, g \in \Gamma,
$$
where $\beta,\gamma>0$ such that $\beta / \gamma \leq C$.
\end{theorem}

\vspace{1mm}

We proceed with the proof of Theorem \ref{theorem-Lamplighter} in the following three steps: we first recall the wreath product construction (see Section \ref{subsection-wreath}) and then recall a way of defining a wall structure on wreath products over free groups due to \cite{CSV08} (see Section \ref{subsection-WreathOverFree}). It was observed in the follow-up article \cite{CSV12} that the corresponding wall-distance can be quantified. In our probabilistically motivated set-up, we give a slightly different, self-contained estimate for the convenience of the reader. We finish with an application of Corollary~\ref{cor-QuantitativeDecay} to obtain the desired result.

\subsubsection{\bf Wreath Products of Finitely Generated Groups.}  \label{subsection-wreath}

Let $H$ and $G$ be any two groups. The (standard, restricted) {\em wreath product} $H\wr G$ is the semidirect product $H^{(G)} \rtimes G$, where 
$$
H^{(G)}=\bigotimes_{x\in G} H
$$ 
denotes the direct sum of copies of $H$ on which $G$ acts by left translation. 
In other words, the elements of $H\wr G$ are pairs $(\varphi,x)$ where $x\in G$ and where $\varphi \colon G \to H$ is a function with $f(y)$ equal to the identity $e_H$ of $H$ for all but finitely many $y\in G$. The group operation is given by
$$
(\varphi_1,x_1)(\varphi_2,x_2)=(\varphi_1 (x_2\varphi_2),x_1x_2),
$$
where $x\varphi(y)=\varphi(x^{-1}y)$ and $(\varphi \psi)(y)=\varphi(y)\psi(y)$. 

It is well-known that if $G$ and $H$ are finitely generated, then $H\wr G$ is finitely generated. In fact, given a finite generating set $S$ of $G$ and a finite generating set $T$ of $H$, the {\em canonical generating set} of $H\wr G$ is defined as follows: for every $t\in T$ let $\varphi_t \in H^{(G)}$ be the function with $\varphi_t(e_G)=t$ and $\varphi_t(x)=e_H$ for all $x\neq e_G$. Moreover, let $e_{H^{(G)}}\colon G\to H$ be the function which is equal to identity everywhere. Then 
$$
\Bigl\{ (e_{H^{(G)}},s) \colon s \in S \Bigr\} \sqcup \Bigl\{ (\varphi_t,e_G) \colon t \in T \Bigr\}
$$
is a finite generating set for $H\wr G$, see e.g.\ \cite{LS99}.

\subsubsection{\bf Wreath Products over Free Groups.}  \label{subsection-WreathOverFree}

In the following we restrict our attention to one particular example, the wreath product $\Gamma = \Z_2 \wr \F_2$, the classical {\em lamplighter group} over the free group on two generators. This is only for simplicity of the exposition and all arguments extend in straightforward manner to the general setting of Theorem \ref{theorem-Lamplighter}.

Let $\skrig$ denote the Cayley graph of $\Gamma$ with respect to the canonical generating set. We can picture this group as the $4$-regular tree with a lamp attached to each vertex. An element $(\varphi,x)\in\Gamma$ is then interpreted as lamps being switched on at the support of $\varphi$, denoted by $\supp(\varphi)$, and a marker standing at position $x$. If $a,b$ denote the canonical generators for $\F_2$, then the generators 
$$
\Bigl\{ (\bzero,a),(\bzero,a^{-1}),(\bzero,b),(\bzero,b^{-1}) \Bigr\},
$$
where $\bzero(x)=0$ for all $x$, correspond to moving the marker in the direction of the appearing generator while leaving the configuration of lamps untouched. The remaining generator $(\1_{e_{\F_2}},e_{\F_2})$ corresponds to switching the state of the lamp at the current position of the marker. For this Cayley graph it is well-known how to describe the graph distance: given two elements $(\varphi,x)$ and $(\psi,y)$ of $\Gamma$, a shortest path in $\skrig$ is such that the marker starts at position $x$, visits the position of all lamps which need to be switched and switches these configurations accordingly, and finally moves to position $y$. Note that the positions of the lamps which need to be switched are given by $\supp(\varphi \psi)$ and that $|\supp(\varphi \psi)|$ many switches have to be made. Therefore
$$
d( (\varphi,x),(\psi,y) ) = |\supp(\varphi \psi)| + \ell(x,\supp(\varphi\psi),y),
$$
where $\ell(x,\supp(\varphi\psi),y)$ denotes the length of a shortest path in the $4$-regular tree which starts at $x$, visits $\supp(\varphi\psi)$ completely and ends at $y$. In other words, the distance is essentially determined by the solution to the above {\em travelling salesman problem} in the Cayley graph of $\F_2$. When $\F_2$ is replaced with a different finitely generated group and the tree is replaced with the corresponding Cayley graph, this description remains correct. However, the particularly nice connection between {\em walks in the base graph} and {\em walls in the lamplighter graph} we describe next makes use of the tree-structure.

The following description of a proper, left-invariant structure of space with walls on $\Gamma$ is from~\cite{CSV08}. We adapt their notation to our more restrictive setting in order to readily obtain quantitative estimates of the associated wall distance. Let $T_4$ denote the canonical Cayley graph of $\F_2$, i.e.\ the $4$-regular tree, whose root is given by the vertex corresponding to $e_{\F_2}$. Note that every edge $e\in E(T_4)$ determines a wall by splitting $T_4$ into the two connected components of $T_4\setminus \{e\}$. For our purposes it is convenient to consider the set $E^{\to}(T_4)$ of oriented edges of $T_4$. For every oriented edge $e$, let $A_e$ denote the connected component of $T_4 \setminus \{e\}$ which the edge points away from. Given a finitely supported function $\psi \colon A_e^c \to \Z_2$, define
$$
E(e,\psi) \coloneqq \Bigl\{ (x,\varphi)\in\Gamma \colon x \in A_e, \varphi_{|A_e^c} = \psi \Bigr\}.
$$
A family of walls on $\Gamma$ is then defined by
$$
\skriw \coloneqq \Bigl\{ \{E(e,\psi),E(e,\psi)^c\} \colon e \in E^{\to}(T_4), \psi \colon A_e^c \to \Z_2 \, \, \mbox{finitely supported} \Bigr\}.
$$

It was proven in \cite[Theorem 1]{CSV08} that $(\Gamma,\skriw)$ is a space with walls on which $\Gamma$ acts properly via left translations. We are interested in a quantitative version, which is given by the following lemma and was also observed in a more general context in \cite[Proposition 7.2]{CSV12}.

\begin{lemma}[\bf Estimate of the Wall-Distance for Lamplighters over Free Groups]\label{lemma-WallDistanceLamplighterF2} 
Let $\Gamma\coloneqq\Z_2 \wr \F_2$. Consider the proper action of $\Gamma$ on the space with walls $(\Gamma,\skriw)$ with
$$
\skriw \coloneqq \Bigl\{ E(e,\psi) \colon e \in E^{\to}(T_4), \psi \colon A_e^c \to \Z_2 \, \, \mbox{finitely supported} \Bigr\}.
$$
Set $o\coloneqq(\bzero,e_{\F_2})$ and consider the canonical Cayley graph of $\Gamma$. Then, for every $(\varphi,x)\in\Gamma$,
$$
w(o,(\varphi,x))\geq\frac{1}{4} |(\varphi,x)|-\frac{1}{4}.
$$
\end{lemma}

\proof Let $(\varphi,x)\in\Gamma$ and recall from above that
$$
|(\varphi,x)| = d(o,(\varphi,x))=|\supp(\varphi)| + \ell(e_{\F_2},\supp(\varphi),x).
$$
To shorten notation, set $\ell=\ell(e_{\F_2},\supp(\varphi),x)$. Let $\rho=\{e_1,\ldots,e_{\ell}\}$ be a shortest path of edges from $e_{\F_2}$ to $x$ which visits $\supp(\varphi)$ completely. Note that every position in $\supp(\varphi)$ except the origin has to be visited by $\rho$, thus $|\supp(\varphi)| \leq 1 + \ell.$
In particular
\begin{equation} \label{eq-travellingsalesmanbound}
|(\varphi,x)| \leq 2\ell + 1.
\end{equation}
We claim that
\begin{equation} \label{eq-walkwallbound}
\ell \leq 2 w(o,(\varphi,x)).
\medskip
\end{equation}
To see this, we describe a subset $\rho'$ of $\rho$, consisting of at least half of the original edges, such that for each in $\rho'$ there is a unique wall separating $o$ and $(\varphi,x)$: let $e_i$ be an edge in $\rho$ which appears for the first time. It is clear from the assumptions on $\rho$ that each edge can appear at most twice and we treat the resulting two cases separately.
\begin{enumerate}
\item Assume that $e_i$ appears a second time, say at $e_j$ with $i\leq j\leq\ell$. Since the walk takes place on a tree, it follows that every edge $e_k$ with $i\leq k\leq j$ appears exactly twice. Treating each $e_k$ as an oriented edge pointing in the direction of the next appearing edge in $\rho$ or towards $x$ in case $i=\ell-1$, we see that this case describes the following situation: the path $\rho$ starts at endpoint $y$ of $e_i$, proceeds in the direction of $e_i$ and explores a part of the corresponding branch of the tree, before returning to $y$ and leaving in the direction of one of the remaining branches originating from $y$. In particular the edge $e_i$ points away from the final position $x$ of the marker. The same holds for every $e_k$ with $i\leq k\leq j$ which appears for the first time. Now fix any such edge $e_k$. Clearly the only reason for its appearance can be that there is an element $z\in\supp(\varphi)$ which lies in the direction it is pointing. Recalling that $A_k\coloneqq A_{e_k}$ denotes the connected component of $T_4\setminus\{e_k\}$ which $e_k$ points away from, we therefore have that $x\in A_{k}$ and that $z\in\supp(\varphi_{|A_k^c})$. It follows that the set
$$
E(e_k,\varphi_{|A_k^c}) = \Bigl\{ (v,\psi)\in\Gamma \colon v \in A_{k}, \psi_{|A_k^c} = \varphi_{|A_k^c} \Bigr\}
$$
contains $(x,\varphi)$, but not $o=(\bzero,e_{\F_2})$ because $\varphi_{|A_k^c}(z)=1\neq\bzero(z)$. We have thus exhibited for half of the edges $e_i,\ldots,e_j$ a unique wall separating $o$ and $(\varphi,x)$.
\item Assume that $e_i$ appears exactly once. We again treat $e_i$ as an oriented edge pointing towards the next appearing edge in $\rho$ or towards $x$ in case $i=\ell-1$. Since the path started at the identity and since $e_i$ appears for the first time, $e_i$ points away from $e_{\F_2}$, i.e.\ $e_{\F_2}\in A_i\coloneqq A_{e_i}$. Since $e_i$ is not traversed twice, the final position $x$ of the marker lies in $A_{e_i}^c$. In particular
$$
E(e_i,\bzero_{|A_i^c}) = \Bigl\{ (v,\psi)\in\Gamma \colon v \in A_{i}, \psi_{|A_i^c} = \bzero_{|A_k^c} \Bigr\}
$$
contains $o=(\bzero,e_{\F_2})$, but not $(x,\varphi)$ because $x\notin A_i$. We have thus exhibited for this edge a unique wall separating $o$ and $(\varphi,x)$.
\end{enumerate}
Applying the above procedure iteratively proves Inequality (\ref{eq-walkwallbound}). Combined with Inequality (\ref{eq-travellingsalesmanbound}), this concludes the proof of Lemma \ref{lemma-WallDistanceLamplighterF2}.
\eproof

\noindent{\bf Proof of Theorem \ref{theorem-Lamplighter}.} Combining Lemma \ref{lemma-WallDistanceLamplighterF2} with Corollary \ref{cor-QuantitativeDecay} yields the desired result. \qed

\subsection{Isometry Groups of Hyperbolic Space: Proof of Theorem \ref{theorem5}.} \label{hyperbolic}

In this section we discuss a relevant class of examples in which the percolations constructed using Theorem \ref{theorem-GeneralConstruction} can be understood in a precise geometric sense. These are finitely generated groups with a proper isometric action on real hyperbolic space.

\subsubsection{\bf Hyperbolic Space as a Space with Measured Walls.} \label{subsection-HyperbolicWalls}

The following observations are due to \cite[Section 3]{CMV04} building on work of Robertson \cite{R98}, see also \cite[Example 3.7]{CDH10} for the notation we adopt here. Let $\H^n$ be real $n$-dimensional hyperbolic space and let $\Isom(\H^n)$ be its isometry group. Let $\skriw_{\H^n}$ denote the set of all closed or open geometric half-spaces in $\H^n$ with boundary an isometric copy of $\H^{n-1}$. Each half-space defines a wall by partitioning the space into the half-space itself and its complement. The group $\Isom(\H^n)$ is unimodular, acts transitively on $\skriw_{\H^n}$ and the stabilizer of a given wall is unimodular. Thus $\skriw_{\H^n}$ carries a non-zero $\Isom(\H^n)$-invariant Borel measure $\mu_{\H^n}$. Note that the set $\skriw_{\H^n}(z,w)$ of walls separating two points $z,w\in\H^n$ is relatively compact, which implies $\mu_{\H^n}(\skriw_{\H^n}(z,w))<\infty$. Hence the $4$-tuple $(\H^n,\skriw_{\H^n},\skrib(\skriw_{\H^n}),\mu_{\H^n})$ defines a space with measured walls. By Crofton's formula \cite[Proposition 2.1]{R98}, there exists a constant $\lambda>0$ such that, for every $z,w\in\H^n$,
\begin{equation} \label{eq-Crofton}
\mu_{\H^n} (\skriw_{\H^n}(z,w)) = \lambda d(z,w),
\end{equation}
where $d(z,w)$ denotes the hyperbolic distance.

\subsubsection{\bf Hyperplane Percolation.} \label{subsection-HyperplanePercolation}

We now suppose that $\skrig$ is either a graph embedded in the hyperbolic plane $\H^2$ such that the automorphisms of $\skrig$ extend to isometries of $\H^2$ or that $\skrig$ is the Cayley graph of a discrete co-compact subgroup $\Gamma$ of isometries of $\mathbb H^n$. Let $\eta$ be a Poisson process on $\skriw_{\H^n}$ with intensity measure $\gamma \mu_{\H^n}$ for some $\gamma>0$. Define a bond percolation on $\skrig$ as in Theorem \ref{theorem-GeneralConstruction}, i.e.\ for every half-space contained in $\eta$, delete all edges whose endpoints are separated by the boundary of this halfspace. We refer to this bond percolation as the {\em hyperplane percolation} on $\skrig$ induced by $\eta$. Note that $\Isom(\H^n)$-invariance of $\eta$ implies $\Aut(\skrig)$-invariance of the corresponding hyperplane percolation in the first case and $\Gamma$-invariance in the second case.

\begin{theorem}[\bf Large Marginals and Exponential Decay] \label{theorem-InfManyInfClusters}
Let $\skrig$ be a transitive, nonamenable, planar graph with one end, resp.~let $\skrig$ be the Cayley graph of a discrete co-compact subgroup $\Gamma$ of ${\rm Isom}(\mathbb H^n)$. Then for every $\alpha<1$, there exists an $\Aut(\skrig)$-invariant, resp.~$\Gamma$-invariant, bond percolation $\P$ on $\skrig$ with $\E[\deg_\omega(o)] > \alpha \deg(o)$ and such that $\tau(o,\cdot)$ has exponential decay.
\end{theorem}

In this case the construction is ergodic (see e.g.~\cite[Remark 2.8]{R98}), hence using Proposition~\ref{prop:ConsequenceVanishingTPF} we may add to the statement of Theorem~\ref{theorem-InfManyInfClusters} the condition that $\mathbb P$ has infinitely many infinite clusters a.s. In the planar case, it is also possible to prove this using \cite[Theorem 3.1]{BS01} and either the hyperbolic geometry or the Harris-FKG inequality for the Poisson process.

\begin{remark}[\bf Co-Compact Fuchsian Groups]
Theorem~\ref{theorem-InfManyInfClusters} applies to planar Cayley graphs of nonamenable one-ended groups. In particular it applies to the examples considered by Lalley \cite{L98,L01}, namely planar Cayley graphs of co-compact Fuchsian groups. For these groups, it was shown in \cite{BS01,L98} that Bernoulli bond percolation exhibits a double phase transition. In this sense, these groups behave quite differently from infinitely ended ones, which all have uniqueness threshold $p_u=1$. With respect to general invariant bond percolations, this difference between one-ended and infinitely-ended nonamenable planar Cayley graphs disappears in the sense that there exist invariant percolations with marginals arbitrarily close to $1$ but with infinitely many infinite clusters a.s. We may also point out that for co-compact Fuchsian groups more refined estimates of the exponential decay in Theorem~\ref{theorem-InfManyInfClusters} are available. More precisely, consider a co-compact Fuchsian group $\Gamma$ and recall that there exists a compact fundamental polygon $T\subset\H_2$ for $\Gamma$, i.e.\ a compact, connected subset, bounded by finitely many geodesics and such that the translates $\{gT\}_{g\in\Gamma}$ tesselate $\H_2$. In particular there are finitely many group elements for which $gT$ and $T$ intersect. These are called side-pairing transformations and constitute a generating set for $\Gamma$. Let $\skrig$ be the corresponding Cayley graph, which can be identified with the dual graph of the tiling $\{gT\}$. Since the tiling is locally finite, we have that for every radius $r>0$, there exists a constant $0<D(r)<\infty$ such that no hyberbolic ball of radius $r$ intersects more than $D(r)$ tiles. Now pick $z_0\in T\setminus\partial T$ and set $\rho\coloneqq \inf_{g,h\in\Gamma} d(gz_0,hz_0)$. It follows by elementary arguments that 
$$
\frac{\rho}{2D(\rho)}|g| \leq d(z_0,gz_0) \leq \diam(T) |g| \quad \mbox{for all} \, \, g \in \Gamma\setminus\{o\},
$$
see e.g.\ \cite[Proposition 1.10]{L01}. 
\nopagebreak {\hfill\rule{2mm}{2mm}}
\end{remark}

\proof[{\bf Proof of Theorem \ref{theorem-InfManyInfClusters}}] Consider first the case that $\mathcal G$ is a transitive, nonamenable, planar graph with one end. By \cite[Proposition 2.1]{BS01}, the assumptions guarantee that $\Aut(\skrig)$ is discrete, hence unimodular (since for a discrete group, the counting measure is Haar measure and is both left- and right-invariant), and that $\skrig$ can be embedded in the hyperbolic plane $\H^2$ in such a way that the action of $\Aut(\skrig)$ extends to an isometric action on $\H^2$. Moreover, the embedding may be taken such that the dual graph is quasi-transitive, hence the embedding is a quasi-isometry. Identify $\skrig$ with its embedding. Let $p\in(0,1)$, let $\eta_p$ be a Poisson process on $\skriw_{\H^2}$ with intensity measure $(1-p)\mu_{\H^2}$ and let $\xi_p$ denote the induced hyperplane process on $\skrig$. Let $\P_p$ denote the law of $\xi_p$, which defines an $\Aut(\skrig)$-invariant bond percolation. Since the embedding is a quasi-isometry, it follows from Crofton's formula, see Equation (\ref{eq-Crofton}), that $\tau_{\xi_p}(o,\cdot)$ has exponential decay and that
$$
\lim_{p\to1} \E[\deg_{\xi_p} (o)]= \deg(o),
$$
which finishes the proof. The case that $\mathcal G$ is the Cayley graph of a discrete co-compact subgroup $\Gamma$ of ${\rm Isom}(\mathbb H^n)$ follows similarly using the Milnor-Schwarz Lemma, see e.g.\ \cite[Theorem~1.18]{R03}.
\eproof

\subsection{Spaces with Measured Walls for Amenable Groups: Proof of Proposition \ref{prop1}.} \label{amenable}

Recall that amenable groups have the Haagerup property \cite{BCV95}. In this section, we show that under the additional assumption that $\Gamma$ is amenable, there exists a natural space with measured walls which admits a proper, invariant $\Gamma$-action, such that our constructive proof of Theorem \ref{theorem-GroupsWithTheHaagerupProperty} recovers the corresponding statement of Theorem \ref{theorem-AmenableGroups} from \cite{BLPS99} -- let us point out that Theorem \ref{theorem-AmenableGroups} may also be proved by the Rokhlin lemma of Ornstein and Weiss \cite{OW87} and results of Adams and Lyons \cite{AL91} and that \cite[Theorem 5.1]{BLPS99} provides a more general statement.

\begin{prop}[Recovering the Amenable Case] \label{prop-SpecialCaseAmenability} 
Let $\skrig$ be the Cayley graph of a finitely generated amenable group $\Gamma$. Consider the proper, invariant action of $\Gamma$ on the space with measured walls $(X,\skriw,\skrib,\mu)$ described in \cite[Proof of Theorem I(5)]{CMV04}. For $p\in(0,1)$, let $\eta_p$ be a Poisson process on $\skriw$ with intensity $(1-p)\mu$ and let $\omega_p$ denote the associated invariant bond percolation on $\skrig$ as in the proof of Theorem \ref{theorem-GroupsWithTheHaagerupProperty}. Then $\omega_p$ has no infinite cluster.
\end{prop}

\proof We first need to introduce the space with measured walls from \cite{CMV04}. To that end, we recall a notion of product for spaces with measured walls from \cite[Definition 5]{CMV04}. Let $(X_n,\skriw_n,\skrib_n,\mu_n)_{n=1}^\infty$ be a sequence of spaces with measured walls. Distinguish a point $x_n^0 \in X_n$ for every $n\in\N$. Now define the set
$$
X \coloneqq \Biggl\{ x = (x_n)_{n=1}^\infty \in \prod_{n=1}^\infty X_n \colon \sum_{n=1}^\infty w_n(x_0^{(n)},x_n) < \infty \Biggr\}.
$$
For each $n\in\N$, let $p_n\colon X \to X_n$ denote the projection onto the $n$'th coordinate. Note that if $n\in\N$ and $\{A,A^c\}\in\skriw_n$, then
$$
W(\{A,A^c\}) \coloneqq p_n^{-1}\bigl( \{A,A^c\} \bigr) = \Bigl\{ \bigl\{ x \in X \colon x_n \in A \bigr\}, \bigl\{ x \in X \colon x_n \in A^c \bigr\} \Bigr\}
$$
defines a wall on $X$. Let $\skriw$ denote the (disjoint) union of such walls, which can be identified with the coproduct of $\skriw_n$. Let $\skrib$ be the $\sigma$-algebra generated by $p_n^{-1}(\skrib_n)$ and let $\mu$ be the unique measure on $\skriw$ whose restriction to $p_n^{-1}(\skrib_n)$ coincides with $\mu_n$. Then $(X,\skriw,\skrib,\mu)$ is a space with measured walls which is called the {\em product} of the {\em pointed} spaces $(X_n,\skriw_n,\skrib_n,\mu_n,x_n^0)$. 

Now let $(K_n)_{n=1}^\infty$ be a sequence of non-empty finite subsets of $\Gamma$ increasing to $\Gamma$. Assuming amenability, there exists a {\em F{\o}lner sequence} for $(K_n)$, i.e.\ a sequence $(A_n)_{n=1}^\infty$ of non-empty finite subsets of $\Gamma$ such that, for every $n\in\N$ and $g\in K_n$,
$$
\frac{|gA_n \Delta A_n|}{|A_n|} < 2^{-n}.
$$
For each $n\in\N$, form the space with measured walls $(\Gamma,\skriw_n,\skrib_n,\mu_n)$, where the set of walls is
$$
\skriw_n = \Bigl\{ \bigl\{gA_n^{-1},(gA_n^{-1})^c \bigr\} \colon g \in \Gamma \Bigr\},
$$
the $\sigma$-field $\skrib_n$ is the power set and $\mu$ is counting measure normalized by $n/|A_n|$, see \cite[Example 1]{CMV04}. Observe that the corresponding wall distance is given by 
$$
w_n(g,h) = \frac{n}{|A_n|} |gA_n \Delta h A_n|,
$$
which is clearly bounded by $n$. In order to obtain an action on a space with unbounded wall distance we therefore need to {\em pack up} (a formulation due to \cite{CMV04}) the sequence of spaces described above. To that end, distinguish the point $x_0^n=o$ for each $n\in\N$. Let $(X,\skriw,\skrib,\mu)$ be the product of the pointed spaces with measured walls $(\Gamma,\skriw_n,\skrib_n,\mu_n,o)$. It is then not difficult to check that $\Gamma$ acts on $X$ by the diagonal action, see \cite[Proof of Theorem I(5)]{CMV04}. Distinguish the point $x_0=(o,o,\ldots)\in X$. 

With these technical preliminaries established, we are now ready to conclude the proof. Since $\omega_p$ is a $\Gamma$-invariant bond percolation, it suffices to prove that the cluster of the identity is finite a.s. Consider the set of walls
$$
\skriv \coloneqq \bigsqcup_{n=1}^\infty \Bigl\{ p_n^{-1}\bigl(\bigl\{gA_n^{-1},(gA_n^{-1})^c \bigr\}\bigr) \colon g\in\Gamma \, \, \mbox{such that} \, \, o \in gA_n^{-1} \ \Bigr\}.
$$
In the particular setting described here, we claim that $\skriv$ may be interpreted as a particular set of all walls which separate $o$ from $\infty$ in $\skrig$: let $n\in\N$ and $g\in\Gamma$ such that $o \in gA_n^{-1}$. If the wall 
$$
W=p_n^{-1}\bigl(\bigl\{gA_n^{-1},(gA_n^{-1})^c \bigr\}\bigr)
$$
satisfies $W\in\supp(\eta_p)$, then, in the configuration $\omega_p$, the edge boundary of a finite neighborhood of $o$ in $\skrig$ is completely removed. Therefore on the event that $\eta_p(\skriv)>0$, the cluster of $o$ in $\omega_p$ is finite. But now, since $\eta_p$ is a Poisson process with intensity $(1-p)\mu$, the fact that
\begin{align*}
\mu(\skriv) & = \sum_{n=1}^\infty \mu_n \Bigl( \Bigl\{ \bigl\{gA_n^{-1},(gA_n^{-1})^c \bigr\} \colon g\in\Gamma \, \, \mbox{such that} \, \, o \in gA_n^{-1} \Bigr\} \Bigr) = \sum_{n=1}^\infty \frac{n}{|A_n|} |A_n| = \infty,
\end{align*}
already implies that $\P(\eta_p(\skriv)>0)=1$. Thus $\omega_p$ has no infinite cluster a.s. \eproof

\subsection{Outlook and Open Questions.} \label{section-outlook}

We provide an outlook on future directions motivated by our results and highlight some important open questions related to the main theme of this article.

Towards further progress on the general program of understanding the interplay between geometric group theory and probability using percolation, it remains an interesting open problem to find characterizations of various other classes of groups. Let us highlight the following particularly important case.

\begin{question}[{Russell Lyons, cf.\cite[Sec.~10, p.~1123]{L00}}]
    Does there exist a characterization of hyperbolic groups through percolation?
\end{question}

Note that there are hyperbolic groups with the Haagerup property as well as hyperbolic groups with property (T). In a similar direction and very concretely related to the present methods, the following remains open (recall Remark \ref{remark-L1Compression} for the context about $L^1$-compression).

\begin{question}
    Does the equivariant $L^1$-compression exponent admit a characterization through invariant percolation? 
\end{question}

We believe that a natural candidate for characterizing the $L^1$-compression exponent through percolation is provided by condition (ii) in Corollary \ref{cor-QuantitativeDecay} with~$\rho=\rho_\alpha$. Notably, this would provide the first probabilistic characterization of this exponent. The following intermediate result is proved in \cite{MR24}: recall that $\psi\colon\Gamma\to[0,\infty)$ is {\em measure definite function} if there exists a measure space $(\Omega,\skrib,\mu)$ and $S\colon\Gamma\to\skrib$ such that $\psi(g^{-1}h)=\mu(S_g \Delta S_h)$.

\begin{theorem}[\bf Quantitative Two-Point Function Decay and Compression Functions] \label{theorem-ExponentialDecay}
Let $\skrig$ be the Cayley graph of a finitely generated group $\Gamma$ and let $\rho\colon[0,\infty]\to[0,\infty]$ be a function with $\{\rho=0\}=\{0\}$. Then each of the following conditions implies the next one:
\begin{enumerate}
\item[{\rm(i)}] There exists an invariant action of $\Gamma$ on a space with measured walls with $\rho$-growth;
\item[{\rm(ii)}] There exists $C>0$ such that for every $\alpha<1$, there exists a $\Gamma$-invariant bond percolation $\P$ on $\skrig$ with $\E[\deg_\omega(o)]>\alpha\deg(o)$ and with $(C,\rho)$-exponential decay;
\item[{\rm(iii)}] There exists a measure definite function on $\Gamma$ which admits $\rho$ as a compression function.
\end{enumerate}
Moreover, if $\Gamma$ is amenable or if the measure definite function in (iii) can be assumed to be the square-root of a conditionally negative definite function, then all three conditions are equivalent.
\end{theorem}

\begin{remark}[Connection with the equivariant $L^1$-compression exponent, continued] 
Condition (iii) in Theorem \ref{theorem-ExponentialDecay} may be interpreted as "almost'' condition (i), as we will now explain: observe that every measure definite function $\psi$ on $\Gamma$ is of the form
\begin{equation} \label{eq-L1Rem-MD}
\psi(g^{-1}h) = \Vert f(g)-f(h) \Vert_E
\vspace{1mm}
\end{equation}
for an $L^1$-space $E$ and a map $f\colon\Gamma \to E$. While this is a formally weaker requirement than being of the form  \eqref{eq-L1Rem-Equ}, it is an open question whether the two coincide, see~\cite{CSV12}. In particular, proving that these notions coincide provides one possible way of obtaining the desired probabilistic characterization of the equivariant $L^1$-compression exponent (in this vein, note that Theorem \ref{theorem-ExponentialDecay} does entail the desired characterization  for amenable groups).  On a related note, since the square-root of any conditionally negative definite kernel is of the form \eqref{eq-L1Rem-Equ}, as described in Section \ref{subsection-MeasuredWalls}, condition~(iii) with $\rho=\rho_\alpha$ always implies condition (i) with $\rho=\rho_{\alpha/2}$.  \nopagebreak {\hfill\rule{2mm}{2mm}}  
\end{remark}

\begin{remark}
    Let us comment on the proof of Theorem \ref{theorem-ExponentialDecay}. Suppose that the implication "(ii)$\Rightarrow$(iii)'' holds. Since the square-root of a conditionally negative definite kernel is the wall distance associated to an invariant action of $\Gamma$ on a space with measured walls, condition (iii) implies condition (i) under the first additional assumption. Under the other additional assumption of amenability, the fact that every measure definite function is the wall distance associated to an {\em invariant} action on a space with measured walls is due to \cite[Proposition 2.8]{CSV12}. Hence, only "(ii)$\Rightarrow$(iii)'' requires justification. A~preliminary version of this article featured a proof, however, this proof does not use the underlying group structure and, after completion, the authors realized that it is in fact an instance of a more general phenomenon. In the more general form it can be applied to provide a characterization of the  (non-equivariant) $L^1$-compression exponent using the strategy presented above and building on the ideas developed in this article, see \cite{MR24}. Since the implication "(ii)$\Rightarrow$(iii)'' in Theorem \ref{theorem-ExponentialDecay} is not used in this article, the proof may be found in the follow-up work \cite[Theorem 3.3]{MR24}.
\end{remark}

We do not know whether the conditions in Theorem \ref{theorem-ExponentialDecay} are in fact equivalent. There are prominent examples of non-amenable groups which satisfy all three conditions: to name two, these include $\Z_2 \wr \F_2$ by Theorem \ref{theorem-Lamplighter} and co-compact Fuchsian groups by Theorem \ref{theorem-InfManyInfClusters}.

The geometry of $L^1$-spaces is particularly closely related to percolation: while the current method can be used to extract information about other $L^p$-target geometries with $p>1$, see~\cite[Remarks 4.4 \& 4.11]{MR24} for more details in this direction, the situation is certainly less clear and seems to be completely open for other target geometries.

\begin{problem}
    Further investigate the interplay between invariant percolation and the property of admitting equivariant embeddings with ''good'' compression into other target geometries, for instance $L^p$ with $p>1$.
\end{problem}

In another direction, it would be interesting to characterize the Haagerup property, resp.~property~(T), through quantities associated to Bernoulli percolation such as $\pv$, perhaps by additionally allowing to replace $\skrig$ by one of its invariant spanning subgraphs.

\begin{question} \label{qu:HaagerupBernoulli}
    Are the conditions in Theorem \ref{theorem-Spanning} equivalent to the Haagerup property? Similarly, does there exist a characterization of property~(T) in terms of natural properties of Bernoulli percolation on invariant spanning subgraphs?
\end{question}

\noindent{\bf Acknowledgement:}
We thank Russell Lyons for reading the preprint version of the article carefully and for providing corrections and very valuable comments, leading in particular to a clarification of Example \ref{example-WordLengthFree}. We also thank Damien Gaboriau for very helpful comments, encouragement and feedback. Finally, we thank an anonymous referee for their careful reading, very helpful comments, for pointing out Proposition~\ref{prop:Robust}  and for pointing out Proposition~\ref{prop:ConsequenceVanishingTPF} and outlining its proof. 
Research of both authors is funded by the Deutsche Forschungsgemeinschaft (DFG) under Germany's Excellence Strategy EXC 2044-390685587, Mathematics M\"unster: Dynamics-Geometry-Structure.

\end{document}